\documentclass[10pt]{article}
\usepackage{color}
\usepackage{graphicx}
\usepackage{amsmath}
\usepackage{amssymb}
\usepackage{mathrsfs}
\usepackage[numbers,sort&compress]{natbib}
\usepackage{amsthm}
\numberwithin{equation}{section}
\usepackage{pgf}
\usepackage{CJK}
\usepackage{graphicx}
\usepackage{tikz}
\usepackage{stmaryrd}
\usepackage{graphicx}
\usepackage[unicode]{hyperref}
\usepackage[latin1]{inputenc}
\usepackage[T1]{fontenc}
\usepackage[english]{babel}
\usepackage{listings}
\usepackage{xcolor,mathrsfs,url}
\usepackage{amssymb}
\usepackage{amsmath}
\usepackage{ifthen}
\newtheorem{theorem}{Theorem}
\newtheorem{lemma}{Lemma}

\newtheorem{corollary}{Corollary}
\newtheorem{Proposition}{Proposition}
\newtheorem{RHP}{RHP}

\usepackage {caption}
\usepackage{float}
\usepackage{subfigure}
\hypersetup{hypertex=true,
	colorlinks=true,
	linkcolor=blue,
	anchorcolor=blue,
	citecolor=red}

\DeclareMathOperator*{\res}{Res}
\begin{document}
	\title{\Large The   Cauchy problem of the Camassa-Holm equation in a  weighted Sobolev space: Long-time and Painlev\'e  asymptotics}
	\author{\Large Kai Xu$^1$, \  Yiling Yang$^1$ and  \  Engui Fan$^{1}$}
	\footnotetext[1]{ \  School of Mathematical Sciences  and Key Laboratory for Nonlinear Science, Fudan University, Shanghai 200433, P.R. China.}
	\footnotetext[2]{ \   Email address:\ Kai XU: \ 22110180047@fudan.edu.cn; Yiling YANG: 19110180006@fudan.edu.cn; Engui FAN:  faneg@fudan.edu.cn}
	\date{ }
	\maketitle
	
	\begin{abstract}
		\baselineskip=17pt
		
	Based on  the $\overline\partial$-generalization of the Deift-Zhou  steepest descent method,  we extend  the long-time and Painlev\'e  asymptotics  for     the  Camassa-Holm  (CH)  equation to
the solutions with initial data in   a  weighted Sobolev space $  H^{4,2}(\mathbb{R})$.
 With   a new scale $(y,t)$  and  a  RH problem associated
with the initial value problem,
		we derive   different long time asymptotic expansions for  the solutions of  the CH equation    in  different space-time solitonic regions.
 The  half-plane  $\{ (y,t): -\infty <y<\infty, \ t> 0\}$  is  divided   into four asymptotic regions:
		1. Fast decay region, $ y/t \in(-\infty,-1/4)$   with   an error   $\mathcal{O}(t^{-1/2})$;  2. Modulation-solitons region, $y/t \in(2,+\infty)$,   the  result can be
  characterized with  an modulation-solitons with   residual error    $\mathcal{O}(t^{-1/2 })$; 3. Zakhrov-Manakov  region,
		$y/t \in(0,2)$ and $y/t \in(-1/4,0)$.
		 The   asymptotic  approximations  is  characterized  by  the dispersion term  with residual error
		$\mathcal{O}(t^{-3/4})$;  4.  Two transition regions,  $|y/t|\approx 2$ and $|y/t| \approx -1/4$, the results   are  describe by the solution of Painlev\'e II equation with error order  $\mathcal{O}(t^{-1/2})$.\\[6pt]
		\noindent {\bf Keywords:}   Camassa-Holm   equation,  Riemann-Hilbert problem,    $\overline\partial$-steepest descent method, long time asymptotics, Painlev\'e II equation.\\[6pt]
		\noindent {\bf MSC 2020:} 35Q51; 35Q15; 37K15; 35C20.
		
	\end{abstract}
	
	\baselineskip=17pt
	
	\newpage
	
	\tableofcontents
	
	\newpage
	
	\section {Introduction}
	\quad
	In this paper, we   study the long time asymptotic behavior for  the solution  to  the Cauchy problem   of the  Camassa-Holm (CH) equation in a weighted Sobolev
space
	\begin{align}
		&u_{t}-u_{txx}+2\omega u_{x}+3uu_{x}=2u_{x}u_{xx}+uu_{xxx},  \label{ch} \\
		&u(x,t= 0)=u_{0}(x),\label{initial}
	\end{align}	
	where $u_{0} \in H^{4,2}(\mathbb{R})$ and  $\omega$ is a positive constant.
	The CH equation (\ref{ch}) was derived as a model for unidirectional propagation of small amplitude shallow water waves  by Camassa and Holm    \cite{Holm1},
	but already appeared earlier in a list by Fuchssteiner
	and Fokas \cite{Fuchssteiner1}.  It also
	arises in the study of the propagation of axially symmetric waves in hyperelastic
	rods \cite{Constantin4,Dai}  and its high-frequency limit models nematic liquid crystals \cite{Hunter,Constantin5}.
	It  has   attracted considerable interest and been studied extensively  due to their   rich mathematical structure and
	remarkable properties. For example,
	the CH equation (\ref{ch})  is an
	infinite-dimensional Hamiltonian system that is completely integrable \cite{Constantin0,Constantin1}.
	The CH equation admits non-smooth peakon-type  solitary or periodic  traveling wave solutions and it
	was shown that these solutions are orbitally stable \cite{Constantin3,Lenells1}.  It has been shown that the solitary waves of CH equation
	are smooth if  $\omega > 0 $   \cite{Johnson,Parker1,Parker2,Parker3}  and peaked   weak solutions if $\omega = 0 $ \cite{Holm1,Beals1,Beals2,Constantin6} .
	The  stability of   peakons  and orbital  stability  of solitary wave solution  for the  CH equation were further shown by Constantin \cite{Constantin2,Constantin9}.
	The global well-posedness for  the  Cauchy problem  of the  CH equation (\ref{ch}) was  studied  \cite{Xin1,Xin2,Xin3}.
	The  algebro-geometric quasiperiodic solutions were constructed by using algebro-geometric method  \cite{Fritz1,Qiao3}.
	With the aid of reciprocal transformation, Darboux  transformation  and
	multi-soliton solutions  for the CH equation were given \cite{Li}.   The inverse spectral transform for the conservative CH
	equation was studied \cite{Eckhardt0, Eckhardt1, Eckhardt2}.

    The CH equation admits a Lax pair which
	allows using the IST method to study its  initial value problem  with $\omega > 0 $ \cite{Constantin1,Constantin7,Kaup,Lenells2}.
	Boutet  de Monvel et al. first investigated the  related Riemann-Hilbert   problem  (RH problem or RHP)
 for the CH equation  \cite{Monvel1, Monvel2}.
   They  further   study  long-time asymptotics  and  Painlev\'e-type asymptotics for CH equation on the line and  half-line
	\cite{CH,CH1,Monvel22} by using the  nonlinear steepest descent method.
 Minakov used this method  to  obtain  long-time  for    the CH  equation with step-like initial data  \cite{Minakov}.
The nonlinear steepest descent method developed by  Deift and Zhou \cite{RN6} is an   effective tool to  rigorously obtain the long-time asymptotics behavior of
other integrable 	systems \cite{RN9,RN10,Grunert2009,MonvelCH,xu2015,xusp,Geng3,XF2020,Xu2013,HG2009}.

The aim of  our  paper is to extend the asymptotic results in  \cite{CH,CH1,Monvel22}     to solutions
with initial data in   lower regularity spaces.  The long time and Painlesv\'e  asymptotic behavior   for the solutions of the Cauchy    problem of  the
  (\ref{ch})-(\ref{initial}) is established in a weighted Sobolev  space $H^{4,2}(\mathbb{R})$.
	Our   results are different from those in \cite{CH,CH1,Monvel22} in three  aspects.
Firstly,  in order   to extend    the Schwartz initial data to    lower regularity  weighted  Sobolev  initial data $u_0  \in H^{4,2}(\mathbb{R})$,
we apply the $\overline\partial$-generalization of the steepest descent method  proposed by
	McLaughlin and   Miller  \cite{MandM2006,MandM2008}.  In recent years, this method
	 has been   successfully  used    to investigate  longtime asymptotics,  soliton
	resolution and  asymptotic stability of N-soliton solutions to
	integrable systems in a weighted Sobolev space
	\cite{DandMNLS,fNLS,Liu3,SandRNLS,YF1,YF3,YF4}.
In addition,  the initial data   considered   in our paper  allows    presence of    discrete spectrum.
  Secondly,  although a   row vector RH problem  was    already constructed in   \cite{CH},
which however is not suitable  for  asymptotic analysis  by applying  $\overline\partial$-steepest descent method.  In our paper we re-construct a new matrix RH problem associated with the Cauchy problem (\ref{ch})-(\ref{initial}).
 Thirdly, according to the number of phase points on the jump contour,   we  present a  complete classification of asymptotic regions   by   dividing
     the whole  half-plane   $(y,t)$   into  four   asymptotic regions, in which
    different leading  order  asymptotic approximations   for the  solutions of  the Cauchy problem (\ref{ch})-(\ref{initial}) are   obtained respectively
	(see Figure \ref{result1}).
	

	\begin{figure}
		\begin{center}
			\begin{tikzpicture}
				\draw[yellow!20, fill=yellow!20](0,0)--(4,0)--(4,2)--(0, 2);
                \draw[green!20, fill=green!20](0,0 )--(-1,3)--(-4,3)--(-4,0);

				\draw[orange!20, fill=orange!20](0,0 )--(0,3)--(4,3)--(4,2);
				\draw[orange!20, fill=orange!20](0,0)--(-1,3)--(0,3)--(0,0);
				\draw [blue!20, fill=blue!20] (0,0)--(4,1.8)--(4,2.2)--(0,0);
				\draw[blue!20, fill=blue!20] (0,0)--(-0.8,3)--(-1.2,3)--(0,0);
				\draw [ -> ] (-4.6,0)--(4.6,0);
				\draw [ -> ](0,0)--(0,3.9);
				\draw [red,thick  ](0,0 )--(4,2);
				\draw [red,thick  ](0,0 )--(-1,3);
				\node    at (0,-0.3)  {$0$};
				\node    at (5,0)  {y};
				\node    at (0,4.2)  {t};
				\node  [below]  at (1.5,2.3) {\scriptsize III. Zakhrov-Manakov};
				\node  [below]  at (2.9,0.7) {\scriptsize I. Modulation-Solitons};
				\node  [below]  at (-2.1,1.6) {\scriptsize IV. Fast-decay};
				\node  [below]  at (-1,3.8) {\scriptsize $ \xi=-1/4 $};
				\node  [below]  at (4.7,2.5) {\scriptsize $  \xi=2 $};
				\node  [below]  at (5,2.1) {\scriptsize II. Painlev\'e};
				\node  [below]  at (-1,3.4) {\scriptsize II. Painlev\'e};
			\end{tikzpicture}
		\end{center}
		\caption{\footnotesize  Asymptotic approximations    of the  CH equation  in  different space-time  solitonic regions,
			where without stationary  phase points  in yellow  region;  four and  two  stationary   phase points in  green and  orange   regions   respectively;  Painlev\'e  asymptotics  in two  blue  regions
  }
		\label{result1}
	\end{figure}
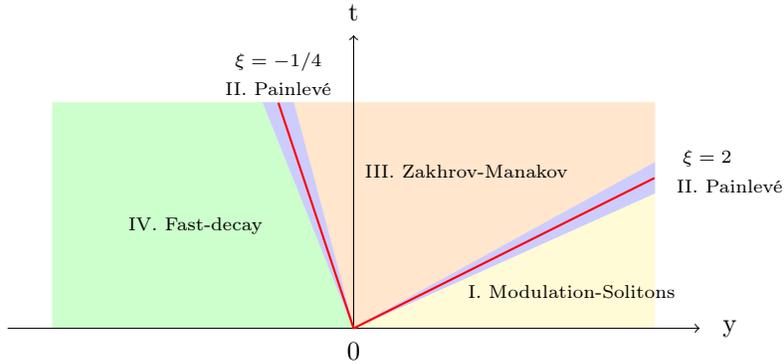

	Our paper is arranged as follows.  In section 2,  under weighted Sobolev initial data $u_0\in  H^{4,2}(\mathbb{R})$,
 we carry out the direct  scattering transform to establish
RH problem associated with the Cauchy problem (\ref{ch})-(\ref{initial}),  which will be used
	to analyze   long-time asymptotics  of the CH equation.
	In section 3,   we  present long-time asymptotics for the CH equation
	in the region I and IV in which the jump contours admit   no phase points.
	In this case, the main contribution to the asymptotic expansion comes from discrete spectrum and a $\bar\partial$-equation.
	In section 4,   we  present long-time asymptotics for the CH equation
	in the regions III.   The main contribution to leading term comes from the  jump contours near   two and four  phase points  respectively.
In section 5, we  present  the Painlev\'e  asymptotics for the CH equation  in  two   transition regions II. \vspace{6mm}

	\section {Inverse scattering transform  and the RH problem}\label{sec2}

In this section, we carry out the direct   scattering transform to establish
RH problem associated with the Cauchy problem (\ref{ch})-(\ref{initial}) with   weighted Sobolev initial data $u_0\in  H^{4,2}(\mathbb{R})$.

\subsection{Spectral analysis }

It is known that if $m(x,0)+\omega>0$ for all $x$,  then $m(x,t)$ exists for all t>0; moreover, $m(x,t)+\omega>0$,
	which justifies the equivalent form of the CH equation \cite{Monvel1}
	\begin{align}\label{ch1}
		(\sqrt{m+\omega})_t=-(u\sqrt{m+\omega})_x ,
	\end{align}	

Without loss of generality, we fix $\omega=1$  in the CH equation (\ref{ch})   replacing    $m$  by ${m}/{\omega}$ and $\lambda$ by $\lambda\omega$.
	The CH equation (\ref{ch}) is completely integrable and admits the Lax pair
	\begin{equation}
		\Psi_x = X \Psi,\hspace{0.5cm}\Psi_t =T \Psi, \label{lax2}
	\end{equation}
	where $X$ and $T$ in  above Lax pair (\ref{lax2})   are traceless matrixes
	\begin{equation}
		X=\left(\begin{array}{cc}
			0 & 1 \\
			\frac{1}{4}+\lambda(m+1) & 0
		\end{array}\right),\nonumber
	\end{equation}
	\begin{equation}
		T=\frac{1}{2}u_x\sigma_3+\left(\begin{array}{cc}
			0 & \frac{1}{2\lambda}-u \\
			(\frac{1}{2\lambda}-u)(\frac{1}{4}+\lambda(m+1))+\frac{1}{2}u_{xx} & 0
		\end{array}\right),
		\nonumber
	\end{equation}
   which
	implies  that  $\det \Psi(x,t)$    is independent of $x$ and $t$ according to the Abel theorem.

   The Lax pair   (\ref{lax2}) for the CH equation  has
	singularities at $\lambda=0, \lambda= \infty$, so the asymptotic  behavior of their eigenfunctions should   be controlled.
	Following the  idea  due to   Boutet de Monvel and Shepelsky \cite{CH},  we need to use   different transformations to analyze these  singularities  $\lambda=0$ and $\lambda=\infty$ respectively.
	
	\noindent \textbf{Case I: $\lambda=\infty$},\quad $m+1>0$
	
	Let $u(x,t)$ be a solution of (\ref{ch}) with $\omega=1$ such that $u(x,t)\in C^0(\mathbb{R}^+,H^4)$.
 Make a transformation
	\begin{align}
		 \widetilde{\Phi}:=\left(\begin{array}{cc}
			1 & 1 \\
			-ik & ik
		\end{array}\right)^{-1}\left(\begin{array}{cc}
			(m+1)^{\frac{1}{4}} & 0 \\
			0 & (m+1)^{\frac{-1}{4}}
		\end{array}\right)\Psi, \nonumber
	\end{align}
then the Lax pair (\ref{lax2}) is changed  into
	\begin{align}
		&\widetilde{\Phi}_x = -ik\sqrt{m+1}\sigma_3\widetilde{\Phi}+U\widetilde{\Phi},\label{lax3}\\
		&\widetilde{\Phi}_t =-ik(\frac{1}{2\lambda}-u\sqrt{m+1})\sigma_3\widetilde{\Phi}+V\widetilde{\Phi}, \label{lax4}
	\end{align}
	where $\lambda=- k^2 -1/4$, and
	\begin{align}
		&U =\frac{1}{4}\frac{m_x}{m+1}\left(\begin{array}{cc}
			0 & 1 \\
			1 &0
		\end{array}\right)-\frac{1}{8ik}\frac{m}{\sqrt{m+1}}\left(\begin{array}{cc}
			-1 & -1 \\
			1 &  1
		\end{array}\right), \nonumber\\
		&V =-uU+\frac{ik}{4\lambda}\sqrt{m+1}\left(\begin{array}{cc}
			-1 & 1 \\
			-1 & 1
		\end{array}\right)+(\frac{u}{4ik}+\frac{ik}{4\lambda})\frac{1}{\sqrt{m+1}}\left(\begin{array}{cc}
			-1 & -1 \\
			1 & 1
		\end{array}\right)+\frac{ik}{2\lambda}\sigma_3. \nonumber
	\end{align}
	  By conservation law of (\ref{ch1}), we define
	\begin{equation}
		p(x,t,k)=x-\int_{x}^{\infty} (\sqrt{m(\xi,t)+1}-1) d\xi+\frac{t}{2\lambda(k)},
	\end{equation}
Making a transformation
	\begin{equation}
		\Phi(k) =\widetilde{\Phi}(k) e^{ikp(x,t,k)\sigma_3}, \label{trans}
	\end{equation}
	then Lax pair  (\ref{lax3})-(\ref{lax4}) becomes
	\begin{align}
		&\Phi_x (k) = -ikp_x[\sigma_3,\Phi(k)]+U\Phi(k),\label{lax5}\\
		&\Phi_t(k)  = -ikp_t[\sigma_3,\Phi(k)]+V\Phi(k), \label{lax6}
	\end{align}
which  leads  to two  Volterra  integral equations
	\begin{equation}
		\Phi_\pm( k) =I+\int^{x}_{\pm \infty}e^{-ik(p(x)-p(y))\hat{\sigma}_3}(U\Phi_\pm(k) ) dy\label{int1}.
	\end{equation}
	
Denote
	 $\Phi_\pm (  k)=\left( \Phi_\pm^{(1)}(  k),  \Phi_\pm^{(2)}(  k)\right), $
	where  $\Phi_\pm^{(1)}$ and $\Phi_\pm^{(2)}( k)$ are
	the first and second columns of $\Phi_\pm ( k)$, respectively.
	With  the  Volterra  integral equation  (\ref{int1}),  it can be shown that

\begin{Proposition} Let the initial data $u_0  \in  H^{4,2}(\mathbb{R})$.  Then we have

\begin{itemize}

\item[(i)]    $\Phi_{-}^{(1)}(  k)$ and $\Phi_{+}^{(2)}(  k)$ are analytical  in the upper-half complex plane $\mathbb{C}^+$;  $\Phi_{+}^{(1)}(  k)$ and $\Phi_{-}^{(2)}(  k)$ are analytical  in
the lower-half complex plane $\mathbb{C}^-$ (see   Figure \ref{fig:figure2}).

\item [(ii)]  As $k\rightarrow \pm\infty$,
\begin{align}
&(\Phi_{-}^{(1)}(  k), \Phi_{+}^{(2)}(  k)) \rightarrow I, \ \ (\Phi_{+}^{(1)}( k), \Phi_{-}^{(2)}(k)) \rightarrow I,\nonumber\\
&	\Phi(  k)=  \frac{1}{k} \alpha(x,t)\left(\begin{array}{cc}
			1 &  1\\
			-1 & -1
		\end{array}\right)+\mathcal{O}\left( 1 \right),\label{asyphi0}
\end{align}
where $\alpha(x,t)$ is an real function.
	
\item[(iii)]  The Jost functions $ \Phi_\pm (  k) $ admit two kinds of  reduction conditions
	\begin{equation}
		\Phi_\pm (  k)=\overline{\Phi_\pm (  -\bar{k})}, \ \ \Phi^{(1)}_\pm (  -k)= \overline{\Phi^{(1)}_\pm (  \bar{k})}. \label{symPhi1}
	\end{equation}
\end{itemize}

\end{Proposition}

	Since   $\widetilde{\Phi}_\pm (  k) $ are two fundamental matrix solutions of the  Lax  pair (\ref{lax3}),  there exists a linear  relation
	\begin{equation}
		\widetilde{\Phi}_+(  k)=\widetilde{\Phi}_-(x, t, k)S(k). \label{scattering}
	\end{equation}
	Combining  with (\ref{trans}),    the  equation (\ref{scattering})  is changed to
	\begin{equation}
		\Phi_{+}(  k)=\Phi_{-}(  k) \mathrm{e}^{-\mathrm{i} k p(x, t, k) \hat{\sigma}_{3}} S(k), \quad k \in \mathbb{R}, k \neq 0, \label{scattering1}
	\end{equation}
	where $S(k)$ is called scattering matrix
	\begin{equation}
		S(k)=\left(\begin{array}{ll}
			\overline{a(\bar{k})} & b(k) \\
			\overline{b(\bar{k})} & a(k)
		\end{array}\right)=\left(\begin{array}{ll}
			a(-k) & b(k) \\
			b(-k) & a(k)
		\end{array}\right). \label{scattering2}
	\end{equation}
	
	From (\ref{scattering1}), $a(k)$ and  $b(k)$ can be expressed by $\Phi_\pm (k) $ as
	\begin{align}
		&a(k)=\det( \Phi_-^{11}\Phi_+^{22}-\Phi_+^{12}\Phi_-^{21}),\quad \quad b(k)=\det(\Phi_+^{12}\Phi_-^{22}-\Phi_-^{12}\Phi_+^{22})e^{2ikp}.\label{scatteringcoefficient1}
	\end{align}
	In addition,   $\Phi_\pm (  k) $  admit the  asymptotics
	\begin{align}
		\Phi_\pm( , k) =I+\dfrac{D_1}{k}+\mathcal{O}(k^{-2}),\hspace{0.5cm}k \rightarrow \infty.\label{asy1}
	\end{align}
	From  (\ref{scatteringcoefficient1}) and (\ref{asy1}),     we obtain  the asymptotic   of $a(k)$
	\begin{align}
		&a(k)=1+\mathcal{O}(k^{-1}),\hspace{0.5cm}k \rightarrow \infty.\label{asya}\\
		&a(k)=\frac{\beta(x,t)}{k}+\mathcal{O}(1),\hspace{0.5cm}k \rightarrow 0,\label{asya0}
	\end{align}
	where $\beta(x,t)$ is an unknown function.
	
	We define  the reflection coefficients   by
	\begin{equation}
		r(k)=\frac{b(k)}{a(k)},\label{symr}
	\end{equation}
	which admits  symmetry reductions	$r(k)=\overline{r(-\bar{k})}.$
	
	Since  $a(k)$ and $b(k)$  have the same coefficient of  $\frac{1}{k}$  in the asymptotic expansion,  we have  $r(0) = 1$.
	
    It has  shown that the zeros of $a(k)$ in the upper half-plane  are simple zeros lying  on the interval $(-\frac{i}{2},0)\cup (0,\frac{i}{2})$ \cite{PD}.
     Suppose that $a(k)$ has $N$ simple zeros $ik_1,...,ik_{N}$ on $\{k|k\in i\mathbb{R}^+\}$. The  symmetries in (\ref{scattering2})  imply that
	\begin{equation}
		a( ik_n)=0 \Leftrightarrow  a\left(- ik_n\right)=0 , \hspace{0.5cm}n=1,...,N.\nonumber
	\end{equation}
	Therefore, the discrete spectrum is
	\begin{equation}
		\mathcal{Z}=\left\{ ik_n,  -ik_n ,  \right\}_{n=1}^{N}, \label{spectrals}
	\end{equation}
	with $k_n\in (0,\frac{1}{2})$ and $-k_n \in (-\frac{1}{2},0)$. And the distribution  of $	\mathcal{Z}$ on   $iR$   is shown  in Figure \ref{fig:figure2}.
	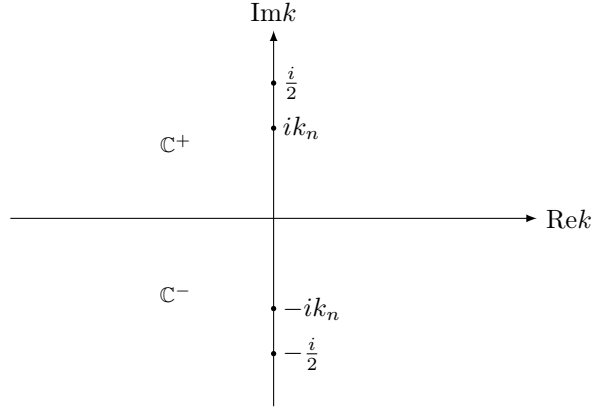
\begin{figure}[H]
		\centering
		\begin{tikzpicture}[node distance=2cm]
			\draw[-latex](-3.5,0)--(3.5,0)node[right]{Re$k$};
			\draw[-latex](0,-2.5)--(0,2.5)node[above]{{\rm Im}$k$};
			\node   at (-1.3,1) {\footnotesize $\mathbb{C}^+$};
			\node   at (-1.3,-1) {\footnotesize $\mathbb{C}^-$};
			\coordinate (A) at (0,1.2);
			\coordinate (B) at (0,-1.2);
			\coordinate (C) at (0,1.8);
			\coordinate (D) at (0,-1.8);
			\fill (A) circle (1pt) node[right] {$ik_n$};
			\fill (B) circle (1pt) node[right] {$-ik_n$};
			\fill (C) circle (1pt) node[right] {$\frac{i}{2}$};
			\fill (D) circle (1pt) node[right] {$-\frac{i}{2}$};
		\end{tikzpicture}
		\caption{\footnotesize   Analytical domains and distribution of the discrete spectrum $\mathcal{Z}$.
			$\Phi_{-}^{(1)}(k), \Phi_{+}^{(2)} (k), a(k)$ are analytical in  $\mathbb{C}^+$ ;
			$\Phi_{+}^{(1)} (k), \Phi_{-}^{(2)}(k)$ are analytical in  $\mathbb{C}^-$ ;   }
		\label{fig:figure2}
	\end{figure}
	
	\noindent \textbf{Case II: $\lambda=0$},\quad $m+1>0$
	
	Define a new transformation
	\begin{equation}
		\Phi^0(k) =Y^{-1}\Psi e^{i(kx+\frac{k}{2\lambda}t)\sigma_3},\label{trans2}
	\end{equation}
	then
	\begin{equation*}
		\Phi^0_\pm (k)  \sim I, \hspace{0.5cm} x \rightarrow \pm\infty,
	\end{equation*}
	and  the Lax pair (\ref{lax2}) change to
	\begin{align}
		&\Phi^0_x = -ik[\sigma_3,\Phi^0]+U^0\Phi^0,\label{lax0.1}\\
		&\Phi^0_t = \frac{-ik}{2\lambda}[\sigma_3,\Phi^0]+V^0\Phi^0, \label{lax0.2}
	\end{align}
	where
	\begin{align}
		&U^0=\frac{-\lambda m}{2ik}\sigma_3+\frac{\lambda m}{2ik}\left(\begin{array}{cc}
			0 & -1 \\
			1 & 0
		\end{array}\right),\nonumber\\
		&V^{0}=u(ik\sigma_3-\frac{\lambda m}{2ik}\left(\begin{array}{cc}
			-1 & -1 \\
			1 & 1
		\end{array}\right))+\frac{m}{4ik}\left(\begin{array}{cc}
			-1 & -1 \\
			1 & 1
		\end{array}\right)\nonumber\\ &+\frac{1}{2}u_x\left(\begin{array}{cc}
			0 & 1 \\
			1 & 0
		\end{array}\right)-\frac{u_{xx}}{4ik}\left(\begin{array}{cc}
			1 & 1 \\
			-1 & -1
		\end{array}\right).\nonumber
	\end{align}
	
	Consider asymptotic expansion as $k\rightarrow \frac{i}{2}$,
	\begin{align}
		\Phi^0=I+ \Phi_1^0 (k-\frac{i}{2})+\mathcal{O}\left( (k-\frac{i}{2})^2\right),\label{asymu0}
	\end{align}
	where
	$$\Phi^0_1= \left(\begin{array}{cc}
		0 & i(u-u_x) \\
		i(u+u_x) & 0
	\end{array}\right). $$
	
	Since $\Phi(k)$ and $\Phi^{0}(k)$ are derived from the same Lax pair (\ref{lax2}), they are linearly dependent. The relations (\ref{trans}) and (\ref{trans2}) lead to
	\begin{equation}
		\Phi_\pm(k):=F(x,t)\Phi^0_\pm (k) e^{ik(\int_{\pm\infty}^x \sqrt{m+1}-1d\xi)\sigma_3},\label{dep}
	\end{equation}
	where
	\begin{align}
		F(x,t)=\frac{1}{2}\left(\begin{array}{cc}
			(m+1)^{\frac{1}{4}}+(m+1)^{-\frac{1}{4}} & (m+1)^{\frac{1}{4}}-(m+1)^{-\frac{1}{4}} \\
			(m+1)^{\frac{1}{4}}-(m+1)^{-\frac{1}{4}}  & (m+1)^{\frac{1}{4}}+(m+1)^{-\frac{1}{4}}
		\end{array}\right).\label{dep1}
	\end{align}

	\subsection{Reflection coefficients}

In this subsection, we establish a mapping between the initial data $u_0(x)$ and the reflection coefficient $r(k)$.
The key is to prove the following proposition

	\begin{Proposition}\label{pror}
		If the initial data $u_0  \in  H^{4,2}(\mathbb{R})$ and $m_0(x)+1=u_{0}(x)-u_{0,xx}(x)+1\geq c_0>0$, then corresponding  reflection  coefficient   $ r(k)\in  H^{1,1}(\mathbb{R})$.
	\end{Proposition}

		Let $I$ is an interval on the real line, $C^0(I,X)$ denotes the space of continuous functions on $I$ taking values in $X$. It is equipped with the norm
	\begin{equation}
		\|f\|_{C^0(I,X)}=\sup\limits_{x\in I}\|f(x)\|_X.\nonumber
	\end{equation}

We prove Proposition \ref{pror}  from  two parts    $I_{\infty}=\{k||k|\geq1\}$ and $I_{0}=\{k||k|<1\}$.
	\subsubsection{Large-k estimates}
	
According to the representation of $a(k),b(k)$, as well as the symmetry of $\Phi(x,t,k)$,  we only need to  consider the first column of Jost solutions.
	
	Let
	\begin{align}
		\textbf{n}_\pm(x,k):=\left(\begin{array}{c}
			n_{11}^{\pm} \\
			n_{21}^{\pm}
		\end{array}\right)=\Phi_\pm^{(1)}(x,k)-\textbf{e}_1, \quad \textbf{e}_1=\left(\begin{array}{c}
			1 \\
			0
		\end{array}\right).\label{n}
	\end{align}
	From (\ref{scatteringcoefficient1}) and (\ref{n}), we rephrase $a(k), b(k)$ as
	\begin{align}
		&	a(k)=(n_{11}^-(0,k)+1)\overline{n_{21}^+(0,k)}-(\overline{n_{11}^+(0,k)}+1)n_{21}^-(0,k),\label{newa}\\
		&	b(k)=\left((\overline{n_{11}^-(0,k)}+1)\overline{n_{21}^+(0,k)}-(\overline{n_{11}^+(0,k)}+1)\overline{n_{21}^-(0,k)}\right)e^{2ikp(0)}.\label{newb}
	\end{align}
    Taking the derivative of (\ref{symr}), we get
	\begin{align}
		r'(k)=\frac{b'(k)}{a(k)}-\frac{b(k)a'(k)}{a^2(k)}.
	\end{align}
	Define the operator $T_\pm$ as follows
	\begin{align}
		T_\pm(f)(x,k)=\int_{\pm\infty}^{x}K_\pm(x,y,k)f(y,k)dy,\label{tpm}
	\end{align}
	where
	\begin{align}
		K_\pm(x,y,k)=\frac{1}{4}\frac{m_x}{m+1}
		\left(\begin{array}{cc}
			0 & 1\\
			e^{2ik(p(x)-p(y))} & 0
		\end{array}\right)-\frac{1}{8ik}\frac{m}{\sqrt{m
				+1}}\left(\begin{array}{cc}
			-1& -1\\
			e^{2ik(p(x)-p(y))} & e^{2ik(p(x)-p(y))}
		\end{array}\right).\label{kpm}
	\end{align}
	From (\ref{int1}),  $\textbf{n}(x,k)$ admits
	\begin{align}
		\textbf{n}_\pm=\textbf{n}_{\pm,0}+T_\pm(\textbf{n}_\pm)=T_\pm(\textbf{e}_1)+T_\pm(\textbf{n}_\pm). \label{n1}
	\end{align}
	Deriving both sides of the equation for $k$, (\ref{n1}) becomes
	\begin{align}
		\textbf{n}_{\pm,k}:=\textbf{n}_{\pm,1}+T_\pm(\textbf{n}_{\pm,k})=(\textbf{n}_{\pm,0})_k+T_{\pm,k}(\textbf{n}_\pm)+T_\pm(\textbf{n}_{\pm,k}),\label{npmk}
	\end{align}
	where
	\begin{align}
		T_{\pm,k}(f)(x,k)=\int_{\pm\infty}^{x}K_{\pm,k}(x,y,k)f(y,k)dy,\nonumber
	\end{align}
	\begin{align}
		K_{\pm,k}(x,y,k)&=\frac{i}{2}\frac{m_x}{m+1}(p(x)-p(y))
		\left(\begin{array}{cc}
			0 & 0\\
			e^{2ik(p(x)-p(y))} & 0
		\end{array}\right)\label{kk}\\ \nonumber
		&+\frac{m}{4k\sqrt{m+1}}(p(x)-p(y))\left(\begin{array}{cc}
			0& 0\\
			-e^{2ik(p(x)-p(y))}& -e^{2ik(p(x)-p(y))}
		\end{array}\right)\\ \nonumber
		&+\frac{1}{8ik^2}\frac{m}{\sqrt{m
				+1}}\left(\begin{array}{cc}
			-1& -1\\
			e^{2ik(p(x)-p(y))} & e^{2ik(p(x)-p(y))}
		\end{array}\right).
	\end{align}
		For simplicity of notation, we omit the subscript "$\pm$".
	\begin{lemma}\label{l1}
		Let $I_{\infty}=\{k||k|\geq1\}$. The following estimates hold.
		\begin{align}
			&\|\textbf{n}_{0}\|_{C^0(\mathbb{R}^
				+,L^2(I_{\infty}))}\lesssim\|m\|_{L^1}+\|m_x\|_{L^2},\label{n0c}\\
			&\|\textbf{n}_{0}\|_{L^2(\mathbb{R}^+,I_{\infty}))}\lesssim\|m\|_{L^{2,\frac{1}{2}}}+\|m_x\|_{L^{2,\frac{1}{2}}},\label{n0l}\\
			&\|(\textbf{n}_{0})_k\|_{C^0(\mathbb{R}^
				+,L^2(I_{\infty}))}\lesssim\|m\|_{L^1}+\|m\|_{L^{2,2}}+\|m_x\|_{L^2}+\|m_x\|_{L^{2,1}}.\label{n0kc}
		\end{align}
		\begin{proof}
			Let
			\begin{align}
				\|\textbf{n}_{0}\|_{C^0(\mathbb{R}^	+,L^2(I_{\infty}))}=\sup\limits_x\|n_{0,1}\|_{L^2(I_{\infty})}+\sup\limits_x\|n_{0,2}\|_{L^2(I_{\infty})},\nonumber
			\end{align}
			where
			\begin{align}
				\|n_{0,1}\|_{L^2(I_{\infty})}&=\left(\int_{I_{\infty}}\left|\int_x^{\infty}\frac{m}{8ik\sqrt{m+1}}dy\right|^{2}dk\right)^{\frac{1}{2}}\leq \int_x^{\infty}\left(\int_{I_{\infty}}\left|\frac{m}{8ik\sqrt{m+1}}\right|^{2}dk\right)^{\frac{1}{2}}dy \nonumber\\
				&=\int_x^{\infty}\left|\frac{m}{8\sqrt{m+1}}\right|\left(\int_{I_{\infty}}\frac{1}{k^{2}}dk\right)^{\frac{1}{2}}dy\lesssim \|m\|_{L^1},\nonumber
			\end{align}
			\begin{align}
				\|n_{0,2}\|_{L^2(I_{\infty})}=\|\frac{1}{4}\int_x^{\infty}\frac{m_y}{m+1}e^{2ik(p(x)-p(y))}dy\|_{L^2(I_{\infty})}+\|\frac{1}{8}\int_x^{\infty}\frac{m}{k\sqrt{m+1}}e^{2ik(p(x)-p(y))}dy\|_{L^2(I_{\infty})}. \label{n02}
			\end{align}
			Using the pairwise definition of the $L^2$ parametrization , the first integral of (\ref{n02}) has the following estimates
			\begin{align}
				 &\|\frac{1}{4}\int_x^{\infty}\frac{m_y}{m+1}e^{2ik(p(x)-p(y))}dy\|_{L^2(I_{\infty})}=\sup\limits_{||\varphi||_{L^2(I_{\infty})}=1}\left|\int_{I_{\infty}}\varphi(k)\int_x^{\infty}\frac{1}{4}\frac{m_y}{m+1}e^{2ik(p(x)-p(y))}dydk\right|\nonumber \\
				&=\sup\limits_{||\varphi||_{L^2(I_{\infty})}=1}\left|\int_x^{\infty}\frac{1}{4}\frac{m_y}{m+1}\widehat {\varphi}(2(p(y)-p(x)))dy\right|\lesssim ||m_y||_{L^2}.\nonumber
			\end{align}
	With  the Minkovski inequality, the second integral of (\ref{n02}) is controlled by $||m||_{L^1}$. Therefore, we get the estimate in (\ref{n0c}).
			
			Let
			\begin{align}
				\|\textbf{n}_{0}\|_{L^2(\mathbb{R}^	+\times L^2(I_{\infty}))}=\|n_{0,1}\|_{L^2(\mathbb{R}^	+\times L^2(I_{\infty}))}+\|n_{0,2}\|_{L^2(\mathbb{R}^	+\times L^2(I_{\infty}))},\nonumber
			\end{align}
			where
			\begin{align}
				\|n_{0,1}\|_{L^2(\mathbb{R}^	+\times L^2(I_{\infty}))}=\left(\int_0^{\infty}\int_{I_{\infty}}\left|\int_x^{\infty}\frac{m}{8ik\sqrt{m+1}}dy\right|^{2}dkdx\right)^{\frac{1}{2}}\lesssim \left(\int_0^{\infty}\int_0^y|m|^2dxdy\right)^{\frac{1}{2}}\lesssim \|m\|_{L^{2,\frac{1}{2}}}, \nonumber
			\end{align}
			\begin{align}
				\|n_{0,2}\|_{L^2(\mathbb{R}^	+\times L^2(I_{\infty}))}=&\left|\int_0^{\infty}\int_{I_{\infty}}\int_x^{\infty}\left|\frac{1}{4}\frac{m_y}{m+1}e^{2ik(p(x)-p(y))}dy\right|^{\frac{1}{2}}dkdx\right|^{\frac{1}{2}}\nonumber\\
				&+\left|\int_0^{\infty}\int_{I_{\infty}}\int_x^{\infty}\left|\frac{1}{8}\frac{m}{k\sqrt{m+1}}e^{2ik(p(x)-p(y))}dy\right|^{\frac{1}{2}}dkdx\right|^{\frac{1}{2}}\nonumber\\
				&\lesssim \left|\int_0^{\infty}\int_0^{y}\Big|m_y\Big|^2dxdy\right|^{\frac{1}{2}}+\left|\int_0^{\infty}\int_0^{y}\Big|m\Big|^2dxdy\right|^{\frac{1}{2}}\lesssim\|m_x\|_{L^{2,\frac{1}{2}}}+\|m\|_{L^{2,\frac{1}{2}}}.\nonumber
			\end{align}
			Thus (\ref{n0l}) has been proved.
			
			Noticing that
			\begin{align}
				|p(x)-p(y)|=|y-x+\int_x^y\sqrt{m+1}-1d\xi|\leq|y-x|+\|m\|_{L^1},\label{p-p}
			\end{align}
			we can apply the method used in estimating $\textbf{n}_{0}$ to obtain the estimate for $(\textbf{n}_{0})_k$.
		\end{proof}
	\end{lemma}
	
	\begin{lemma}\label{l2}
		Supposing that $u_0 \in H^{4,2}(\mathbb{R})$, the following bound of operator $T_k$ holds uniformly in $u_0(x)\in H^{4,2}(\mathbb{R})$, and $T_k$ is Lipschitz continuous for $u_0(x)$ with
		\begin{equation}
			\|T_k\|_{L^2(\mathbb{R}^+\times I_{\infty})\to C^0(\mathbb{R}^+,\ L^2(I_{\infty}))}\lesssim
			\|m\|_{L^1}(\|m\|_{L^2}+\|m_x\|_{L^{2,1}})+\|m\|_{H^{1,1}}.\label{tk}
		\end{equation}
		\begin{proof}
			Notice that
			\begin{equation}
				T_k(f)(x,k)=\int_{\pm\infty}^xK_k(x,y,k)f(y,k)dy,
			\end{equation}
			where $K_k(x,y,k)$ is shown in (\ref{kk}).
			Thus, for any given function $f(x,k)\in L^2(\mathbb{R}^+\times I_{\infty})$, it follows that
				\begin{align}
					&|T_k(f)(x,k)|\nonumber\\
					&\lesssim \left[(\int_x^{\infty}|(p(y)-p(x))m_y|^2dy)^{\frac{1}{2}}+(\int_x^{\infty}|m|^2dy)^{\frac{1}{2}}+
					(\int_x^{\infty}|m(p(y)-p(x))|)|^2dy)^{\frac{1}{2}}\right]
					(\int_x^{\infty}\|f\|^2_{L^2(I_\infty)}dy)^{\frac{1}{2}},\nonumber
				\end{align}
			then
			\begin{equation}
				||T_k(f)(x,k)||_{C^0(\mathbb{R}^+,L^2(I_{\infty}))}\lesssim
				\left(||m||_{L^1}(||m||_{L^2}+||m_x||_{L^{2,1}})+||m||_{H^{1,1}}\right)||f||_{L^2(\mathbb{R}^+\times I_{\infty})},\nonumber
			\end{equation}
			which implies (\ref{tk}).
		\end{proof}
	\end{lemma}
	
	\begin{lemma}\label{l3}
		Suppose that $u_0 \in H^{4,2}(\mathbb{R})$. The resolvent $(I-T)^{-1}$ exists as a bounded operator in $C^0(\mathbb{R}^+,L^2(I_{\infty}))$, and the operator $L=(I-T)^{-1}-I$ is an integral operator with the continuous integral kernel $L(x,y,k)$, which satisfies the following estimate
		\begin{equation}
			|L(x,y,k)|\leqslant e^{||h||_{L^1}}h(y).\label{lguji}
		\end{equation}
		\begin{proof}
			By (\ref{tpm}) and (\ref{kpm}), it's obvious that $T$ is a Volterra operator and the integral kernel $K(x,y,k)$ satisfies
			\begin{equation}
				\sup_{k\in I_{\infty}}|K(x,y,k)|\leqslant h(y),\quad
				h(y)=|m_y|+|m|.\label{suph}
			\end{equation}
			Suppose $f(x,k)\in C^0(\mathbb{R}^+,L^2(I_{\infty})$, and denote $f^*(x)\triangleq\sup_{y\geqslant x}||f(\cdot,k)||_{L^2(I_{\infty})}$. From above estimate (\ref{suph}), we deduce that
			\begin{equation}
				(Tf)^*(x)\leqslant\int_x^{\infty}
				h(y)f^*(y)dy.\label{tf}
			\end{equation}
			Thus we obtain that $(I-T)^{-1}$ exists and is a bounded operator on $C^0(\mathbb{R}^+,L^2(I_{\infty}))$. The operator $L$ has an integral kernel as follows
			\begin{equation}
				L(x,y,k)=\left\{\begin{array}{ll}\sum_{n=1}^{\infty}K_n(x,y,k),\quad &x\leqslant y,\\
					0,& x>y,\end{array}\right.  \nonumber
			\end{equation}
			where
			\begin{equation}
				K_n(x,y,k)=\int_{x\leqslant y_1\leqslant \dots \leqslant y_{n-1}}K(x,y_1,k)K(y_1,y_2,k)\dots K(y_{n-1},y,k)dy_{n-1}...dy_1.\label{kn}
			\end{equation}
			Combining (\ref{kn}) with (\ref{suph}), the estimate
			\begin{equation}
				|K_n(x,y,k)|\leqslant \frac{1}{(n-1)!}\left(\int_x^{\infty}h(\zeta)d\zeta\right)^{n-1}h(y) \nonumber
			\end{equation}
			holds, which proves (\ref{lguji}).
		\end{proof}
	\end{lemma}
	\noindent\textbf{Remark:}
	Notice that the formula $L=(I-T)^{-1}-I$ can be rewritten as $L=T+T(I-T)^{-1}T$, and the bounds on $T:L^2\to C^0,$ $T:C^0\to L^2 ,$ $T:L^2\to L^2$ are proved by (\ref{suph}) and (\ref{tf}). It's easy to deduce that $L$ belongs to  $\mathcal{B}(C^0(\mathbb{R}^+,L^2(I_{\infty})))\cap\mathcal{B}(L^2(\mathbb{R}^+\times I_{\infty}))$ with the following estimates
	\begin{eqnarray}
		||L||_{\mathcal{B}(C^0(\mathbb{R}^+,L^2(I_{\infty})))}\lesssim
		(||m_x||_{L^1}+||m||_{L^1})\exp(||m_x||_{L^1}+||m||_{L^1}),\nonumber\\
		||L||_{\mathcal{B}(L^2(\mathbb{R}^+\times I_{\infty}))}\lesssim
		(||m_x||_{L^{2,1}}+||m||_{L^{2,1}})\exp(||m_x||_{L^1}+||m||_{L^1}). \nonumber
	\end{eqnarray}

	In order to prove Proposition \ref{pror} under the condition that $k\in I_{\infty}$, we  need to prove the  three propositions.
	\begin{Proposition}\label{pron}
		The map $u_0(x)$ $\rightarrow$ $\textbf{n}_\pm(0,k)$ is Lipschitz continuous from $H^{4,2}(\mathbb{R})$ into $H^{1}(I_{\infty})$.
	\end{Proposition}

\begin{proof}
	To prove $\textbf{n}\in H^1(I_{\infty})$, we need to show that $\textbf{n}\in L^2(I_{\infty})$ and $\textbf{n}_{k}\in L^2(I_{\infty})$:
	\begin{enumerate}
		\item For each $u_0 \in H^{4,2}(\mathbb{R})$ and $k\in I_{\infty}$, there exists a unique solution of (\ref{n1}) such that $\textbf{n}\in C^0(\mathbb{R}^+,L^2(I_{\infty}))\cap L^2(\mathbb{R}^+\times I_{\infty})$ and the map $u_0(x)\to \textbf{n}$ is Lipschitz continuous from $H^{4,2}(\mathbb{R})$ to $C^0(\mathbb{R}^+,L^2(I_{\infty}))\cap L^2(\mathbb{R}^+\times I_{\infty})$, which can be directly proved by Lemma \ref{l1}, Lemma \ref{l3} and the Remark;
		\item By (\ref{n0kc}), Lemma \ref{l2} and the result of item 1, the inhomogeneous term of (\ref{npmk})
		\[\textbf{n}_{1}=(\textbf{n}_{0})_k+T_{k}(\textbf{n})\]
		belongs to $C^0(\mathbb{R}^+,L^2(I_{\infty}))$. Combining this result with Lemma \ref{l3} and the Remark, for $u_0  \in H^{4,2}(\mathbb{R})$ and $k\in I_{\infty}$, (\ref{npmk}) admits a unique solution $\textbf{n}_{k}$ with $\textbf{n}_{k} \in C^0(\mathbb{R}^+,L^2(I_{\infty}))$, and the map $u_0(x)\to\textbf{n}_{k} $ is Lipschitz continuous from $H^{4,2}(\mathbb{R})$ to $C^0(\mathbb{R}^+,L^2(I_{\infty}))$.
	\end{enumerate}
\end{proof}

	\begin{Proposition}\label{prokb}
		The map $u_0(x)$ $\rightarrow$ $kb(k)$ is Lipschitz continuous from $H^{4,2}(\mathbb{R})$ into $L^2(I_{\infty})$.
	\end{Proposition}

\begin{proof}
	By  (\ref{newb}), we find that
	\begin{equation}
		kb(k)=e^{2ikp(0)}\left(\overline{n_{11}^-(0,k)}(k\overline{n_{21}^+(0,k)})-\overline{n_{11}^+(0,k)}(k\overline{n_{21}^-(0,k)})+k\overline{n_{21}^+(0,k)}
		-k\overline{n_{21}^-(0,k)}\right). \nonumber
	\end{equation}
	We first handle with the item $kn_{21}$:
	\begin{equation}
		\begin{aligned}
			&kn_{21}:=\Pi_1+\Pi_2,
		\end{aligned}
	\end{equation}
	where
	\begin{equation}
		\begin{aligned}
			&\Pi_1=\int^x_{\pm \infty}\frac{k}{4}\frac{m_y}{m+1}e^{2ik(p(x)-p(y))}(n_{11}+1)dy;\\
			&\Pi_2=\int^x_{\pm \infty}-\frac{1}{8i}
			\frac{m}{\sqrt{m+1}}e^{2ik(p(x)-p(y))}dy-\int^x_{\pm \infty}\frac{1}{8i}
			\frac{m}{\sqrt{m+1}}e^{2ik(p(x)-p(y))}(n_{11}+n_{21})dy.
		\end{aligned}
	\end{equation}
	As for $\Pi_1$, integrating by part over $y$, it comes to
	\begin{equation}
		\begin{aligned}
			&\Pi_1=-\frac{1}{8i}\int_{\pm\infty}^x\frac{m_y}{(m+1)^{3/2}}(n_{11}+1)de^{2ik(p(x)-p(y))}\\
			&=-\frac{1}{8i}\frac{m_x}{(m+1)^{3/2}}-\frac{1}{8i}\frac{m_x}{(m+1)^{3/2}}n_{11}+\frac{1}{8i}\int_{\pm\infty}^x\left(\frac{m_y}{(m+1)^{3/2}}(n_{11}+1)\right)_ye^{2ik(p(x)-p(y))}dy\\
			&:=\Pi_0+\Pi_{11}=-\frac{1}{8i}\frac{m_x}{(m+1)^{3/2}}+\Pi_{11}. \nonumber
		\end{aligned}
	\end{equation}
	By the similar method as in Lemma \ref{l1}, it's easy to deduce that $\Pi_{11},\ \Pi_{2}\in L^2(I_{\infty})$ and their norms are controlled by $\|m\|_{H^{2,2}}$. We now denote $kn_{21}$ as
	\[kn_{21}=\Pi_0(x)+\Pi_{11}(x,k)+\Pi_{2}(x,k),\]
	where $\Pi_0$ is bounded and independent with $k$.\\
	Thus,
	\begin{equation}
		\begin{aligned}
			 kb(k)=e^{2ikp(0)}&\left[\overline{n_{11}^-(0,k)}\left(\overline{\Pi_0(0)}+\overline{\Pi_{11}^+(0,k)}+\overline{\Pi_2^+(0,k)}\right)-\overline{n_{11}^+(0,k)}\left(\overline{\Pi_0(0)}+\overline{\Pi_{11}^-(0,k)}+\overline{\Pi_2^-(0,k)}\right)\right.\\
			&+\left.\overline{\Pi_{11}^+(0,k)}+\overline{\Pi_2^+(0,k)}-\overline{\Pi_{11}^-(0,k)}-\overline{\Pi_2^-(0,k)}\right]\in L^2(I_{\infty}).\nonumber
		\end{aligned}
	\end{equation}
	\end{proof}

	\begin{Proposition}\label{prokbb}
		The map $u_0(x)$ $\rightarrow$ $kb'(k)$ is Lipschitz continuous from $H^{4,2}(\mathbb{R})$ into $L^2(I_{\infty})$.
	\end{Proposition}
	(\ref{newa}), (\ref{newb}) and Proposition \ref{pron} state that $a(k), b(k)\in H^1(I_{\infty})$, which proved Proposition \ref{pror} by combining the results of Propositions \ref{prokb} and \ref{prokbb}.

\begin{proof}
	
	We first take the derivative of $b(k)$, and multiply it by $k$:
	\begin{equation}
		\begin{aligned}
			kb'(k)=2ip(0)(kb(k))+e^{2ikp(0)}&\left[(\overline{n_{11}^-(0,k)})_k(k\overline{n_{21}^+(0,k)})-
			(\overline{n_{11}^+(0,k)})_k(k\overline{n_{21}^-(0,k)})\right.\\
			&+\overline{n_{11}^-(0,k)}(k(\overline{n_{21}^+(0,k)})_k)
			-\overline{n_{11}^+(0,k)}(k(\overline{n_{21}^-(0,k)})_k)\\
			&+\left.k(\overline{n_{21}^+(0,k)})_k-k(\overline{n_{21}^-(0,k)})_k\right].
		\end{aligned}\label{kbp}
	\end{equation}
	As proved above, the first line of (\ref{kbp}) belongs to $L^2(I_{\infty})$. Denote the remains as
	\begin{equation}
		\tilde{\Pi}=e^{2ikp(0)}\left[\overline{n_{11}^-(0,k)}(k(\overline{n_{21}^+(0,k)})_k)
		-\overline{n_{11}^+(0,k)}(k(\overline{n_{21}^-(0,k)})_k)+k(\overline{n_{21}^+(0,k)})_k-k(\overline{n_{21}^-(0,k)})_k\right].\label{pip}
	\end{equation}
	Using the same method shown in Proposition \ref{prokb}, we deal with $k(n_{21})_k$:
	\begin{equation}
		\begin{aligned}
			&k(n_{21})_k=\tilde{\Pi}_1+\tilde{\Pi}_2,\\
			&\tilde{\Pi}_1=\int_{\pm\infty}^x\left[\frac{ik}{2}(p(x)-p(y))e^{2ik(p(x)-p(y))}(n_{11}+1)+\frac{k}{4}\frac{m_y}{m+1}e^{2ik(p(x)-p(y))}(n_{11})_k\right]dy,\\
			&\tilde{\Pi}_2=\int_{\pm\infty}^x\left[\left(\frac{1}{8ik}+\frac{1}{4}(p(x)-p(y))\right)\frac{m}{\sqrt{m+1}}e^{2ik(p(x)-p(y))}\left((n_{11}+1)+n_{21}\right)\right]dy\\
			&+\int_{\pm\infty}^x\frac{1}{4}\frac{m}{\sqrt{m+1}}e^{2ik(p(x)-p(y))}\left((n_{11})_k+(n_{21})_k\right)dy. \nonumber
		\end{aligned}
	\end{equation}
	From the estimate (\ref{p-p}) and the fact that $m(x)\in H^{2,2}$, it's obvious that
	\begin{equation}
		\lim_{x\to \pm\infty}|m_x(p(x)-p(y))|=0.
	\end{equation}
	Take a divisional integral on $\tilde{\Pi}_1$:
	\begin{equation}
		\begin{aligned}
			\tilde{\Pi}_1&=\int_{\pm\infty}^x\left[-\frac{1}{4}(p(x)-p(y))\frac{m_y}{(m+1)^{3/2}}(n_{11}+1)+
			\frac{i}{2}\frac{m_y}{(m+1)^{3/2}}(n_{11})_k\right]de^{2ik(p(x)-p(y))}\\
			&=\frac{i}{2}\frac{m_x}{(m+1)^{3/2}}(n_{11})_k+\int_{\pm\infty}^x\left[\frac{1}{4}\frac{m_y}{(m+1)^{3/2}}\left((p(x)-p(y))(n_{11}+1)-2i (n_{11})_k\right)\right]_ydy.
		\end{aligned}\nonumber
	\end{equation}
	Using the results we have deduced all above, it's not hard to get that $\tilde{\Pi}_1(0,k),\ \tilde{\Pi}_2(0,k)\in L^2(I_{\infty})$, thus $k(n_{21})_k(0,k)\in L^2(I_{\infty})$. Finally we have proved $\tilde{\Pi}$ defined as (\ref{pip}) belongs to $L^2(I_{\infty})$.

	\end{proof}

	\subsubsection{Small-k estimates}	
	We need to find a new set of Jost solutions for the estimation in the case of small $k$.
	Let
	  $$\Phi_{\pm}^*( k)=Y(k)\Phi_{\pm}( k), \ \ \ Y(k):=\left(\begin{array}{cc}
		1 & 1 \\
		-ik & ik
	\end{array}\right).$$
Then 	
\begin{align}
		\Phi_{\pm}^*( k)=Y(k)+\int^{x}_{\pm \infty}e^{-ik(p(x)-p(y))\hat{\sigma}_3}(U^*\Phi_\pm^*)(y,t,k)dy,\label{phi*}
	\end{align}
	where
	\begin{align}
		U^* =Y(k)U Y^{-1}(k)=\frac{1}{4}\frac{m_x}{m+1}\sigma_3-\frac{1}{2}\frac{m}{\sqrt{m+1}}\left(\begin{array}{cc}
			0 & 0 \\
			1 & 0
		\end{array}\right). \nonumber
	\end{align}
	Direct calculation gives
	\begin{align}
		&a^*(k)=\Phi_-^{11,*}\Phi_+^{22,*}-\Phi_+^{12,*}\Phi_-^{21,*} ,\quad \quad b(k)=(\Phi_+^{12,*}\Phi_-^{22,*}-\Phi_-^{12,*}\Phi_+^{22,*})e^{2ikp(0)}, \nonumber
	\end{align}
which implies that
	\begin{align}
		r^*(k)=\frac{b^*(k)}{a^*(k)}=\frac{b(k)}{a(k)}=r(k).\label{newr*}
	\end{align}
	
	The feature of (\ref{phi*}) and (\ref{newr*}) suggest that we can proceed with this part of the proof as in large $k$ estimates, leading to the corresponding conclusion.
	\begin{Proposition}\label{pron1}
		The maps $u_0(x)$ $\rightarrow$ $a^*(k)$, $u_0(x)$ $\rightarrow$ $b^*(k)$ are Lipschitz continuous from $H^{4,2}(\mathbb{R})$ into $H^{1}(I_{0})$, and the maps $u_0(x)$ $\rightarrow$ $kb^*(k)$, $u_0(x)$ $\rightarrow$ $k(b^{*})'(k)$ are Lipschitz continuous from $H^{4,2}(\mathbb{R})$ into $L^2(I_{0})$.
\end{Proposition}
		
\begin{proof}
			For the case $k\in I_0$, we define $\tilde{\textbf{n}}=(\Phi^{*}-I)\textbf{e}_1$. Then the key process is to prove that $u_0(x)$ $\rightarrow$ $\tilde{\textbf{n}}$ is Lipschitz continuous from $H^{4,2}(\mathbb{R})$ into $H^{1}(I_{0})$. However, this problem can be solved with the methods we used in Proposition {\ref{pron}}. The remaining conclusion of the Proposition \ref{pron1} is simply  shown by using divisional integrals to make integral estimates, like the process of proving Proposition \ref{prokb} and \ref{prokbb}.
		\end{proof}

	
	Combining the conclusions of Proposition \ref{pron}-\ref{pron1}, we have a complete proof of Proposition \ref{pror}.

	\subsection{A  basic  RH problem}
	\quad
	We introduce a  new   scale
	\begin{equation}
		y(x,t)=x-\int_{x}^{+\infty} \left(\sqrt{m+1}-1\right) d\xi. \label{trans3}
	\end{equation}
	The price to pay for this scale is that the solution of the initial problem can be given only implicitly.
	By the definition of the new scale $y(x, t)$, we define
	\begin{equation}
		M(k)=M(k;y,t)\triangleq \left\{ \begin{array}{ll}
			\left(\Phi_{-}^{(1)}(k;y,t), \frac{ \Phi_{+}^{(2)} (k;y,t) } {a(k)}\right),   &\text{as } k\in \mathbb{C}^+,\\[12pt]
			\left( \frac{\Phi_{+}^{(1)}(k;y,t)}{a(-k)},\Phi_{-}^{(2)}(k;y,t)\right)  , &\text{as }k\in \mathbb{C}^-,\\
		\end{array}\right. \label{tr1}
	\end{equation}
	where the  phase function is
	\begin{equation}
		\theta(k)= k \left[\frac{y}{t}+\frac{1}{2\lambda (k)}  \right].\label{theta}
	\end{equation}
	Then $M(k)$    satisfies  the following  RH problem for the new variable $(y,t)$
	\begin{RHP}\label{RHP2}
		Find a matrix-valued function $	M(k) $ which satisfies:
		
		$\blacktriangleright$ Analyticity: $M(k)$ is meromorphic in $\mathbb{C}\setminus \mathbb{R}$ and has single poles;
		
		$\blacktriangleright$ Symmetry: $M(k)=\overline{M(-\bar{k})}$;
		
		$\blacktriangleright$ Jump condition: $M(k)$ has continuous boundary values $M_\pm(k)$ on $\mathbb{R}$ and
		\begin{equation}
			M_+(k)=M_-(k)V(k),\hspace{0.5cm}k \in \mathbb{R}, \nonumber
		\end{equation}
		where
		\begin{equation}
			V(k)=\left(\begin{array}{cc}
				1 & r(k)e^{-2it\theta}\\
				-\overline{r(k)}e^{2it\theta}&  1-|r(k)|^2
			\end{array}\right);\nonumber
		\end{equation}
		
		$\blacktriangleright$ Asymptotic behaviors:
		\begin{align}
			&M(k) = I+\mathcal{O}(k^{-1}),\hspace{0.5cm}k \rightarrow \infty,\nonumber\\
			&M(k)=\frac{\alpha_-(x,t)}{k}\left(\begin{array}{cc}
				1 &0\\
				-1& 0
			\end{array}\right)+\mathcal{O}(1),\hspace{0.5cm}k \to 0,\ {\rm Im} k>0,\nonumber\\
			&M(k)=\frac{\alpha_-(x,t)}{k}\left(\begin{array}{cc}
				0 &1\\
				0& -1
			\end{array}\right)+\mathcal{O}(1),\hspace{0.5cm}k \to 0,\ {\rm Im}k<0. \nonumber
		\end{align}
		$\blacktriangleright$ Residue conditions: $M(k)$ has simple poles at each point in $ \mathcal{Z} $ with:
		\begin{align}
			&\res_{k=ik_n}M(k)=\lim_{k\to ik_n}M(k)\left(\begin{array}{cc}
				0 & c_ne^{-2it\theta(ik_n)}\\
				0& 0
			\end{array}\right),\\
			&\res_{k=-ik_n}M(k)=\lim_{k\to -ik_n}M(k)\left(\begin{array}{cc}
				0 &0\\
				\bar{c}_ne^{2it\theta(-ik_n)} & 0
			\end{array}\right),
		\end{align}
	\end{RHP}
where $c_n=b(k_n)/a'(k_n), \ n=1,\cdots, N$  are called norming constants.
	To reconstruct the potential $u(x,t)$, we define two functions:
	\begin{align}
		&\mu_{1}(y, t, k):=\frac{\left(\Phi_{-}^{11}\left(y, t, k\right)+\Phi_{-}^{21}\left(y, t, k\right)\right)} {a\left(k\right)}, \ &\mu_{2}(y, t, k):=\Phi_{+}^{12}\left(y, t, k\right)+\Phi_{+}^{22}\left(y, t, k\right). \nonumber
	\end{align}
	Therefore,   we  get the following formula
	\begin{align}
		&u(y, t)=\frac{1}{2 \mathrm{i}} \lim _{k \rightarrow \frac{\mathrm{i}}{2}}\left(\frac{\mu_{1}(y, t ; k) \mu_{2}(y, t ; k)}{\mu_{1}(y, t ; \frac{\mathrm{i}}{2}) \mu_{2}(y, t ; \frac{\mathrm{i}}{2})}-1\right) \frac{1}{k-\frac{\mathrm{i}}{2}},\label{weishi}\\
		&x(y, t)=y+\ln \frac{\mu_{1}\left(y, t ; \frac{\mathrm{i}}{2}\right)}{\mu_{2}\left(y, t ; \frac{\mathrm{i}}{2}\right)}. \label{xy}
	\end{align}
	Combining (\ref{tr1}) and (\ref{weishi}), the solution $u(x, t)$ of the initial value problem of the CH equation is represented by the solution of RHP \ref{RHP2}
	\begin{align}
		&u(y, t)=\frac{1}{2 \mathrm{i}} \lim _{k \rightarrow \frac{\mathrm{i}}{2}}\left(\frac{\left(M_{11}(k)+M_{21}(k)\right)\left(M_{12}(k)+M_{22}(k)\right)}{\left(M_{11}(\frac{i}{2})+M_{21}(\frac{i}{2})\right)\left(M_{12}(\frac{i}{2})+M_{22}(\frac{i}{2})\right)}-1\right) \frac{1}{k-\frac{\mathrm{i}}{2}} . \label{cover}
	\end{align}
	
	Noticing that $k = 0$ is a singularity point of $M(k)$,  we  make a  transformation to remove it.
	\begin{align}
		\left(I-\frac{1}{k-1}\sigma_1\right)M(k) =\left(I-\frac{1}{k-1}\sigma_1\left(M^J(y,t,1)\right)^{-1}\right)M^J( k),\label{fenjie}
	\end{align}
	then $M^J( k)$  satisfies a new RH problem
	\begin{RHP}\label{RHP2-2}
		Find a matrix-valued function $	M^J(k) $ which satisfies:
		
		$\blacktriangleright$ Analyticity: $M^J(k)$ is analytic in $\mathbb{C}\setminus \left(\mathbb{R} \cup \mathcal{Z}  \right)$;
		
		$\blacktriangleright$ Jump condition: $M^J(k)$ has continuous boundary values $M^J_\pm(k)$ on $\mathbb{R}$ and
		\begin{equation}
			M^J_+(k)=M^J_-(k)V(k),\hspace{0.5cm}k \in \mathbb{R},
		\end{equation}
		
		$\blacktriangleright$ Asymptotic behaviors:
		\begin{align}
			&M^J(k) = I+\mathcal{O}(k^{-1}),\hspace{0.5cm}k \rightarrow \infty,\label{jianjin5}
		\end{align}
		
		$\blacktriangleright$ Residue conditions: $M^J(k)$ has simple poles at each point in $ \mathcal{Z} $ with:
		\begin{align}
			&\res_{k=ik_n}M^J(k)=\lim_{k\to ik_n}M^J(k)\left(\begin{array}{cc}
				0 & c_ne^{-2it\theta(ik_n)}\\
				0& 0
			\end{array}\right),\\
			&\res_{k=-ik_n}M^J(k)=\lim_{k\to -ik_n}M^J(k)\left(\begin{array}{cc}
				0 &0\\
				\bar{c}_ne^{2it\theta(-ik_n)} & 0
			\end{array}\right).
		\end{align}
	\end{RHP}
	
	\subsection{Solvability of RH problem}\label{subsec2.3}

In this subsection, we show the solvability of   RHP \ref{RHP2-2}  with
  Zhou's Vanishing Lemma.
	Although  the jumps that poles  convert to  do not satisfy Hermite symmetry,
they  are oriented to preserve Schwartz reflection symmetry in the real axis $\mathbb{R}$ and the matrix-valued function $F(k)=M^J(k)M^J(\bar k)^H$  is still analytic for $k \in \mathbb{C}\setminus \mathbb{R}$.
 We get the following proposition
	
	\begin{Proposition}
		By  changing the asymptotic condition (\ref{jianjin5}) into $M^J(k)= \mathcal{O}(k^{-1})  \hspace{0.2cm} as \hspace{0.2cm} k \rightarrow \infty$, the solution to the modified RHP \ref{RHP2-2}  is identically zero.
	\end{Proposition}

		\begin{proof}
			We first recall the jump matrix produced by poles.
			\begin{align}
				V(k)=\left\{\begin{array}{ll}\left(\begin{array}{cc}
						1 & \frac{c_ne^{-2it\theta(ik_n)}}{k-ik_n} \\
						0&  1
					\end{array}\right),   &\text{as } k\in\partial \mathcal{D}_n:=\{k:|k-ik_n|=\rho\},\\[12pt]
					\left(\begin{array}{cc}
						1 & 0\\
						\frac{\bar{c}_ne^{2it\theta(-ik_n)}}{k+ik_n} & 1
					\end{array}\right),   &\text{as } k\in\partial \bar{\mathcal{D}}_n:=\{k:|k+ik_n|=\rho\}.\\[12pt]
				\end{array}\right.
			\end{align}
			Note that $V(k)=V(\bar{k})^H$ for $k \in \partial \mathcal{D}_n \cup \partial \bar{\mathcal{D}}_n$, where the superscript $H $ denotes the complex conjugation and transpose of a given matrix. Although the jumps on the real axis do not have good symmetry, the required properties can be obtained by simple calculations.
			
			Now we focus on the matrix-valued function $F(k)=M^J(k)M^J(\bar k)^H$,  we  prove that
			\begin{equation}\label{F_+}
				\int_{\mathbb{R}}F_{\pm}(k)dk=0.
			\end{equation}
			From the definition of $F(k)$, it is clear that $F(k)$ is analytic in  $\mathbb{C} \setminus (\mathbb{R}\cup \partial \mathcal{D}_n \cup \partial \bar{\mathcal{D}}_n )$. For $k \in  \partial \mathcal{D}_n $, we  find that
			\begin{align}
				F_+(k)=&M_+^J(k)M_-^J(\bar k)^H=M_-^J(k)V(k)\left(V^{-1}(\bar{k})\right)^HM_+^J(\bar k)^H \nonumber \\
				=&M_-^J(k)M_+^J(\bar k)^H=F_-(k).\nonumber
			\end{align}
			Thus by Morera's theorem, $F(k)$ is analytic in $\mathbb{C} \setminus \mathbb{R}$. The Beals-Coifman theorem implies that $M^J(k)\sim \mathcal{O}(k^{-1}),\hspace{0.2cm} k \in \mathbb{C} \setminus \mathbb{R},$ and $F(k) \sim \mathcal{O}(k^{-2}),\hspace{0.2cm} k \in \mathbb{C} \setminus \mathbb{R}$, then (\ref{F_+}) follows from the Cauchy intergral theorem.
			
			Take the expression of $F(k)$ into (\ref{F_+}). Notice that for $k \in \mathbb{R}$,
			\begin{equation}
				F_+(k)=M_+^J(k)M_-^J(\bar k)^H=M_-^J(k)V(k)M_-^J(\bar k)^H,\nonumber
			\end{equation}
			and
			\begin{equation}
				F_-(k)=M_-^J(k)M_+^J(\bar k)^H=M_-^J(k)V(k)^HM_-^J(\bar k)^H,\nonumber
			\end{equation}
			so
			\begin{equation}\label{V}
				\int_{\mathbb{R}}M_-^J(k)\left(V(k)+V(k)^H\right)M_-^J(k)^H=0.
			\end{equation}
			For $k \in \mathbb{R}\setminus \{0\}$,
			\begin{align}
				V(k)+V(k)^H=2\left(\begin{array}{cc}
					1 & 0 \\
					0&  1-|r(k)|^2
				\end{array}\right), \nonumber
			\end{align}
			since the element $1-|r(k)|^2=\frac{1}{|a|^2}>0$, the matrix above is positive definite and has positive eigenvalues. Meanwhile, $M^J(0)\neq0$. Therefore, we can deduce from (\ref{V}) that $M_-^J(k)=0$. Furthermore , $M_+^J(k)=M_-^J(k)V(k)=0,\hspace{0.2cm} k \in  \mathbb{R}$. The rest part of the proof is similar.
		\end{proof}

	Further by  the Fredholm alternative theorem, we then obtain
	
	\begin{Proposition}
		Let the initial data $u_0  \in H^{4,2}(\mathbb{R})$, then  the RHP \ref{RHP2-2} has a unique solution.
	\end{Proposition}

	\subsection{Classification  of  asymptotic regions }\label{subsec2.4}
	We note that the jump matrix and residue conditions of RHP \ref{RHP2} contain the exponential function $e^{\pm2it\theta}$, which is an oscillatory term in long-time asymptotics. Therefore we need to control the real part of $2it\theta$ and decompose the jump matrix by decay region.
	\begin{align}
		&\text{Re}(2it\theta)=-2t\text{Im}k\left[\xi+\frac{2(\text{Re}^2k+\text{Im}^2k)-\frac{1}{2} }{(2(\text{Re}^2k-\text{Im}^2k)+\frac{1}{2})^2+(4\text{Re}k\text{Im}k)^2}\right],\label{Reitheta}
	\end{align}
	where $\xi=\frac{y}{t}$. The signature  of $\text{Im}\theta$ are shown in Figure \ref{figtheta}. Calculation gives that $\kappa_0^2(\xi)=\frac{\sqrt{1+4\xi}-1-\xi}{4\xi}$, which suggest us to divide half-plane  $-\infty <y<\infty, \ t> 0$  in four space-time regions:
	
	\noindent Case I: $\xi<-1/4$ in   Figure \ref{figtheta} (a),  \hspace{0.7cm} Case II: $-1/4<\xi<0$  in  Figures  \ref{figtheta} (b),
	
	\noindent Case III: $0\leq\xi<2$ in   Figure \ref{figtheta} (c),\hspace{0.3cm} \quad Case IV: $\xi>2$  in   Figure \ref{figtheta} (d).
	
	\begin{figure}[h]
		\centering
		\subfigure[]{\includegraphics[width=4cm,height=4cm]{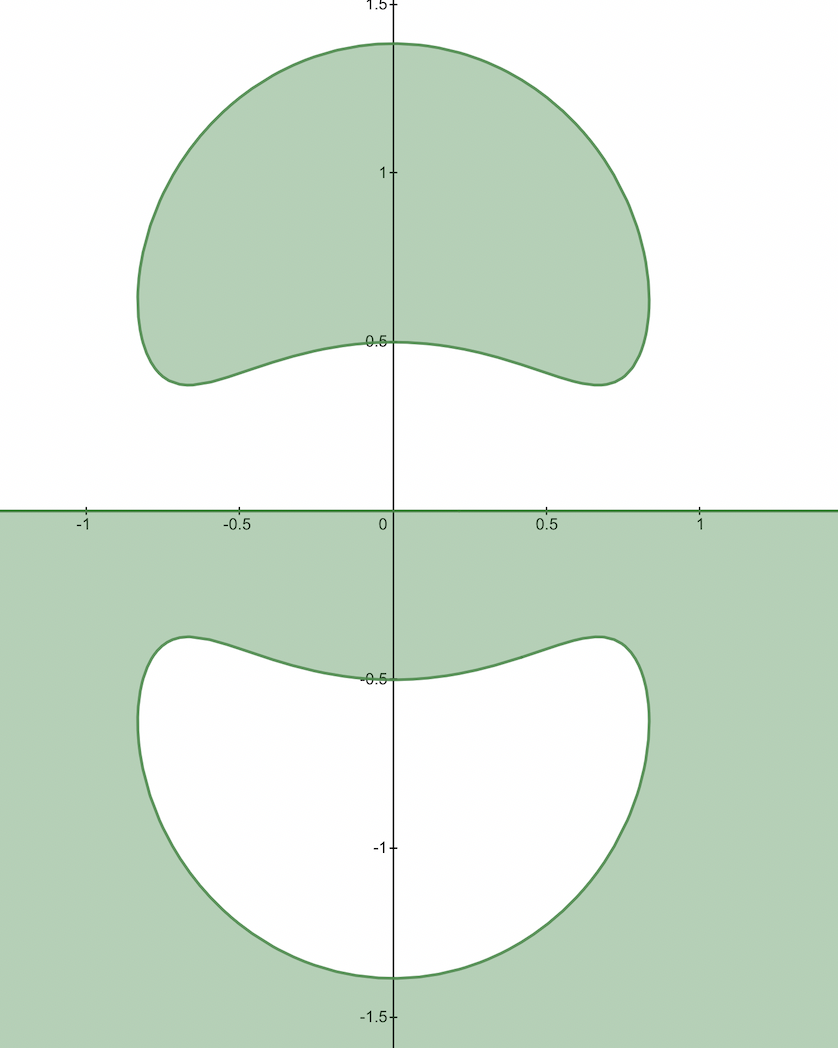}\hspace{0.5cm}
			\label{fig:desmos-graph}}\hspace{3mm}
		\subfigure[]{\includegraphics[width=4cm, height=4cm]{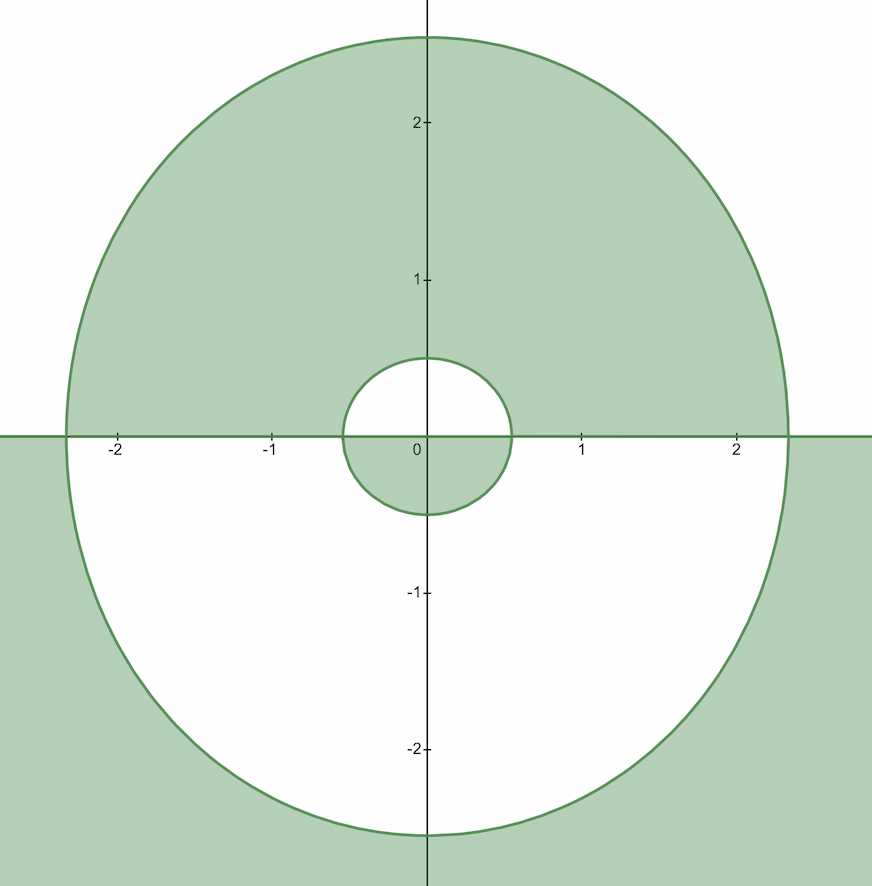}
			\label{fig:desmos-graph-5}}	

		\subfigure[]{\includegraphics[width=4cm,height=4cm]{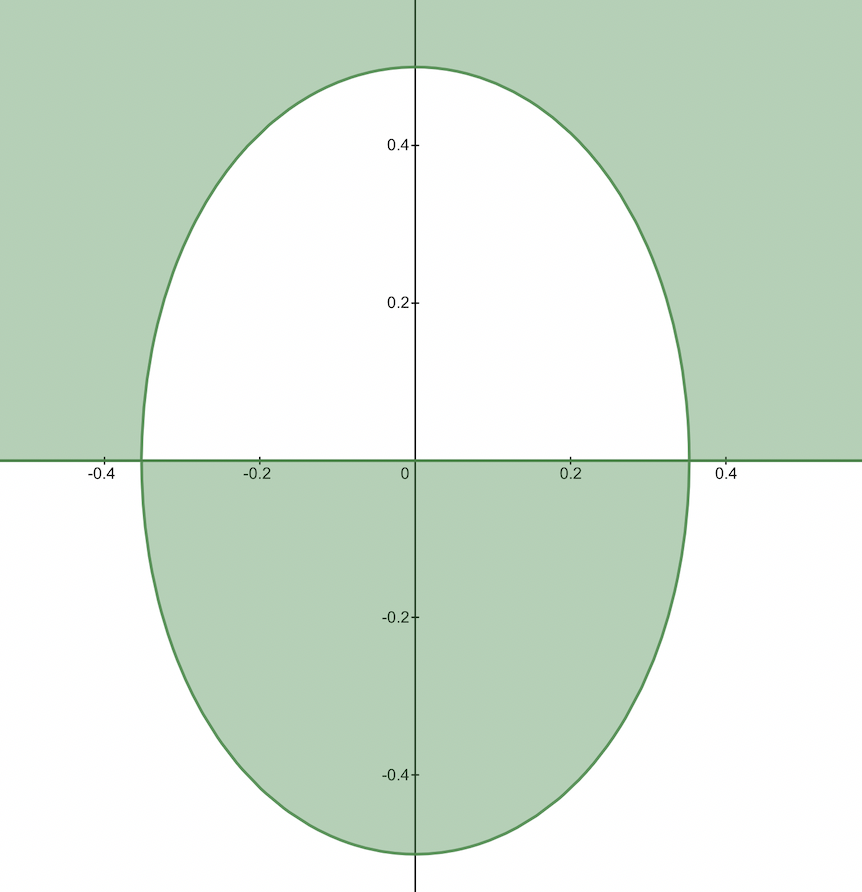}\hspace{0.5cm}
			\label{fig:desmos-graph-3}}\hspace{3mm}
		\subfigure[]{\includegraphics[width=4cm, height=4cm]{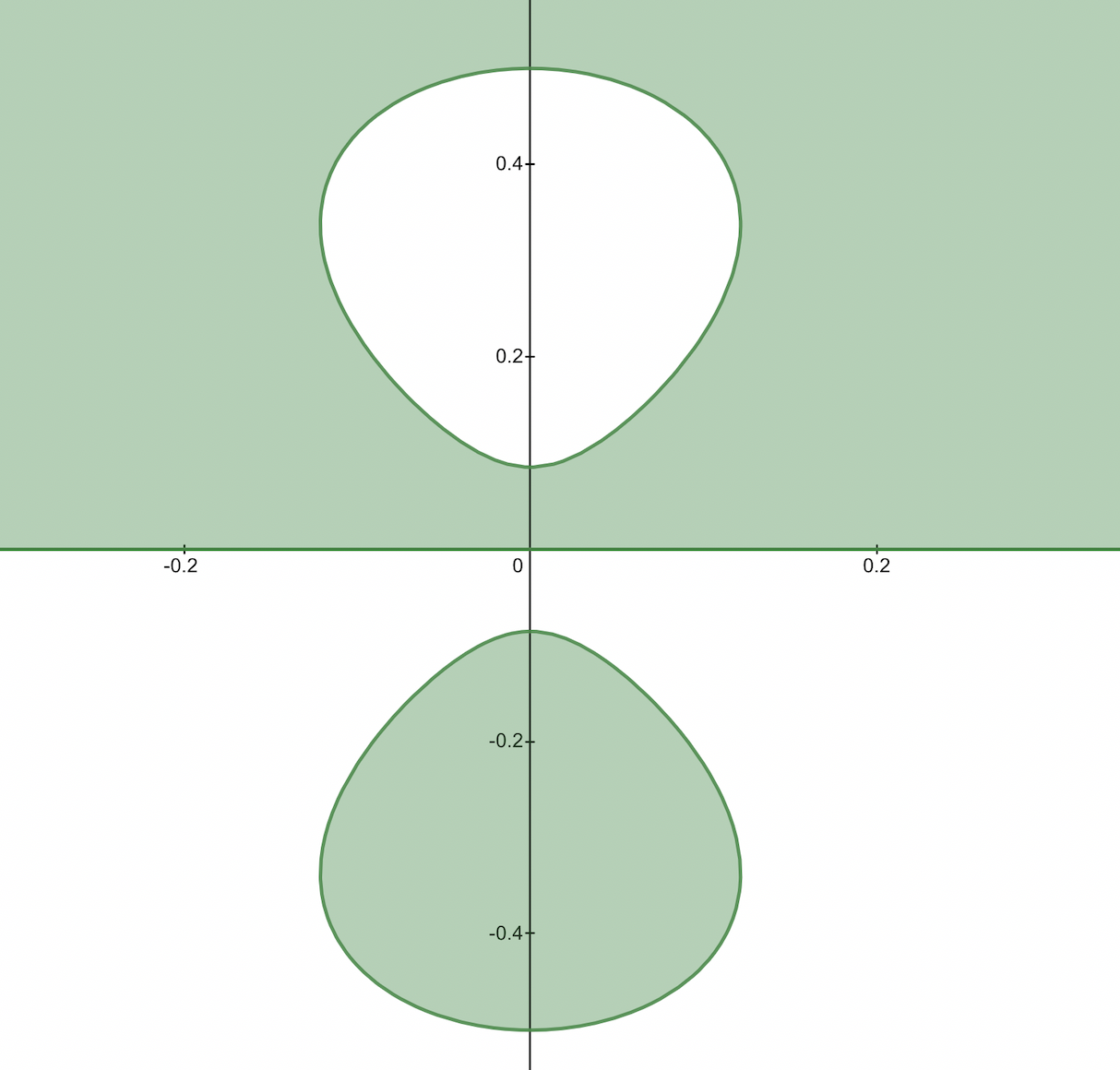}
			\label{fig:desmos-graph-6}}
		\caption{\footnotesize The classification  of $\text{sign Im}\theta$.  In the green regions, $\text{Im}\theta>0$,  which  implies that $|e^{2it\theta}|\to 0$ as $t\to\infty$.While  in the white regions,   $\text{Im}\theta<0$, which implies  $|e^{-2it\theta}|\to 0$ as $t\to\infty$.   The green curves  $\text{Im}\theta=0$ are critical lines between decay and growth regions.  }
		\label{figtheta}
	\end{figure}
	
	Notice that the division above omits two critical cases : $\xi=-\frac{1}{4}$ and $\xi=2$. For the sake of completeness of the discussion, we will place these two cases in Section 5.
	
	For the cases I and IV, there is no stationary phase point on the real axis,
	while for the cases II and III, there exist four and two stationary phase points on
	the real axis, denoted as $\xi_1> \xi_2>\xi_3 > \xi_4$ and $\xi_1>\xi_2$ respectively.

	The jump matrix  $V(k)$  admits the following two  factorizations
	\begin{align}
		V(k)&=\left(\begin{array}{cc}
			1 &  0 \\
			-\bar{r} e^{2it\theta} & 1
		\end{array}\right)\left(\begin{array}{cc}
			1 & r e^{-2it\theta}\\
			0 & 1
		\end{array}\right) \label{V11}\\
		&=\left(\begin{array}{cc}
			1 & \frac{r e^{-2it\theta}}{1-|r|^2}\\
			0& 1
		\end{array}\right)(1-|r|^2)^{-\sigma_3}\left(\begin{array}{cc}
			1 & 0 \\
			\frac{-\bar{r} e^{2it\theta}}{1-|r|^2} & 1
		\end{array}\right).\label{V2}
	\end{align}
	We will utilize  these factorizations to open  the jump contours  so that the oscillating factor $e^{\pm2it\theta}$ are decaying in corresponding region respectively.


	\section{Long-time asymptotics  in   region   without phase point}\label{sec3}
	\subsection{Deformation of the RH problem}\label{subsec3.1}
	One will observe that the long time asymptotic of RHP \ref{RHP2} is influenced by the growth and decay of the oscillatory term. In this subsection, our aim is to introduce a transform of $M(k;x,t)\rightarrow M^{(1)} (k;x,t)$ so that $M^{(1)} (k;x,t)$ is well behaved as $t\rightarrow \infty$ along the characteristic line.
	\subsubsection{Conjugation}
	For $\xi<-\frac{1}{4}$  and  $\xi>2$, we choose the decomposition (\ref{V11}) , (\ref{V2}) respectively. Introduce the notation that will be used later. Define
	\begin{align}
		&\Delta^+=\{ik_n|Im\theta(ik_n)\geq0,k_n \in \mathbb{R}^+\}\nonumber\\
		&\Delta^-=\{ik_n|Im\theta(ik_n)<0,k_n \in \mathbb{R}^+\}
	\end{align}
	and
	\begin{align}
		&\Lambda=\{ik_n||ik_n-i\kappa_1|<\rho\},\hspace{0.3cm} \{\kappa_1,-\kappa_1\}=i\mathbb{R}\cap\{k|Im\theta(k)=0\},\nonumber\\
		&\rho=\frac{1}{2}min\{min|ik_j-ik_l|,min|ik_j-\frac{i}{2}|,min|ik_j-i\kappa_1|\}.\nonumber
	\end{align}
	Denote
	\begin{align}
		L(\xi)=\left\{ \begin{array}{ll}
			\emptyset,   &\text{as } \xi<-1/4,\\[4pt]
			\mathbb{R} , &\text{as }\xi>2,\\
		\end{array}\right.
	\end{align}
where   $\emptyset$ means that  for any function $f(x)$,  we make a definition of engagement
	\begin{align*}
		\int_{\emptyset}f(x)dx=0.
	\end{align*}
	Define the function
	\begin{equation}
		T(k)=\prod_{ik_n\in\Delta^+}\frac{ k+ik_n}{k-ik_n}\delta(k,\xi),\label{T_1}
	\end{equation}
	where
	\begin{equation}
		\delta(k,\xi)=exp\{-i\int_{L(\xi)}\frac{\mu(s)}{s-k}ds\},\ \  \ \ \mu(k)=\frac{1}{2\pi}log(1-|r|^2). \nonumber
	\end{equation}
	
	\begin{Proposition}\label{proT_1}
		The function $T(k)$ defined by (\ref{T_1}) has the following properties:\\
		(a) $T(k)$ is meromorphic in $\mathbb{C}$, and for each $ik_n\in\Delta^+$, $T(k)$ has  simple poles at $ik_n $ and  simple zeros at $-ik_n $; $T(k)$ is analytic and nonzero elsewhere.\\
		(b) $\overline{T(\bar k)}=\frac{1}{T(k)}$;\\
		(c) $\lim_{k\to \infty}T(k)=1$; \\
		(d)   $T(k)$ has boundary values $T_\pm(k)$ satisfying:
		\begin{align}
			&T_+(k)=(1-|r(k)|^2)T_-(k),\hspace{0.5cm}\xi>2, \\
			&T_+(k)=T_-(k),\hspace{0.5cm}\xi<-\frac{1}{4};
		\end{align}
		(e) For $k=\frac{i}{2}$
		\begin{equation}
			T(\frac{i}{2})=\prod_{ik_n\in\Delta^+}\frac{ \frac{i}{2}+ik_n}{\frac{i}{2}-ik_n}\delta(\frac{i}{2},\xi);\nonumber
 		\end{equation}
		(f) As $k\to \frac{i}{2}$, $T(k)$ has asymptotic expansion  as
		\begin{align}
			T(k)&:=T(\frac{i}{2})+T_1(k-\frac{i}{2}) +\mathcal{O}((k-\frac{i}{2})^2)\label{expT0},
		\end{align}
		with
		\begin{align}
		T_1=\sum_{j,s=1;j\neq s}^{\mathcal{N}(\Delta^+)}\left(\frac{-2ik_{n_s}(ik_{n_j}+\frac{i}{2})}{(\frac{i}{2}-ik_{n_j})(\frac{i}{2}-ik_{n_s})^2}\right) J_0+\prod_{ik_n\in\Delta^+}\frac{ \frac{i}{2}+ik_n}{\frac{i}{2}-ik_n}J_1,
		\end{align}
	where
		\begin{equation}
			J_0=\exp\{-\frac{i}{2\pi}\int_{L(\xi)}\frac{ln(1-|r(s)|^2)}{s-\frac{i}{2}}ds\};
		\end{equation}
		\begin{equation}
			J_1=-\frac{i}{2\pi}\int_{L(\xi)}\frac{ln(1-|r(s)|^2)}{(s-\frac{i}{2})^2}ds\}.
		\end{equation}
		\begin{proof}
			Properties (a)-(e) are obtained by simple calculation from the definition of $T(k)$ in (\ref{T_1}). Property (f) is from the  Laurent expansion of $T(k)$ at $k=\frac{i}{2}$.
		\end{proof}
	\end{Proposition}
	
	We define a new unknown function $M^{(1)}(k)$ using the function defined above
	\begin{equation}
		M^{(1)}(k)=M^J(k)T^{\sigma_3}(k),\nonumber
	\end{equation}
	then $M^{(1)}$ solves the following RH problem:
	\begin{RHP}\label{RHP3}
		Find a matrix-valued function $ M^{(1)}(k) $ which satisfies:
		
		$\blacktriangleright$ Analyticity: $M^{(1)}(k)$ is analytic in $\mathbb{C}\setminus (\mathbb{R}\bigcup\mathcal{Z})$ and has single poles;
		
		$\blacktriangleright$ Jump condition: $M^{(1)}(k)$ has continuous boundary values $M^{(1)}_\pm(k)$ on $\mathbb{R}$ and
		\begin{equation}
			M_+^{(1)}(k)=M_-^{(1)}(k)V^{(1)}(k,\xi),\hspace{0.5cm}k \in \mathbb{R},
		\end{equation}
		where
		\begin{equation}
			V^{(1)}(k,\xi)=\left\{\begin{array}{ll}\left(\begin{array}{cc}
					1 & 0 \\
					-T^2\overline{r}e^{2it\theta}&  1
				\end{array}\right)\left(\begin{array}{cc}
					1 & T^{-2}re^{-2it\theta} \\
					0 &  1
				\end{array}\right),   &\text{as } \xi<-\frac{1}{4} , k \in \mathbb{R},\\[12pt]
				\left(\begin{array}{cc}
					1 & \frac{r(k)e^{-2it\theta}T_-^{-2}(k)}{1-|r(k)|^2}\\
					0 & 1
				\end{array}\right)\left(\begin{array}{cc}
					1 & 0 \\
					\frac{-\bar{r}(k)e^{2it\theta}T_+^{2}(k)}{1-|r(k)|^2} & 1
				\end{array}\right),   &\text{as } \xi>2 , k \in \mathbb{R};\\[12pt]
			\end{array}\right.\label{jumpv1}
		\end{equation}
		
		$\blacktriangleright$ Asymptotic behaviors:
		\begin{align}
			M^{(1)}(k) = I+\mathcal{O}(k^{-1}),\hspace{0.5cm}k \rightarrow \infty;
		\end{align}
		$\blacktriangleright$ Residue conditions: $M^{(1)}(k)$ has simple poles at each point in $ \mathcal{Z} $ with:
		\begin{align}
			\res_{k=ik_n}M^{(1)}(k)=\left\{\begin{array}{ll}\lim_{k\rightarrow ik_n}M^{(1)}(k)\left(\begin{array}{cc}
					0 & 0\\
					c_n^{-1} e^{2it\theta{(ik_n)}}[(\frac{1}{T})^{'}(ik_n)]^{-2}& 0
				\end{array}\right),   &\text{as } ik_n\in 	\Delta^+,\\[12pt]
				\lim_{k\rightarrow ik_n}M^{(1)}(k)\left(\begin{array}{cc}
					0 & c_nT^{-2}e^{-2it\theta{(ik_n)}}\\
					0 & 0
				\end{array}\right),   &\text{as } ik_n\in 	\Delta^-\bigcup\Lambda;\\[12pt]
			\end{array}
			\right. \nonumber
		\end{align}
								\begin{align}
									\res_{k=-ik_n}M^{(1)}(k)=\left\{\begin{array}{ll}\lim_{k\rightarrow -ik_n}M^{(1)}(k)\left(\begin{array}{cc}
											0 &  	\bar{c}_n^{-1} e^{-2it\theta{(-ik_n)}}[T^{'}(-ik_n)]^{-2}\\
											0 & 0
										\end{array}\right),   &\text{as } ik_n\in 	\Delta^+,\\[12pt]
										\lim_{k\rightarrow -ik_n}M^{(1)}(k)\left(\begin{array}{cc}
											0 & 0\\
											\bar{c}_nT^{2}e^{2it\theta{(-ik_n)}} & 0
										\end{array}\right),   &\text{as } ik_n\in 	\Delta^-\bigcup\Lambda.\\[12pt]
									\end{array}
									\right.\nonumber
								\end{align}
							\end{RHP}

							\subsubsection{A mixed $\bar{\partial}$-RH problem}
							The next step is to introduce factorizations of the jump matrix whose factors admit continuous-but
							not necessarily analytic-extensions off the real axis. The price we pay for this non-analytic transformation is that the new
							unknown matrix has nonzero $\bar{\partial}$-derivatives inside the regions in which the extensions are introduced and satisfies a hybrid
							$\bar{\partial}$-RH problem.
							
							Define the domain and contour as follows:
							\begin{align}
								&\Omega_{2n+1}=\left\lbrace k\in\mathbb{C}|n\pi \leq\arg k \leq n\pi+\varphi \right\rbrace,\nonumber\\
								&\Omega_{2n+2}=\left\lbrace k\in\mathbb{C}|(n+1)\pi -\varphi\leq\arg k \leq (n+1)\pi \right\rbrace,\nonumber
							\end{align}
							where $n=0,1$. And
							\begin{align}
								&\Sigma_k=e^{(k-1)i\pi/2+i\varphi}R_+,\hspace{0.5cm}k=1,3,\nonumber\\
								&\Sigma_k=e^{ki\pi/2-i\varphi}R_+,\hspace{0.5cm}k=2,4,\nonumber
							\end{align}
							which is the boundary of $\Omega_k$ respectively.
							Take  $\varphi $  be  a sufficiently small fixed angle achieving following conditions: \\
							(i) each $\Omega_i$ doesn't intersect $\left\lbrace k\in\mathbb{C}|\text{Im }\theta(k)=0\right\rbrace  $,\\
							(ii) $\cos2\varphi >\frac{(-1+\sqrt{1-\frac{2}{\xi}})^2-2}{2}$ for $\xi<-\frac{1}{4}$;\\
							(iii) $\frac{(1-\sqrt{1-\frac{2}{\xi}})^2-2}{2}<\cos2\varphi<\frac{(1+\sqrt{1-\frac{2}{\xi}})^2-2}{2}$ for $\xi>2$.
							
							A more intuitive expression of the above process can be seen in the figure \ref{figR2}.
							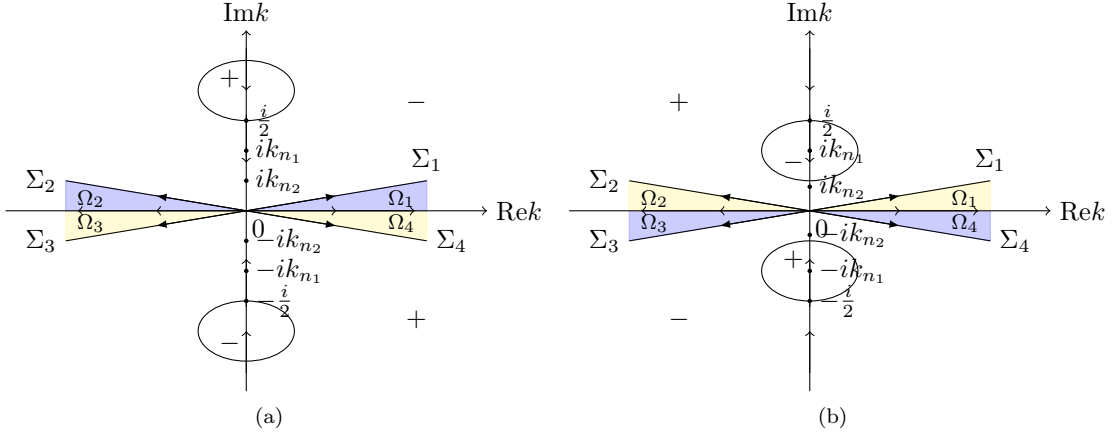
\begin{figure}[h]
								\centering
								\subfigure[]{
									\begin{tikzpicture}[node distance=2cm,  scale = 0.8]
										\draw[yellow!30, fill=yellow!20] (0,0)--(3,-0.5)--(3,0)--(0,0)--(-3,-0.5)--(-3,0)--(0,0);
										\draw[blue!30, fill=blue!20] (0,0)--(3,0.5)--(3,0)--(0,0)--(-3, 0.5)--(-3,0)--(0,0);
										\draw(0,0)--(3,0.5)node[above]{$\Sigma_1$};
										\draw(0,0)--(-3,0.5)node[left]{$\Sigma_2$};
										\draw(0,0)--(-3,-0.5)node[left]{$\Sigma_3$};
										\draw(0,0)--(3,-0.5)node[right]{$\Sigma_4$};
										\draw[->](-4,0)--(4,0)node[right]{ Re$k$};
										\draw[->](0,-3)--(0,3)node[above]{ Im$k$};
										\draw[-latex](0,0)--(-1.5,-0.25);
										\draw[-latex](0,0)--(-1.5,0.25);
										\draw[-latex](0,0)--(1.5,0.25);
										\draw[-latex](0,0)--(1.5,-0.25);
										\coordinate (C) at (-0.2,2.2);
										\coordinate (D) at (2.2,0.2);
										\fill (D) circle (0pt) node[right] {\footnotesize $\Omega_1$};
										\coordinate (J) at (-2.2,-0.2);
										\fill (J) circle (0pt) node[left] {\footnotesize $\Omega_3$};
										\coordinate (k) at (-2.2,0.2);
										\fill (k) circle (0pt) node[left] {\footnotesize $\Omega_2$};
										\coordinate (k) at (2.2,-0.2);
										\fill (k) circle (0pt) node[right] {\footnotesize $\Omega_4$};
										\coordinate (I) at (0.2,0);
										\fill (I) circle (0pt) node[below] {$0$};
										\draw[  ][->](0,0)--(-1.5,0);
										\draw[ ][->](-1.5,0)--(-2.8,0);
										\draw[ ][->](0,0)--(1.5,0);
										\draw[ ][->](1.5,0)--(2.8,0);
										\draw[ ][->](0,2.7)--(0,2);
										\draw[ ][->](0,1.6)--(0,0.8);
										\draw[ ][->](0,-2.7)--(0,-2);
										\draw[ ][->](0,-1.6)--(0,-0.8);
										\draw[ ](0,2) ellipse (0.8 and 0.5);
										\draw[ ](0,-2) ellipse (0.8 and 0.5);
										\coordinate (A) at (-0.6,2.2);
										\fill (A) circle (0pt) node[right] {$+$};
										\coordinate (B) at (-0.6,-2.2);
										\fill (B) circle (0pt) node[right] {$-$};
										\coordinate (C) at (2.5,1.8);
										\fill (C) circle (0pt) node[right] {$-$};
										\coordinate (D) at (2.5,-1.8);
										\fill (D) circle (0pt) node[right] {$+$};
										\coordinate (E) at (0,1.5);
										\fill (E) circle (1pt) ;
										\coordinate (F) at (0,-1.5);
										\fill (F) circle (1pt) ;
										\fill (E) circle (1pt)node[right] {$\frac{i}{2}$};
										\fill (F) circle (1pt)node[right] {$-\frac{i}{2}$};
										\coordinate (G) at (0,-1);
										\fill (G) circle (1pt)node[right] {$-ik_{n_1}$}; 
										\coordinate (H) at (0,1);
										\fill (H) circle (1pt) node[right] {$ik_{n_1}$}; 
										\coordinate (I) at (0,1);
										\fill (I) circle (1pt) ;
										\coordinate (J) at (0,-1);
										\fill (J) circle (1pt) ;
										\coordinate (K) at (0,-0.5);
										\fill (K) circle (1pt) ;
										\coordinate (L) at (0,0.5);
										\fill (L) circle (1pt) ;
										\fill (K) circle (1pt)node[right] {$-ik_{n_2}$}; 
										\fill (L) circle (1pt) node[right] {$ik_{n_2}$}; 

								\end{tikzpicture}}
								\subfigure[]{
									\begin{tikzpicture}[node distance=2cm,  scale = 0.8]
										\draw[blue!30, fill=blue!20] (0,0)--(3,-0.5)--(3,0)--(0,0)--(-3,-0.5)--(-3,0)--(0,0);
										\draw[yellow!30, fill=yellow!20] (0,0)--(3,0.5)--(3,0)--(0,0)--(-3, 0.5)--(-3,0)--(0,0);
										\draw(0,0)--(3,0.5)node[above]{$\Sigma_1$};
										\draw(0,0)--(-3,0.5)node[left]{$\Sigma_2$};
										\draw(0,0)--(-3,-0.5)node[left]{$\Sigma_3$};
										\draw(0,0)--(3,-0.5)node[right]{$\Sigma_4$};
										\draw[->](-4,0)--(4,0)node[right]{ Re$k$};
										\draw[->](0,-3)--(0,3)node[above]{ Im$k$};
										\draw[-latex](0,0)--(-1.5,-0.25);
										\draw[-latex](0,0)--(-1.5,0.25);
										\draw[-latex](0,0)--(1.5,0.25);
										\draw[-latex](0,0)--(1.5,-0.25);
										\coordinate (C) at (-0.2,2.2);
										\coordinate (D) at (2.2,0.2);
										\fill (D) circle (0pt) node[right] {\footnotesize $\Omega_1$};
										\coordinate (J) at (-2.2,-0.2);
										\fill (J) circle (0pt) node[left] {\footnotesize $\Omega_3$};
										\coordinate (k) at (-2.2,0.2);
										\fill (k) circle (0pt) node[left] {\footnotesize $\Omega_2$};
										\coordinate (k) at (2.2,-0.2);
										\fill (k) circle (0pt) node[right] {\footnotesize $\Omega_4$};
										\coordinate (I) at (0.2,0);
										\fill (I) circle (0pt) node[below] {$0$};
										\draw[  ][->](0,0)--(-1.5,0);
										\draw[ ][->](-1.5,0)--(-2.8,0);
										\draw[ ][->](0,0)--(1.5,0);
										\draw[ ][->](1.5,0)--(2.8,0);
										\draw[ ][->](0,2.7)--(0,2);
										\draw[ ][->](0,1.6)--(0,0.8);
										\draw[ ][->](0,-2.7)--(0,-2);
										\draw[ ][->](0,-1.6)--(0,-0.8);
										\draw[ ](0,1) ellipse (0.8 and 0.5);
										\draw[ ](0,-1) ellipse (0.8 and 0.5);
										\coordinate (A) at (-0.6,0.8);
										\fill (A) circle (0pt) node[right] {$-$};
										\coordinate (B) at (-0.6,-0.8);
										\fill (B) circle (0pt) node[right] {$+$};
										\coordinate (C) at (-2.5,1.8);
										\fill (C) circle (0pt) node[right] {$+$};
										\coordinate (D) at (-2.5,-1.8);
										\fill (D) circle (0pt) node[right] {$-$};
										\coordinate (E) at (0,1.5);
										\fill (E) circle (1pt) ;
										\coordinate (F) at (0,-1.5);
										\fill (F) circle (1pt) ;
										\fill (E) circle (1pt)node[right] {$\frac{i}{2}$};
										\fill (F) circle (1pt)node[right] {$-\frac{i}{2}$};
										\coordinate (G) at (0,-1);
										\fill (G) circle (1pt)node[right] {$-ik_{n_1}$}; 
										\coordinate (H) at (0,1);
										\fill (H) circle (1pt) node[right] {$ik_{n_1}$}; 
										\coordinate (I) at (0,1);
										\fill (I) circle (1pt) ;
										\coordinate (J) at (0,-1);
										\fill (J) circle (1pt) ;
										\coordinate (K) at (0,-0.4);
										\fill (K) circle (1pt) ;
										\coordinate (L) at (0,0.4);
										\fill (L) circle (1pt) ;
										\fill (K) circle (1pt)node[right] {$-ik_{n_2}$}; 
										\fill (L) circle (1pt) node[right] {$ik_{n_2}$}; 

								\end{tikzpicture}}
								\caption{\footnotesize  Figures  (a) and (b)  corresponding to  the cases $\xi<-1/4$ and  $\xi>2$ respectively.   The blue regions  have same decay/grow  properties; The same for yellow  regions.}
								\label{figR2}
							\end{figure}

							Define a new unknown function
							\begin{equation}
								R^{(2)}(k,\xi)=\left\{\begin{array}{lll}
									\left(\begin{array}{cc}
										1 & R_j(k,\xi)e^{-2it\theta}\\
										0 & 1
									\end{array}\right), & \xi<-\frac{1}{4},k\in \Omega_j,j=1,2;\\
									\\
									\left(\begin{array}{cc}
										1 & 0\\
										R_j(k,\xi)e^{2it\theta} & 1
									\end{array}\right),  &\xi<-\frac{1}{4},k\in \Omega_j,j=3,4;\\
									\\
									\left(\begin{array}{cc}
										1 & 0\\
										R_j(k,\xi)e^{2it\theta} & 1
									\end{array}\right), & \xi>2,k\in \Omega_j,j=1,2;\\
									\\
									\left(\begin{array}{cc}
										1 & R_j(k,\xi)e^{-2it\theta}\\
										0 & 1
									\end{array}\right),  &\xi>2,k\in \Omega_j,j=3,4;\\
									\\
									I,  &elsewhere,\\
								\end{array}\right.\label{R(2)+}
							\end{equation}
							where the $R_j$ are given in the following Proposition.
							\begin{Proposition}\label{proR}
								The functions $R_j$: $\bar{\Omega}_j\to C$, $j=1,2,3,4$ have boundary values as follows:\\
								For $\xi \in(-\infty,-1/4)$,
								\begin{align}
									&R_j(k,\xi)=\Bigg\{\begin{array}{ll}
										-r(k)T^{-2}(k) & k\in \mathbb{R}^i,\\
										-r(0)T^{-2}(0)  &k\in \Sigma_j,\\
									\end{array} \hspace{0.6cm}j=1,2 \nonumber \\
									&R_j(k,\xi)=\Bigg\{\begin{array}{ll}
										-\bar{r}(k)T^{2} (k)&k\in \mathbb{R}^i, \\
										-\bar{r}(0)T^{2} (0) &k\in \Sigma_j,\\
									\end{array} \hspace{0.6cm}j=3,4.\nonumber
								\end{align}
								For $\xi \in(2,+\infty)$,
								\begin{align}
									&R_j(k,\xi)=\Bigg\{\begin{array}{ll}
										\frac{\bar{r}}{1-|r(k)|^2}T_+^{2} & k\in \mathbb{R}^i,\\
										\bar{r}(0)exp\{\frac{1}{\pi i}\int_{\mathbb{R}}\frac{\log{(1-|r|^2)}}{s}ds\}  &k\in \Sigma_j,\\
									\end{array} \hspace{0.6cm} j=1,2 \nonumber\\
									&R_j(k,\xi)=\Bigg\{\begin{array}{ll}
										\frac{r}{1-|r(k)|^2}T_-^{-2} &k\in \mathbb{R}^i, \\
										-r(0)exp\{\frac{1}{\pi i}\int_{\mathbb{R}}\frac{\log{(1-|r|^2)}}{s}ds\} &k\in \Sigma_j,\\
									\end{array} \hspace{0.6cm}j=3,4. \nonumber
								\end{align}	
								where 	\begin{align}
									i=\Bigg\{\begin{array}{ll}
										+, & j=1,4,\\
										-,  &j=2,3.\\
									\end{array} \nonumber
								\end{align}	
								$R_j$  have the following properties:
								for $j=1,2,3,4,$
								\begin{align}
									&|\bar{\partial}R_j(k)|\lesssim|r'(|k|)|+|k|^{-1/2}, \  \text{for all $k\in \Omega_j$, j=1,2,}\label{dbarRj}\\
									&|\bar{\partial}R_j(k)|\lesssim|\bar{r}{'}(|k|)|+|k|^{-1/2}, \  \text{for all $k\in \Omega_j$, j=3,4}\label{dbarRj2}\\
									&\bar{\partial}R_j(k)=0,\hspace{0.5cm}\text{if } k\in elsewhere. \label{dbarRj0}
								\end{align}
								\begin{proof}
									The proof is similar with \cite{YF3}. Here, we only take $R_1(k)$ of Case IV as an example. By Plemelj formula, the boundary value on $\mathbb{R}^+ $ of $R_1(k)$ can be rewritten as
									\begin{align}
										\tilde{r}(k):=\frac{\bar{r}(k)}{1-|r(k)|^2}T_+^{2} =\prod_{ik_n\in\Delta^+}\left(\frac{ k+ik_n}{k-ik_n}\right)^2\bar{r}(k)exp\{\frac{1}{\pi i}\int_{\mathbb{R}}\frac{\log{(1-|r|^2)}}{s}ds\}.\nonumber
									\end{align}	
									Construct a continuous function $R_1(k)$ on $\Omega_1$:
									\begin{align}
										R_1(k)=\left(\tilde{r}(|k|)-\tilde{r}(0)\right)cos(\frac{\pi}{2\varphi}\phi)+\tilde{r}(0).
									\end{align}	
									Denote $k=\rho e^{i\phi}$, then $\bar{\partial}=\frac{e^{i \phi}}{2}\left(\partial_{\rho}+i \rho^{-1} \partial_\phi\right)$. So
									\begin{align}
										\bar{\partial}R_1(k)&=\frac{1}{2}\overline{r^{'}(|k|)}e^{i\phi}\prod_{ik_n\in\Delta^+}\left(\frac{ k+ik_n}{k-ik_n}\right)^2exp\{\frac{1}{\pi i}\int_{\mathbb{R}}\frac{\log{(1-|r|^2)}}{s}ds\}cos(\frac{\pi}{2\varphi}\phi)\nonumber\\
										&\quad-\frac{1}{2}e^{i\phi}\rho^{-1}\left(\tilde{r}(|k|)-\tilde{r}(0)\right)sin\left(\frac{\pi}{2\varphi}\phi\right)\frac{\pi}{2\varphi}.
									\end{align}	
									By Holder inequality, we bound the second term and then get the estimates.
								\end{proof}
							\end{Proposition}
							
							Make the matrix transformation:
							\begin{align}
								M^{(2)}(k)=M^{(1)}(k)R^{(2)}(k),\nonumber
							\end{align}
							then $M^{(2)}$ solves the mixed $\bar{\partial}$-RH problem:
							
							\begin{RHP}\label{RHP4}
								Find a matrix valued function  $ M^{(2)}(k)$ with following properties:
								
								$\blacktriangleright$ Analyticity:  $M^{(2)}(k)$ is continuous in $\mathbb{C}\backslash(\Sigma^{(2)}\cup\mathcal{Z})$, where $\Sigma^{(2)}=\cup_{j=1}^4\Sigma_j$;
								
								$\blacktriangleright$ Jump condition: $M^{(2)}(k)$ has continuous boundary values $M^{(2)}_\pm(k)$ on $\Sigma^{(2)}$ and
								\begin{equation}
									M^{(2)}_{+}(k)=M^{(2)}_{-}(k)V^{(2)}(k),\hspace{0.5cm}k \in \Sigma^{(2)},
								\end{equation}
								where
								\begin{align}
									V^{(2)}(k,\xi)=\Bigg\{\begin{array}{ll}
										R^{(2)}(k)^{-1}|_{\Sigma_{1}\cup\Sigma_{3}}&\text{as } 	k\in\Sigma_{1}\cup\Sigma_{3};\\[12pt]
										R^{(2)}(k)|_{\Sigma_{2}\cup\Sigma_{4}}&\text{as } 	k\in\Sigma_{2}\cup\Sigma_{4};\\[12pt]
									\end{array}\label{jumpv2}
								\end{align}
								
								$\blacktriangleright$ Asymptotic behaviors:
								\begin{align}
									M^{(2)}(k) = I+\mathcal{O}(k^{-1}),\hspace{0.5cm}k \rightarrow \infty;
								\end{align}
								
								$\blacktriangleright$ $\bar{\partial}$-Derivative: For $k\in\mathbb{C}$
								we have
								\begin{align}
									\bar{\partial}M^{(2)}(k)=M^{(2)}(k)\bar{\partial}R^{(2)}(k),
								\end{align}
								where
								\begin{equation}
									\bar{\partial}R^{(2)}(k,\xi)=\left\{\begin{array}{lll}
										\left(\begin{array}{cc}
											0 & \bar{\partial}R_j(k,\xi)e^{-2it\theta}\\
											0 & 0
										\end{array}\right), & \xi<-\frac{1}{4},k\in \Omega_j,j=1,2;\\
										\\
										\left(\begin{array}{cc}
											0 & 0\\
											\bar{\partial}R_j(k,\xi)e^{2it\theta} & 0
										\end{array}\right),  &\xi<-\frac{1}{4},k\in \Omega_j,j=3,4;\\
										\\
										\left(\begin{array}{cc}
											0 & 0\\
											\bar{\partial}R_j(k,\xi)e^{2it\theta} & 0
										\end{array}\right), & \xi>2,k\in \Omega_j,j=1,2;\\
										\\
										\left(\begin{array}{cc}
											0 & \bar{\partial}R_j(k,\xi)e^{-2it\theta}\\
											0 & 0
										\end{array}\right),  &\xi>2,k\in \Omega_j,j=3,4;\\
										\\
										0,  &elsewhere;\\
									\end{array}\right.\label{DBARR1}
								\end{equation}
								
								$\blacktriangleright$ Residue conditions: $M^{(2)}$(k) has the same residue conditions with the RHP \ref{RHP3}.
								
							\end{RHP}
							
							In order to solve the RHP \ref{RHP4}, we decompose it into a pure RH problem  $ M^{(R)}(k)$  with $\bar\partial R^{(2)}\equiv0$
							and a pure $\bar{\partial}$ problem $ M^{(3)}(k)$ with $\bar\partial R^{(2)}\not=0$, that is,
							\begin{equation}
								M^{(2)}(k)=M^{(3)}(k) M^{(R)}(k),\label{M2}
							\end{equation}
							where  the  matrix function $M^{(R)}$ satisfies
							\begin{RHP}\label{RHP5}
								Find a matrix valued function  $ M^{(R)}(k)$ admits same condition as RHP \ref{RHP4}.
								
								$\blacktriangleright$ Analyticity:  $M^{(R)}(k)$ is meromorphic in $\mathbb{C} \backslash\Sigma^{(2)}$;
								
								$\blacktriangleright$ Jump condition: $M^{(R)}(k)$ has the same jump matrix
								with $M^{(2)}(k)$ ;
								
								$\blacktriangleright$ Asymptotic behaviors:
								\begin{align}
									M^{(R)}(k) =I+\mathcal{O}(k^{-1}),\hspace{0.5cm}k \rightarrow \infty;
								\end{align}
								
								$\blacktriangleright$ Residue conditions: $M^{(R)}$(k) has the same residue conditions with the RHP \ref{RHP3}.
								
							\end{RHP}

							\subsection{Analysis on a pure RH problem}
							\subsubsection{Classify the poles}
							
							In this section, our aim is to convert the poles away from $\left\lbrace k\in\mathbb{C}|\text{Im }\theta(k)=0\right\rbrace  $ into jumps.
							
					Define
							\begin{equation}
								G(k)=\left\{ \begin{array}{ll}
									\left(\begin{array}{cc}
										1 & 0\\
										-c_n^{-1}T(k)e^{2it\theta(ik_n)}[(\frac{1}{T})^{'}(ik_n)]^{-1} & 1
									\end{array}\right),   &\text{as } 	k\in\partial \mathcal{D}_n,ik_n\in\Delta^+;\\[12pt]
									\left(\begin{array}{cc}
										1 & c_nT^{-2}e^{-2it\theta(ik_n)}(k-ik_n)^{-1}\\
										0 & 1
									\end{array}\right),   &\text{as } k\in\partial \mathcal{D}_n,ik_n\in\Delta^-;\\[12pt]
									\left(\begin{array}{cc}
										1 & -\bar{c}_n^{-1}T^{-1}e^{-2it\theta(-ik_n)}[T^{'}(-ik_n)]^{-1}\\
										0 & 1
									\end{array}\right),   &\text{as } 	k\in\partial \bar{\mathcal{D}}_n,ik_n\in\Delta^+;\\[12pt]
									\left(\begin{array}{cc}
										1 & 0\\
										\bar{c}_nT^2e^{2it\theta(-ik_n)}(k+ik_n) ^{-1}& 1
									\end{array}\right),   &\text{as } k\in\partial \bar{\mathcal{D}}_n,ik_n\in\Delta^-;\\[12pt]
									I,   &\text{elsewhere }.
												\end{array}\right.\label{jumpG}
											\end{equation}
										and make a  transformation
							\begin{equation}
								M_R(k)=M^{(R)}(k)G(k),
							\end{equation}
then  $M_R(k)$  satisfies the following RH problem:
											
											\begin{RHP}\label{RHP6}
												Find a matrix valued function  $ M_R(k)$ with following properties:
												
												$\blacktriangleright$ Analyticity:  $M_R(k)$ is meromorphic in $\mathbb{C}\backslash\Sigma_R$;
												
												$\blacktriangleright$ Jump condition:
												\begin{equation}
													M_{+,R}(k)=M_{-,R}(k)V_R(k),  \hspace{0.5cm}k \in \Sigma_R;\nonumber
												\end{equation}
												where
												\begin{align}
													V_R(k)=\Bigg\{\begin{array}{ll}
														V^{(2)}(k),&\text{as } 	k\in\Sigma^{(2)};\\[12pt]
														G(k),&\text{as } 	k\in\bigcup_{\{n|ik_n\in (\mathbb{C}^+\cap\mathcal{Z})\backslash\Lambda\}}\left(\partial\mathcal{D}_n\cup\partial\bar{\mathcal{D}}_n\right); \nonumber
													\end{array}
												\end{align}
												The $\Sigma_R$ is shown in Figure \ref{figVR}.
												
												$\blacktriangleright$ Asymptotic behaviors:
												\begin{align}
													M_{R}(k) =& I+\mathcal{O}(k^{-1}),\hspace{0.5cm}k \rightarrow \infty; \nonumber
												\end{align}
												
												$\blacktriangleright$ Residue conditions:
												\begin{align}
													&\res_{k=ik_n}M_R(k)=\lim_{k\rightarrow ik_n}M_R(k)\left(\begin{array}{cc}
														0 & c_nT^{-2}e^{-2it\theta{(ik_n)}}\\
														0 & 0
													\end{array}\right),  \label{resMR} &\text{as } \hspace{0.2cm}ik_n\in \Lambda;
												\end{align}
												\begin{align}
													&\res_{k=-ik_n}M_R(k)=\lim_{k\rightarrow -ik_n}M_R(k)\left(\begin{array}{cc}
														0 & 0\\
														\bar{c}_nT^{2}e^{2it\theta{(-ik_n)}} & 0
													\end{array}\right),   &\text{as } \hspace{0.2cm} ik_n\in \Lambda;
												\end{align}
											\end{RHP}
											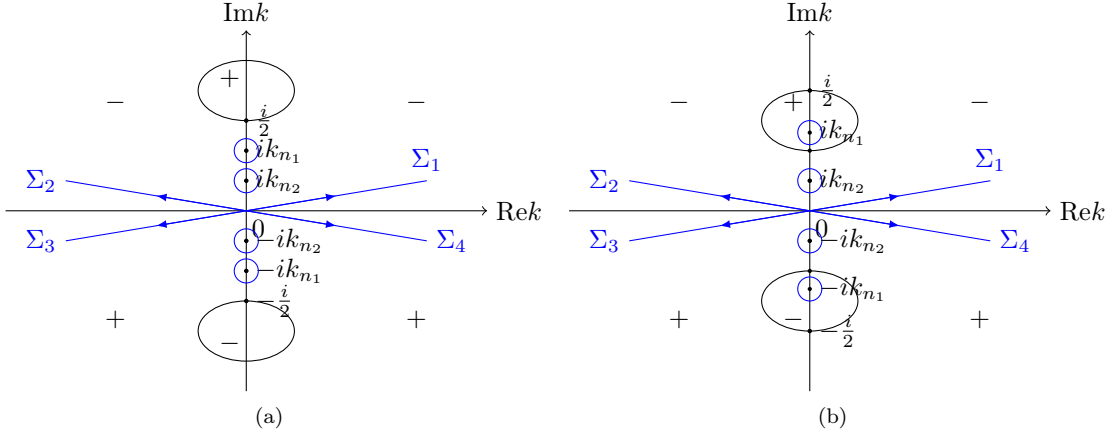
\begin{figure}[h]
												\centering
													\subfigure[]{
												\begin{tikzpicture}[node distance=2cm,  scale = 0.8]
													\draw[->](-4,0)--(4,0)node[right]{ Re$k$};
													\draw[->](0,-3)--(0,3)node[above]{ Im$k$};
													
													\coordinate (I) at (0.2,0);
													\fill (I) circle (0pt) node[below] {$0$};
													\draw[blue](0,0)--(3,0.5)node[above]{$\Sigma_1$};
													\draw[blue](0,0)--(-3,0.5)node[left]{$\Sigma_2$};
													\draw[blue](0,0)--(-3,-0.5)node[left]{$\Sigma_3$};
													\draw[blue](0,0)--(3,-0.5)node[right]{$\Sigma_4$};
													\draw[blue][-latex](0,0)--(-1.5,-0.25);
													\draw[blue][-latex](0,0)--(-1.5,0.25);
													\draw[blue][-latex](0,0)--(1.5,0.25);
													\draw[blue][-latex](0,0)--(1.5,-0.25);
													\draw[ ](0,2) ellipse (0.8 and 0.5);
													\draw[ ](0,-2) ellipse (0.8 and 0.5);
													\coordinate (A) at (-0.6,2.2);
													\fill (A) circle (0pt) node[right] {$+$};
													\coordinate (A1) at (-2.5,-1.8);
													\fill (A1) circle (0pt) node[right] {$+$};
													\coordinate (B) at (-0.6,-2.2);
													\fill (B) circle (0pt) node[right] {$-$};
													\coordinate (C) at (2.5,1.8);
													\fill (C) circle (0pt) node[right] {$-$};
													\coordinate (C1) at (-2.5,1.8);
													\fill (C1) circle (0pt) node[right] {$-$};
													\coordinate (D) at (2.5,-1.8);
													\fill (D) circle (0pt) node[right] {$+$};
													\coordinate (E) at (0,1.5);
													\fill (E) circle (1pt) ;
													\coordinate (F) at (0,-1.5);
													\fill (F) circle (1pt) ;
													\fill (E) circle (1pt)node[right] {$\frac{i}{2}$};
													\fill (F) circle (1pt)node[right] {$-\frac{i}{2}$};
													\coordinate (G) at (0,-1);
													\fill (G) circle (1pt)node[right] {$-ik_{n_1}$}; 
													\coordinate (H) at (0,1);
													\fill (H) circle (1pt) node[right] {$ik_{n_1}$}; 
													\coordinate (I) at (0,1);
													\fill (I) circle (1pt) ;
													\coordinate (J) at (0,-1);
													\fill (J) circle (1pt) ;
													\coordinate (K) at (0,-0.5);
													\fill (K) circle (1pt) ;
													\coordinate (L) at (0,0.5);
													\fill (L) circle (1pt) ;
													\draw[blue] (0,1) circle (0.2);
													\draw[blue] (0,-1) circle (0.2);
													\draw[blue] (0,0.5) circle (0.2);
													\draw[blue] (0,-0.5) circle (0.2);
													\fill (K) circle (1pt)node[right] {$-ik_{n_2}$};
													\fill (L) circle (1pt) node[right] {$ik_{n_2}$};
												\end{tikzpicture}}
											\centering
											\subfigure[]{
												\begin{tikzpicture}[node distance=2cm,  scale = 0.8]
													\draw[->](-4,0)--(4,0)node[right]{ Re$k$};
												\draw[->](0,-3)--(0,3)node[above]{ Im$k$};
												\coordinate (I) at (0.2,0);
												\fill (I) circle (0pt) node[below] {$0$};
												\draw[blue](0,0)--(3,0.5)node[above]{$\Sigma_1$};
												\draw[blue](0,0)--(-3,0.5)node[left]{$\Sigma_2$};
												\draw[blue](0,0)--(-3,-0.5)node[left]{$\Sigma_3$};
												\draw[blue](0,0)--(3,-0.5)node[right]{$\Sigma_4$};
												\draw[blue][-latex](0,0)--(-1.5,-0.25);
												\draw[blue][-latex](0,0)--(-1.5,0.25);
												\draw[blue][-latex](0,0)--(1.5,0.25);
												\draw[blue][-latex](0,0)--(1.5,-0.25);
												\draw[ ](0,1.5) ellipse (0.8 and 0.5);
												\draw[ ](0,-1.5) ellipse (0.8 and 0.5);
												\coordinate (A) at (-0.6,1.8);
												\fill (A) circle (0pt) node[right] {$+$};
												\coordinate (A1) at (-2.5,-1.8);
												\fill (A1) circle (0pt) node[right] {$+$};
												\coordinate (B) at (-0.6,-1.8);
												\fill (B) circle (0pt) node[right] {$-$};
												\coordinate (C) at (2.5,1.8);
												\fill (C) circle (0pt) node[right] {$-$};
												\coordinate (C1) at (-2.5,1.8);
												\fill (C1) circle (0pt) node[right] {$-$};
												\coordinate (D) at (2.5,-1.8);
												\fill (D) circle (0pt) node[right] {$+$};
												\coordinate (E) at (0,2);
												\fill (E) circle (1pt) ;
												\coordinate (F) at (0,-2);
												\fill (F) circle (1pt) ;
												\fill (E) circle (1pt)node[right] {$\frac{i}{2}$};
												\fill (F) circle (1pt)node[right] {$-\frac{i}{2}$};
												\coordinate (G) at (0,-1.3);
												\fill (G) circle (1pt)node[right] {$-ik_{n_1}$}; 
												\coordinate (H) at (0,1.3);
												\fill (H) circle (1pt) node[right] {$ik_{n_1}$}; 
												\coordinate (I) at (0,1);
												\fill (I) circle (1pt) ;
												\coordinate (J) at (0,-1);
												\fill (J) circle (1pt) ;
												\coordinate (K) at (0,-0.5);
												\fill (K) circle (1pt) ;
												\coordinate (L) at (0,0.5);
												\fill (L) circle (1pt) ;
												\draw[blue] (0,1.3) circle (0.2);
												\draw[blue] (0,-1.3) circle (0.2);
												\draw[blue] (0,0.5) circle (0.2);
												\draw[blue] (0,-0.5) circle (0.2);
												\fill (K) circle (1pt)node[right] {$-ik_{n_2}$};
												\fill (L) circle (1pt) node[right] {$ik_{n_2}$};
											\end{tikzpicture}}
												\caption{\footnotesize  (a), (b) show the Jump path of $M_R$ under the condition $\xi<-1/4$ and  $\xi>2$ respectively.  $\Sigma_R$ composed of the blue circles and lines in the figure.}
												\label{figVR}
											\end{figure}
	                                     Next, we prove the existence and uniqueness of solution of the above RHP \ref{RHP6}.
											
											\begin{Proposition}
												For any given scattering data $\mathcal{D}=\{r(k),\{ik_n,c_n\}_{n\in \aleph}\}$, the solution $M_R(k)$ of the above RH problem exists uniquely and is equivalent to the RHP \ref{RHP2} corresponding to the reflection-free N soliton solution of the corrected scattering data $\tilde{\mathcal{D}}=\{\tilde{r}=0,\{ik_n,\tilde{c}_n\}_{n\in \aleph}\}$ by a display transformation, where the correlation function $\tilde{c}_n=c_n\delta(ik_n)$.
												\begin{proof}
													$M_R(k)$ contains a number of jump lines consisting of small circles and poles in $\Lambda$. Take a transformation
													\begin{equation}
														\hat{M}(k)=M_R(k)G^{-1}(k)\left(\prod_{ik_n\in\Delta^+}\frac{ k+ik_n}{k-ik_n}\right)^{-\sigma_3},
													\end{equation}
													 then   \\
													a. $\hat{M}(k)$ is asymptotic to I when $k \rightarrow\infty$; \\
													b. $\hat{M}(k)$ has no jump on the small circles ; \\
													c. When $ik_n\not\in \Lambda$, the transformation converts the jump into the poles. For $ik_n\in \Lambda$,
													\begin{align}
														\nonumber &\res_{k=ik_n}\hat{M}(k)=\res_{k=ik_n}M_R(k)G^{-1}(k)\left(\prod_{ik_n\in\Delta^+}\frac{ k+ik_n}{k-ik_n}\right)^{-\sigma_3}\\ \nonumber
														=&\lim_{k \rightarrow ik_n}\hat{M}(k)\left(\prod_{ik_n\in\Delta^+}\frac{ k+ik_n}{k-ik_n}\right)^{-\sigma_3}\left(\begin{array}{cc}
															0 & c_nT^{-2}e^{-2it\theta{(ik_n)}}\\
															0 & 0
														\end{array}\right)G^{-1}(k)\left(\prod_{ik_n\in\Delta^+}\frac{ k+i}{k-ik_n}\right)^{-\sigma_3}\nonumber\\ =&\lim_{k \rightarrow ik_n}\hat{M}(k)\left(\begin{array}{cc}
															0 & c_n\delta(ik_n) e^{-2it\theta{(ik_n)}}\\
															0 & 0
														\end{array}\right),
													\end{align}
													thus, $\tilde{c}_n=c_n\delta(ik_n)$;
													
													So far, we have proved that $\hat{M}(k)$ is the solution of RHP \ref{RHP2} under the corrected scattering data $\tilde{\mathcal{D}}=\{\tilde{r}=0,\{ik_n,c_n\delta(ik_n)\}_{ik_n\in \aleph}\}$. And the solution of this kind of RH problem  is existentially unique, so $M_R(k)$ exists uniquely.
												\end{proof}
											\end{Proposition}
										
											\subsubsection{Soliton solutions}
											
Let  $M^{\Lambda}(k)$ denote the RH problem for $M_R(k)$  without jumps,  then  $M^{\Lambda}(k)$ satisfies
											
											\begin{RHP}\label{RHP7}
												Find a matrix valued function  $ M^\Lambda(k)$ with following properties:
												
												$\blacktriangleright$ Analyticity:  $M^{\Lambda}(k)$ is meromorphic in $\mathbb{C}$;
												
												$\blacktriangleright$ $\bar{\partial}$-Derivative: $\bar{\partial}R^{(2)}(k)=0$;
												
												$\blacktriangleright$ Asymptotic behaviors:
												\begin{align}
													M^\Lambda(k) =& I+\mathcal{O}(k^{-1}),\hspace{0.5cm}k \rightarrow \infty; \nonumber
												\end{align}
												
												$\blacktriangleright$ Residue conditions: $M^{\Lambda}$(k) has the same residue conditions with the RHP \ref{RHP6}.
												
											\end{RHP}

											\begin{Proposition}
												For any given corrected scattering data $\tilde{\mathcal{D}}_\Lambda=\{\tilde{r}=0,\{ik_n,\tilde{c}_n\},ik_n\in \Lambda\}$, the solution of RHP \ref{RHP7} exists uniquely and is shown to be constructed:\\
												\textbf{I}: if $\Lambda=\varnothing$, then
												\begin{equation}
													M^\Lambda(k)=I;\label{msol1}
												\end{equation}
												\textbf{II}: if $\Lambda\neq\varnothing$ with $\Lambda=\left\lbrace ik_n\right\rbrace_{n=1}^{\Xi} $, where the symbol $\Xi$ denotes the number of poles in $\Lambda$, then
												\begin{align}
													M^\Lambda(k)&=I+
													\sum_{n=1}^{\Xi}\left(\begin{array}{cc}
														\frac{\gamma_n}{k+ik_n} & \frac{\beta_n}{k-ik_n}\\
														\frac{\zeta_n}{k+ik_n} & \frac{\alpha_n}{k-ik_n}
													\end{array}\right) .\label{msol2}
												\end{align}
											Here  $\alpha_n,\beta_n,\gamma_n,\zeta_n$  are determined by linearly dependant equations:
												\begin{align}
													&c_n^{-1}T^2(ik_n)e^{2it\theta{(ik_n)}}\beta_n=1+\sum_{h=1}^{\Xi}\frac{\gamma_h}{ik_n+ik_h} , \label{cn1}\\	
													&c_n^{-1}T^2(ik_n)e^{2it\theta{(ik_n)}}\alpha_n=\sum_{h=1}^{\Xi}\frac{\zeta_h}{ik_n+ik_h} ,\label{cn2} \\
													&c_n^{-1}T^2(ik_n)e^{2it\theta{(ik_n)}}\gamma_n=\sum_{h=1}^{\Xi}\frac{-\beta_h}{ik_n+ik_h},\label{cn3}\\
													&c_n^{-1}T^2(ik_n)e^{2it\theta{(ik_n)}}\zeta_n=1+\sum_{h=1}^{\Xi}\frac{-\alpha_h}{ik_n+ik_h}.\label{cn4}
												\end{align}
											\end{Proposition}
											\begin{proof}
												The uniqueness of solution follows from the Liouville's theorem. The result of \textbf{I} can be simple obtain.
												
												As for \textbf{II}, the residue and asymptotic conditions ensure that $M^\Lambda(k)$ can have an expansion of the form as mentioned above. $\alpha_n,\beta_n,\gamma_n,\zeta_n$ are obtained from  four linearly dependant equations  (\ref{cn1})-(\ref{cn4}) which come from  substituting  (\ref{msol2}) into (\ref{resMR}).
											\end{proof}
											
											For convenience, denote the asymptotic expansion of $M^{\Lambda}(k)$ as $k\to \frac{i}{2}$:
											\begin{align}
												M^{\Lambda}(k)=M^{\Lambda}(\frac{i}{2})+M^{\Lambda}_1(k-\frac{i}{2})+\mathcal{O}((k-\frac{i}{2})^{2}).\label{asymr}
											\end{align}
											
											\subsubsection{Error estimate between $M_R$ and $M^\Lambda$}
											Here we try to match $M_R(k)$ to the soliton solution $M^{\Lambda}(k)$ and show that the error between the two is a small norm RHP. We define $M^{err}$ as the error between $M_R$ and $M^\Lambda$ with
											\begin{equation}
												M^{err}(k)=M_R(k)\left(M^\Lambda(k)\right)^{-1}, \label{Merr}
											\end{equation}
											 which  satisfies the following   RH problem
											
											\begin{RHP}\label{RHP8}
												Find a matrix-valued function $M^{err}(k)$ with following identities:
												
												$\blacktriangleright$ Analyticity: $M^{err}(k)$  is analytical  in $\mathbb{C}\setminus  \Sigma_{R} $;

												$\blacktriangleright$ Asymptotic behaviors:
												\begin{align}
													&  M^{err}(k) \sim I+\mathcal{O}(k^{-1}),\hspace{0.5cm}|k| \rightarrow \infty;
												\end{align}

												$\blacktriangleright$ Jump condition: $M^{err}(k)$ has continuous boundary values $M^{err}_\pm(k)$ on $\Sigma_R$ satisfying
												$$M^{err}_+(k)=M^{err}_-(k)V^{err}(k),$$
												where the jump matrix $V^{err}(k) $ is given by
												\begin{equation}
													V^{err}(k)=M^{\Lambda}(k)V_R(k)M^{\Lambda}(k)^{-1}. \label{Verr}
												\end{equation}
											\end{RHP}

										\begin{lemma}\label{lemmav2}
											The jump matrix $ V_R(k)$ defined by (\ref{jumpG}) satisfies
											\begin{align}
												&\parallel V_R(k)-I\parallel_{L^\infty(\Sigma_R\setminus\Sigma^{(2)})}=\mathcal{\mathcal{O}}(e^{- 2\rho_0t} ),\label{guji1}
											\end{align}
											where
											\begin{equation}
												\rho_0=\min_{ik_n\in\Delta\setminus \Lambda}\left\lbrace |\text{Im}\theta_n| >\delta_0\right\rbrace.\label{rho0}
											\end{equation}
										\end{lemma}
										
										\begin{proof}
											Take $k\in\{k:|k-ik_n|=\rho,ik_n\in\Delta^+\}$ as an example.
											\begin{align}
												\parallel V_R(k)-I\parallel_{L^\infty}&=|c_n^{-1}T(k)e^{2it\theta(ik_n)}[(\frac{1}{T})^{'}(ik_n)]^{-1}|\nonumber\\
												&\lesssim e^{-\text{Re}(2it\theta_n)}\lesssim e^{2t\text{Im}(\theta_n)}\leq e^{-2\rho_0t}.\nonumber
											\end{align}
											The proofs for the rest of the cases are similar.
										\end{proof}
										\begin{corollary}\label{v2p}
											For $1\leq p\leq +\infty$, the jump matrix $V_R(k)$ satisfies
											\begin{align}
												&	\parallel V_R(k)-I\parallel_{L^p(\Sigma_R\setminus\Sigma^{(2)})}\leq C_pe^{- 2\rho_0t} ,\\
												&	\parallel V_R(k)-I\parallel_{L^p(\Sigma^{(2)})}\leq C_pt^{-\frac{1}{p}},
											\end{align}
											where  $C_p\geq 0$  depends  on  $p$.
										\end{corollary}
											
											\begin{lemma}\label{lemmav2}
												The jump matrix $ V^{err}(k)$ satisfies
												\begin{align}
													&\parallel V^{err}(k)-I\parallel_{L^\infty(\Sigma_R\setminus\Sigma^{(2)})}=\mathcal{\mathcal{O}}(e^{- 2\rho_0t} ),\label{guji2}
												\end{align}
											\end{lemma}
											
											\begin{proof}
												Take $k\in\{k:|k-ik_n|=\rho,ik_n\in\Delta^+\}$ as an example.
												\begin{align}
													|V_R(k)-I|=|M^{\Lambda}(k)(V_R-I)M^{\Lambda}(k)^{-1}|\leq c|V_R-I|\leq e^{-2\rho_0t}.\nonumber
												\end{align}
												The proofs for the rest of the cases are similar.
											\end{proof}
											
											\begin{corollary}\label{verrp}
												For $1\leq p\leq +\infty$, the jump matrix $V^{err}(k)$ satisfies
												\begin{align}
												&	\parallel V^{err}(k)-I\parallel_{L^p(\Sigma_R\setminus\Sigma^{(2)})}\leq C_pe^{- 2\rho_0t},\\
													&	\parallel V^{err}(k)-I\parallel_{L^p(\Sigma^{(2)})}\leq C_pt^{- \frac{1}{p}}.
												\end{align}
											
											\end{corollary}
					
										 Denote
											\begin{align}
							  C_{\omega_e}f=C_-(f(V^{err}-I)), \label{nlscauchy65}
											\end{align}
											where $C_{-}$ is the Cauchy projection operator, defined as follows
											\begin{equation}\label{eq:5.23}
												C_{-}f(k)=\operatorname*{lim}\limits_{k'\to k\in\Sigma_R}  \frac{1}{2\pi i}\int_{\Sigma^{(2)}}\frac{f(s)}{s-k'}ds,
											\end{equation}
											and $\|C^{-}\|_{L^{2}}$ is bounded. By  Beals-Coifman's theorem, the solution of the above RH problem can be expressed as
											\begin{align}
												&M^{err}(k)=I+\frac{1}{2\pi i}\int_{\Sigma_R}  \frac{\mu_{e} (s)(V^{err}(s)-I)}{s-k}ds,\label{cauchy68}
											\end{align}
											where $\mu_{e} (k)\in L^{2}(\Sigma_R)$ satisfies
											\begin{align}
												& (1-C_{\omega_e})\mu_{e}(k)=I. \label{fcauchy66}
											\end{align}
											Using (\ref{guji2})-(\ref{nlscauchy65}), we can see that
											\begin{equation}\label{eq:5.24}
												||C_{\omega_e}||_{L^{2}(\Sigma_R)}\leq
												||C_{-}||_{L^{2}(\Sigma_R)}||V^{err}(k)-I||_{L^{\infty}(\Sigma_R) }\leq  c t^{-\frac{1}{2}},
											\end{equation}
									which implies that the operator $(1-C_{\omega_e})^{-1}$ exists  for large  $t$, so    $\mu_e$ and  $M^{err}(k)$   exist uniquely.
					Next  we make  some estimates and asymptotic behaviors of $M^{err}(k)$.
											
											\begin{Proposition}
												For $M^{err}(k)$ defined in (\ref{Merr}), it stratifies
												\begin{align}
													&|M^{err}(k)-I|\leq ct^{-\frac{1}{2}}.
												\end{align}
												When $k=\frac{i}{2}$,
												\begin{equation}
													M^{err}(\frac{i}{2})=I+\frac{1}{2\pi i}\int_{\Sigma_R}\dfrac{\mu_e(s)(V^{err}(s)-I)}{s-\frac{i}{2}}ds,\label{Merri}
												\end{equation}
												As $k\to \frac{i}{2}$, the Laurent expansion of $M^{err}(k)$ is
												\begin{align}
													M^{err}(k)=	M^{err}(\frac{i}{2})+M_1^{err}(k-\frac{i}{2})+\mathcal{O}((k-\frac{i}{2})^{2}),\label{expMerr}
												\end{align}
												where
												\begin{equation}
													M_1^{err}=\frac{1}{2\pi i}\int_{\Sigma_R}\frac{\mu_e(s)(V^{err}(s)-I)}{(s-\frac{i}{2})^2}ds.\label{Merr1}
												\end{equation}
												In addition, $M^{err}(\frac{i}{2})$ and  $M_1^{err}$ satisfy following  asymptotics
												\begin{equation}
													|M^{err}(\frac{i}{2})-I|\lesssim\mathcal{O}(t^{-\frac{1}{2}}) ,\hspace{0.5cm} M_1^{err}\lesssim\mathcal{O}(t^{-\frac{1}{2}}).\label{asyerrMerr}
												\end{equation}
												\begin{proof}
													By (\ref{cauchy68}),
													\begin{align}
														&M^{err}(k)-I= \frac{1}{2\pi i}\int_{\Sigma_R}  \frac{  V^{err}(s)-I }{s-k}ds + \frac{1}{2\pi i}\int_{\Sigma_R}  \frac{ (\mu_{e} (s)-I) (V^{err}(s)-I)}{s-k}ds ,\nonumber
													\end{align}
													from which the following estimate is obtained
													\begin{align}
														| M^{err}(k)-I  | & \leq  \|  V^{err}(s)-I\|_{L^2(\Sigma_R)}\| \frac{ 1 }{s-k}  \|_{L^2(\Sigma_R)}+ \nonumber\\
														& \ \ \ \  \|  V^{err}(s)-I\|_{L^\infty(\Sigma_R)} \|  \mu_{e} (s)-I\|_{L^2(\Sigma_R)}  \| \frac{ 1 }{s-k}  \|_{L^2(\Sigma_R)}\nonumber \\
														& \leq  c t^{-\frac{1}{2}}. \nonumber
													\end{align}
													(\ref{expMerr})-(\ref{Merr1}) can be obtained by  Taylor expansion. At last, since $ (s-\frac{i}{2})^{-2} $ is bounded in $\Sigma_R$, the proof of (\ref{asyerrMerr}) is given as follows:
													\begin{align}
														| M_1^{err} | \lesssim || V^{err}-I ||_{L^2}||\mu_{e} (s)-I||_{L^2}+||V^{err}-I||_{L^1} \lesssim \mathcal{O}(t^{-\frac{1}{2}}). \nonumber
													\end{align}
												\end{proof}
												
											\end{Proposition}

											\subsection{Analysis on a pure $\bar{\partial}$ problem}\label{subsec3.3}

											In this subsection, we mainly consider the following pure $\bar{\partial}$-problem which is defined by (\ref{M2})
											
											\begin{RHP}\label{RHP9}
												\noindent (Pure $\bar{\partial}$-problem) Find   $  M^{(3)}(k)$ with following identities:
												
												$\blacktriangleright$ Analyticity: $M^{(3)}(k)$ is continuous   and has sectionally continuous first partial derivatives in $\mathbb{C}$.

												$\blacktriangleright$ Asymptotic behavior:
												\begin{align}
													&M^{(3)}(k) \sim I+\mathcal{O}(k^{-1}),\hspace{0.5cm}k \rightarrow \infty;\label{asymbehv7}
												\end{align}
												
												$\blacktriangleright$ $\bar{\partial}$-Derivative: We have
												$$\bar{\partial}M^{(3)}(k)=M^{(3)}(k)W^{(3)}(k),\ \ k\in \mathbb{C},$$
												where
												\begin{equation}
													W^{(3)}(k)=M^{(R)}(k)\bar{\partial}R^{(2)}(k)M^{(R)}(k)^{-1}.
												\end{equation}
											\end{RHP}
											
											The solution of the pure $\bar\partial$ problem $ M^{(3)}(k)$ is given by the following integral equation
											\begin{equation}\label{feq:61}
												M^{(3)}(k)=
												I-\frac{1}{\pi}\iint_{\mathbb{C}}\frac{M^{(3)}W^{(3)}}{s-k}\mathrm{d}A(s),
											\end{equation}
											where $\mathrm{d}A(s)$ is the Lebesgue measure in the complex plane. The equation (\ref{feq:61}) can also be expressed as an operator equation
											\begin{equation}
												(I-J) M^{(3)}(k) =I, \label{opers}
											\end{equation}
											where $J$ is the Cauchy operator
											\begin{equation}
												Jf(k)= -\frac{1}{\pi}\iint_{\mathbb{C}}\frac{f(s)W^{(3)}(s)}{s-k}\mathrm{d}A(s).\label{fe3st3}
											\end{equation}
											
											\begin{lemma}\label{lemmatheta}
												
												For $\xi \in (-\infty,-\frac{1}{4})$, $ {\rm Im }\theta(k)$ has following estimation:
												\begin{align}
													&\text{Im }\theta(k)\leq |k||sin\omega|\left(\xi+\frac{1}{1+cos2\varphi+\sqrt{2(1+cos2\varphi)}}\right) ,\hspace{0.5cm} \text{as }k\in\Omega_1, \Omega_2, \\
													&\text{Im }\theta(k)\geq -|k||sin\omega|\left(\xi+\frac{1}{1+cos2\varphi+\sqrt{2(1+cos2\varphi)}}\right) ,\hspace{0.3cm} \text{as }k\in\Omega_3, \Omega_4.
												\end{align}
												For $\xi \in (2,+\infty)$, the estimation is shown as follow:
												\begin{align}
													&\text{Im }\theta(k)\geq |k||sin\omega|\left(\xi+\frac{1}{1+cos2\varphi-\sqrt{2(1+cos2\varphi)}}\right) ,\hspace{0.5cm} \text{as }k\in\Omega_1, \Omega_2, \\
													&\text{Im }\theta(k)\leq -|k||sin\omega|\left(\xi+\frac{1}{1+cos2\varphi-\sqrt{2(1+cos2\varphi)}}\right) ,\hspace{0.3cm} \text{as }k\in\Omega_3, \Omega_4.
												\end{align}
												\begin{proof}
													We just take $k \in \Omega_1$ as an example, and the other regions are similarly.
													
													Since
													$\theta(k)=k\xi-\frac{k}{2k^2+\frac{1}{2}}$, 
													let $k=|k|e^{i\omega}=u+iv$, then
													\begin{equation}
														{\rm Im} \theta(k)=ksin\omega \left(\xi+\frac{2|k|^2-\frac{1}{2}}{(2|k|^2cos2\omega+\frac{1}{2})^2+4|k|^2 sin^22\omega}\right). \nonumber
													\end{equation}
													Denote
													\begin{align}
														h(x;a)=\frac{x}{(x+\frac{1}{2})^2+a(x+\frac{1}{2})+\frac{1}{4}},\nonumber
													\end{align}
													where $x=2|k|^2-\frac{1}{2}$, $a=cos2\omega$, and then
													\begin{align}
												h(x,a)\in\left(-\infty,1+a-\sqrt{2(1+a)}\right)\bigcup\left(1+a+\sqrt{2(1+a)},+\infty\right).\nonumber
												\end{align}
												From ${\rm Im} \theta=v\left(\xi+h(x;a)\right)$ with $x \in (-\frac{1}{2},+\infty)$, $a \in [cos2\varphi,1]$, the conclusion of the lemma is obtained.
												\end{proof}
											\end{lemma}
											
											\begin{corollary}\label{Imtheta}
												There exist constants $c(\xi)<0$, $\tilde{c}(\xi)>0$ relative to $\xi$ that the imaginary part of phase function (\ref{theta}) $\text{Im }\theta(k)$ have following evaluation for $k=|k|e^{i\omega}=u+i v $:\\
												When $\xi\in(-\infty,-1/4)$,
												\begin{align}
													& {\rm Im }\theta(k)\leq c(\xi)v ,\hspace{0.5cm} \text{as }k\in\Omega_1, \Omega_2,\\ \label{3}
													& {\rm Im }\theta(k)\geq -c(\xi)v ,\hspace{0.5cm} \text{as }k\in\Omega_3, \Omega_4.	
												\end{align}
												When $\xi\in(2,+\infty)$,
												\begin{align}
													& {\rm Im }\theta(k)\geq \tilde{c}(\xi)v ,\hspace{0.5cm} \text{as }k\in\Omega_1, \Omega_2,\\
													& {\rm Im }\theta(k)\leq -\tilde{c}(\xi)v ,\hspace{0.5cm} \text{as }k\in\Omega_3, \Omega_4.
												\end{align}	
											\end{corollary}
										
											Next we prove that the above operator $J$ is small parametric when $t$ is sufficiently large.
											\begin{Proposition}\label{Joperator}
												To the case  $\xi\in(-\infty,-1/4)\cup(2,+\infty)$, for sufficiently large $t$,  we have
												\begin{equation}
													|| J ||_{ L^{\infty}\rightarrow L^{\infty} }  \leqslant ct^{-\frac{1}{2}}.\label{fest914}
												\end{equation}
												Therefore $(1-J)^{-1}$ exists, and thus the operator equation (\ref{opers}) has a unique solution.
												\begin{proof}
													Here, we just take the case $\xi\in(-\infty,-1/4)$ as an example. In order to get the estimate (\ref{fest914}), it is only need to prove that for $f\in L^{\infty}(\mathbb{C}) $,
													\begin{equation}
														||Jf ||_{ L^{\infty} (\mathbb{C}) }  \leqslant ct^{-1/2 }|| f ||_{ L^{\infty}(\mathbb{C})}.\nonumber
													\end{equation}
													Here, we just prove the case $k\in \Omega_1$. By (\ref{fe3st3}),
													\begin{equation}
														|Jf(k)|\leq c || f ||_{ L^{\infty}}  \iint_{\Omega_1}\frac {|W^{(3)}(s)| }{|s-k|  }dA(s),\label{fest313}
													\end{equation}
													where
													\begin{align}
														|W^{(3)}(s)|  &\leq|M^{(R)}||\bar{\partial}R^{(2)}||(M^{(R)})^{-1}| \leq |M^{(R)}||\bar{\partial}R^{(2)}||\sigma_2(M^{(R)})^{T}\sigma_2| \nonumber \\
														& \leq |M^{(R)}|^{2}|\bar{\partial}R^{(2)}| \leq c|\bar{\partial}R_1(k)||e^{-2it\theta}|,
													\end{align}
													thus
													\begin{align}
														\iint_{\Omega_1}\frac {|W^{(3)}(s)| }{|s-k|  }dA(s) \lesssim \iint_{\Omega_1}\frac {|\bar{\partial}R_1|e^{2tIm\theta}}{|s-k|  }dA(s)
													\end{align}
													
													Referring  into (\ref{dbarRj}) in Proposition \ref{proR}, the integral $\iint_{\Omega_1}\frac {|\bar{\partial}R_1|e^{2tIm\theta}}{|s-k|  }dA(s)$	can be divided to two part:
													\begin{align}
														\iint_{\Omega_1}\dfrac{|\bar{\partial}R_1 (s)|e^{2t\text{Im}\theta}}{|s-k|}dA(s)\lesssim I_1+I_2,	\label{chai}
													\end{align}
													with
													\begin{align}
														I_1=\iint_{\Omega_1}\dfrac{|r'(|s|)|e^{2t\text{Im}\theta}}{|s-k|}dA(s),\hspace{0.5cm}
														I_2=\iint_{\Omega_1}\dfrac{|s|^{-1/2}e^{2t\text{Im}\theta}}{|s-k|}dA(s).	
													\end{align}
													Denoting $s=u+iv$, $k=k_R+ik_I$, (\ref{3}) gives that
													\begin{align}
														&\text{Im }\theta(s)\leq c(\xi)2tv\leq c'tv.
													\end{align}
													Therefore,
													\begin{align}
														&I_{1}= \int_0^\infty e^{c'tv}  dv  \int_v^\infty  \frac { |r'(|s|)| }{|s-k| }du
														\leq    \int_0^\infty  e^{c'tv } \| r'(|s|) \|_{L^2(v,\infty)} \bigg\| \frac{ 1 } {s-k}\bigg\|_{L^2(v,\infty)} dv. \label{est623}
													\end{align}
												Note that
													\begin{align}
														&   \bigg\| \frac{ 1} {s-k}\bigg\|^2_{L^2(v,\infty)}=  \int_{v}^\infty
														\frac{1}{|s-k|^2}du   \leq \int_{-\infty}^\infty    \frac{1}{ (u-k_R)^2+(v-k_I)^2     }    du  \nonumber\\
														&=   \frac{1}{| v-k_I|}   \int_{-\infty}^\infty    \frac{1}{ 1+   y^2     }    d  y
														=    \frac{\pi} {|v- k_I|} ,  \label{ineq3}
													\end{align}
													where $ y=\frac{u-k_R}{v-k_I}$. And
													\begin{align}
														&  \| r'(|s|) \|_{L^2(v,\infty)}^2 =  \int_{v}^\infty
														\big | r'( \sqrt{u^2+v^2})\big|^2  du  \leq  \int_{\sqrt{2}v}^\infty
														\big | r'(\tau)\big|^2  \frac{\sqrt{u^2+v^2}}{u} d\tau   \nonumber\\
														&\leq  \sqrt{2}  \int_{\sqrt{2}v}^\infty
														\big | r'(\tau)\big|^2   d\tau
														\leq  \| r'(s) \|_{L^2(\mathbb{R})}^2. \label{ineq33}
													\end{align}
												(\ref{ineq3})-(\ref{ineq33}) implies that
													\begin{align}
														&  I_{1} \leq c \|r'\|_{L^2(\mathbb{R})} \int_0^\infty e^{c'tv} |v-k_I|^{-1/2}  dv  \nonumber \\
														&=c \|r'\|_{L^2(\mathbb{R})} \left[   \int_0^{k_I}   \frac{ e^{c'tv }} {\sqrt{k_I-v}}  dv+ \int_{k_I}^\infty   \frac{ e^{c'tv }} {\sqrt{v-k_I}}  dv\right] \nonumber \\
														&\lesssim t^{-\frac{1}{2}}.\label{fh33gf1}
													\end{align}
													To estimate $I_{2}$, we consider the following two $L^p$-estimates $(p>2)$.
													\begin{align}
														&  \||s|^{-1/2} \|_{L^p(v,\infty)}  = \left(\int_{v}^\infty
														\frac{1}{|u+iv |^{p/2}}du\right)^{1/p} =  \left(\int_{v}^\infty
														\frac{1}{(u^2+v^2)^{p/4}}du\right)^{1/p}  \nonumber\\
														&= v^{1/p-1/2} \left(\int_1^\infty
														\frac{1}{(1+x^2)^{p/4}}dx\right)^{1/p} \leq cv^{1/p-1/2};\\
														& \big\| |s-k|^{-1}  \big\|_{L^q(v,\infty)}  \leq c |v-k_I|^{1/q-1 }, \ \ 1/p+1/q=1. \ \label{Lpest}
													\end{align}
													Using the two estimates above, we arrive at
													\begin{align}
														I_{2}&\leq \int_0^\infty e^{c'tv}  dv  \int_v^\infty  \frac {  |s|^{-1/2}  }{ |s-k| }du \nonumber\\
														&\leq    \int_0^\infty  e^{c'tv } \||s|^{-1/2} \|_{L^p(v,\infty)} \big\| |s-k|^{-1}  \big\|_{L^q(v,\infty)} dv.\nonumber\\
														&\leq c \left[    \int_0^{k_I} e^{c'tv } v^{1/p-1/2} |v-k_I|^{1/q-1} dv  +\int_{k_I}^\infty e^{c'tv } v^{1/p-1/2} |v-k_I|^{1/q-1}dv \right].\label{piepe}
													\end{align}
													Calculating the two integrals yields
													\begin{align}
														I_{2}\leq ct^{-\frac{1}{2}}.\label{jej}
													\end{align}
													By (\ref{fest313})-(\ref{chai}), the result of Proposition \ref{Joperator} is proved.
												\end{proof}
											\end{Proposition}
											
											\begin{Proposition}\label{asyM3i}
												There exist a  positive constant $\frac{1}{4}<\rho<\frac{1}{2}  $ such that the solution $M^{(3)}(k)$  of  $\bar{\partial}$-problem  admits the following estimation
												\begin{align}
													\parallel M^{(3)}(\frac{i}{2})-I\parallel=\parallel\frac{1}{\pi}\iint_\mathbb{C}\dfrac{M^{(3)}(s)W^{(3)} (s)}{s-\frac{i}{2}}dA(s)\parallel\lesssim t^{-\frac{3}{2}+2\rho}.\label{m3i}
												\end{align}
												As $k\to \frac{i}{2}$, $M^{(3)}(k)$ has asymptotic expansion
												\begin{equation}
													M^{(3)}(k)=M^{(3)}(\frac{i}{2})+M^{(3)}_1(x,t)(k-\frac{i}{2})+\mathcal{O}((k-\frac{i}{2})^{2}),
												\end{equation}
												where $M^{(3)}_1(x,t) $ is a $k$-independent coefficient with
												\begin{equation}
													M^{(3)}_1(x,t)=\frac{1}{\pi}\iint_C\dfrac{M^{(3)}(s)W^{(3)} (s)}{(s-\frac{i}{2})^2}dA(s),
												\end{equation}
												and $M^{(3)}_1(x,t)$  satisfies
												\begin{equation}
													|M^{(3)}_1(x,t)|\lesssim t^{-\frac{3}{2}+2\rho}.\label{M31}
												\end{equation}
											\end{Proposition}
											\begin{proof}
												First we estimate (\ref{m3i}). The proof proceeds along the same steps as the proof of above Proposition. 	(\ref{opers}) and (\ref{fest914}) implies that for large $t$,   $\parallel M^{(3)}\parallel_\infty \lesssim1$. And here we only estimate the integral on sector $\Omega_1$ as $\xi<-1/4$. Let $s=u+vi$.	We also divide $M^{(3)}(\frac{i}{2})-I$ to two parts:
												\begin{equation}
													\frac{1}{\pi}\iint_\mathbb{C}\dfrac{M^{(3)}(s)W^{(3)} (s)}{s-\frac{i}{2}}dA(s)\lesssim I_3+I_4,
												\end{equation}
												with
												\begin{align}
													I_3=\iint_{\Omega_1}\dfrac{|r' (|s|)|e^{2t\text{Im}\theta}}{|\frac{i}{2}-s|}dA(s),\hspace{0.5cm}
													I_4=\iint_{\Omega_1}\dfrac{|s|^{-\frac{1}{2}}e^{2t\text{Im}\theta}}{|\frac{i}{2}-s|}dA(s).	
												\end{align}	
												For $s\in\Omega_1$, there exists a constant $c>0$ so that $|\frac{i}{2}-s|\leq c $. So
												\begin{align}
													I_3\leq&\int_{0}^{+\infty}\int_{v}^{+\infty}\dfrac{|r' (|s|)|e^{c'tv}}{|\frac{i}{2}-s|}dudv\nonumber\\
													&\lesssim \int_{0}^{+\infty} \parallel r'\parallel_{L^1(\mathbb{R}^+)}  e^{c'tv} dv\lesssim \int_{0}^{+\infty}  e^{c'tv} dv\lesssim t^{-1}.\nonumber
												\end{align}
												As for $I_4$, we partition it to two parts:
												\begin{align}
													 I_4\leq\int_{0}^{\frac{1}{4}}\int_{v}^{+\infty}\dfrac{|s|^{-\frac{1}{2}}e^{c'tv}}{|\frac{i}{2}-s|}dudv+\int_{\frac{1}{4}}^{+\infty}\int_{v}^{+\infty}\dfrac{|s|^{-\frac{1}{2}}e^{c'tv}}{|\frac{i}{2}-s|}dudv:=I_{41}+I_{42}.
												\end{align}
												For $0<v<\frac{1}{4}$, $|s-\frac{i}{2}|^2=u^2+(v-\frac{1}{2})^2>u^2+v^2=|s|^2$, while as $v>\frac{1}{4}$, $|s-\frac{i}{2}|^2<|s|^2$. Then  the first integral has:
												\begin{align}
													&I_{41}\leq\int_{0}^{\frac{1}{4}}\int_{v}^{+\infty} (u^2+v^2)^{-\frac{1}{4}-\rho}(u^2+(v-\frac{1}{2})^2)^{-\frac{1}{2}+\rho}due^{c'tv}dv\nonumber\\
													&\leq \int_{0}^{\frac{1}{4}}\left[ \int_{v}^{+\infty}\left(1+\left(\frac{u}{v} \right)^2  \right) ^{-\frac{1}{4}-\rho}v^{\frac{1}{2}-2\rho} d\frac{u}{v}\right]  (v^2+(v-\frac{1}{2})^2)^{-\frac{1}{2}+\rho}e^{c'tv}dv\nonumber\\
													&\lesssim  \int_{0}^{\frac{1}{4}}v^{\frac{1}{2}-2\rho}(\frac{1}{8})^{-\frac{1}{2}+\rho}e^{c'tv}dv\lesssim t^{-\frac{3}{2}+2\rho}.
												\end{align}
												The second integral can be bounded in similar way:
												\begin{align}
													\quad \quad &I_{42}\leq\int_{\frac{1}{4}}^{+\infty}e^{c'tv}\parallel |s|^{-\frac{1}{2}}\parallel_{L^p} \parallel |\frac{i}{2}-s|^{-1}\parallel_{L^q} dv \leq\int_{\frac{1}{4}}^{+\infty}(v)^{-\frac{1}{2}+2\rho}|v-\frac{1}{2}|^{-2\rho}e^{c'tv}dv\nonumber\\
													&= \int_{\frac{1}{4}}^{\frac{1}{2}} v^{-\frac{1}{2}+2\rho}(\frac{1}{2}-v)^{-2\rho}e^{c'tv}dv+\int_{\frac{1}{2}}^{1}v^{-\frac{1}{2}+2\rho}(v-\frac{1}{2})^{-2\rho}e^{c'tv}dv+\int_{1}^{+\infty}v^{-\frac{1}{2}+2\rho}(v-\frac{1}{2})^{-2\rho}e^{c'tv}dv\nonumber\\
													&\leq e^{\frac{c'}{4}t}\int_{\frac{1}{4}}^{\frac{1}{2}} (\frac{1}{2}-v)^{-2\rho}dv+e^{\frac{c'}{2}t}\int_{\frac{1}{2}}^{1}(v-\frac{1}{2})^{-2\rho}dv\int_{1}^{+\infty}v^{-\frac{1}{2}+2\rho}e^{c'tv}dv\lesssim t^{-(\frac{1}{2}+2\rho)}.\nonumber
												\end{align}
												This estimation  is strong enough to obtain the result (\ref{m3i}). And (\ref{M31}) is obtained by deflating $|\frac{i}{2}-s|^2$ to a constant   for $s\in\Omega_1$.
											\end{proof}
											
											\subsection{Long time asymptotic behaviors}
											We begin to construct the long time asymptotics of the CH equation (\ref{ch}).
											According to the series of RHP transformations we have done before, $M(k)$ is given as follows.\\
											(i) For $\xi<-\frac{1}{4}$,
											\begin{align}
												M(k)=&\left(I-\frac{1}{k-1}\sigma_1\right)^{-1}\left[I-\frac{1}{k-1}\sigma_1\left(M^{(3)}(1)M^{err}(1)\right)^{-1}\right] M^{(3)}(k)M^{err}(k) \label{ope}\\
												:=&F(y,t,k)M^{(3)}(y,t,k)M^{err}(y,t,k);\nonumber
											\end{align}
											To  reconstruct   $u(y,t)$ by using (\ref{cover}),    in above  equation ,  we  take    $k\to \frac{i}{2}$   out of $\bar{\Omega}$.
											Further taking the Laurent expansion of each element under the large time asymptotics into (\ref{ope}) ,  we obtain that
											\begin{align}
												 M(k)=&\left(I+\mathcal{O}(t^{-\frac{1}{2}})\right)\left(I+\mathcal{O}(t^{-\frac{3}{2}+2\rho})+\mathcal{O}((k-\frac{i}{2})^2)\right)\left(I+\mathcal{O}(t^{-\frac{1}{2}})+\mathcal{O}((k-\frac{i}{2})^2)\right)\nonumber \\
												=&I+\mathcal{O}(t^{-\frac{1}{2}})+\mathcal{O}((k-\frac{i}{2})^2).
											\end{align}
											Let $f(k)=\left(M_{11}(k)+M_{21}(k)\right)\left(M_{12}(k)+M_{22}(k)\right)$, then (\ref{cover}) comes to
											\begin{align}
												u(y,t)=\frac{1}{2i}\frac{f'(\frac{i}{2})}{f(\frac{i}{2})}. \label{newu}
											\end{align}
											Substituting  above estimates  into  (\ref{newu}) and (\ref{xy}),
											\begin{align}
												 u(x,t)=&u(y(x,t),t)=\frac{1}{2i}\left(\frac{M_{11}'(\frac{i}{2})+M_{21}'(\frac{i}{2})}{M_{11}(\frac{i}{2})+M_{21}(\frac{i}{2})}+\frac{M_{12}'(\frac{i}{2})+M_{22}'(\frac{i}{2})}{M_{12}(\frac{i}{2})+M_{22}(\frac{i}{2})}\right)=\mathcal{O}(t^{-\frac{1}{2}}),  \label{resultu}
											\end{align}
											and
											\begin{align}
												x(y,t)&=y+\ln\left( \frac{M_{11}(\frac{i}{2})+M_{21}(\frac{i}{2})}{\left[M_{12}(\frac{i}{2})+M_{22}(\frac{i}{2})\right]a^2(\frac{i}{2})}\right). \nonumber
											\end{align}
											
											(ii) For $\xi>2$,
											\begin{align}
												M(k)=&\left(I-\frac{1}{k-1}\sigma_1\right)^{-1}\left[I-\frac{1}{k-1}\sigma_1\left(M^{J}(1)\right)^{-1}\right] M^J(k) \nonumber\\
												:=&F(y,t,k)M^J(k)=F(y,t,k)M^{(3)}(k)M^{err}(k)M^{\Lambda}(k)T^{-\sigma_3}(k);\label{ope4}
											\end{align}
											
											By simple calculations, let
											\begin{align}
												F(y,t,k)=
												\nonumber
												\left(\begin{array}{cc}
													F_{11}+F_{12}(k-\frac{i}{2})& F_{21}+F_{22}(k-\frac{i}{2})\\
													F_{31}+F_{32}(k-\frac{i}{2})& F_{41}+F_{42}(k-\frac{i}{2})
												\end{array}\right)+\mathcal{O}((k-\frac{i}{2})^2)+\mathcal{O}(t^{-\frac{1}{2}}).
											\end{align}
											Denote the asymptotic expansion of $T^{-1}(k)$ at $\frac{i}{2}$ as
											\begin{align}
												T^{-1}(k)=T_{0,-1}+T_{1,-1}(k-\frac{i}{2})+\mathcal{O}((k-\frac{i}{2})^2),\nonumber
											\end{align}
											then the Laurent expansion of $M(k)$ at $\frac{i}{2}$ can be written as
											\begin{align}\label{3-99}
												M(k)=\left(\begin{array}{cc}
													f_{11}(y,t,k)& f_{12}(y,t,k)\\
													f_{21}(y,t,k) & f_{22}(y,t,k)
												\end{array}\right)+ \mathcal{O}((k-\frac{i}{2})^2)+\mathcal{O}(t^{-\frac{1}{2}}),
											\end{align}
											where
											\begin{align}
												f_{11}:=&f_{1,11}\left(T_{0,-1}+T_{1,-1}(k-\frac{i}{2})\right)+f_{2,11}T_{0,-1}\left(k-\frac{i}{2}\right)\nonumber \\
												=&\left(F_{11}M_{11}^{\Lambda}(\frac{i}{2})+F_{21}M_{21}^{\Lambda}(\frac{i}{2})\right)\left(T_{0,-1}+T_{1,-1}(k-\frac{i}{2})\right)+ \nonumber \\ &\left(F_{11}M_{1,11}^{\Lambda}+F_{12}M_{11}^{\Lambda}(\frac{i}{2})+F_{21}M_{1,21}^{\Lambda}+F_{22}M_{21}^{\Lambda}(\frac{i}{2})\right)T_{0,-1}\left(k-\frac{i}{2}\right);\nonumber \\
												f_{12}:=&f_{1,12}\left(T_{0}+T_{1}(k-\frac{i}{2})\right)+f_{2,12}T_{0}\left(k-\frac{i}{2}\right)\nonumber \\
												=&\left(F_{11}M_{12}^{\Lambda}(\frac{i}{2})+F_{21}M_{22}^{\Lambda}(\frac{i}{2})\right)\left(T_{0}+T_{1}(k-\frac{i}{2})\right)+ \nonumber \\ &\left(F_{11}M_{1,12}^{\Lambda}+F_{12}M_{12}^{\Lambda}(\frac{i}{2})+F_{21}M_{1,22}^{\Lambda}+F_{22}M_{22}^{\Lambda}(\frac{i}{2})\right)T_{0}\left(k-\frac{i}{2}\right); \nonumber \\
												f_{22}:=&f_{1,22}\left(T_{0}+T_{1}(k-\frac{i}{2})\right)+f_{2,22}T_{0}\left(k-\frac{i}{2}\right)\nonumber\\
												=&\left(F_{31}M_{12}^{\Lambda}(\frac{i}{2})+F_{41}M_{22}^{\Lambda}(\frac{i}{2})\right)\left(T_{0}+T_{1}(k-\frac{i}{2})\right)+ \nonumber \\ &\left(F_{41}M_{1,12}^{\Lambda}+F_{32}M_{12}^{\Lambda}(\frac{i}{2})+F_{41}M_{1,22}^{\Lambda}+F_{42}M_{22}^{\Lambda}(\frac{i}{2})\right)T_{0}\left(k-\frac{i}{2}\right); \nonumber \\
												f_{21}:=&f_{1,21}\left(T_{0,-1}+T_{1,-1}(k-\frac{i}{2})\right)+f_{2,21}T_{0,-1}\left(k-\frac{i}{2}\right)\nonumber\\
												=&\left(F_{31}M_{11}^{\Lambda}(\frac{i}{2})+F_{41}M_{21}^{\Lambda}(\frac{i}{2})\right)\left(T_{0,-1}+T_{1,-1}(k-\frac{i}{2})\right)+ \nonumber \\ &\left(F_{31}M_{1,11}^{\Lambda}+F_{32}M_{11}^{\Lambda}(\frac{i}{2})+F_{41}M_{1,21}^{\Lambda}+F_{42}M_{21}^{\Lambda}(\frac{i}{2})\right)T_{0,-1}\left(k-\frac{i}{2}\right). \nonumber
											\end{align}
											Substituting each component of $M(k)$ into $(\ref{newu})$,  it follows that
											\begin{align}
												u(x,t)=&u(y(x,t),t):=u^{r}(y,t,\tilde{\mathcal{D}}_\Lambda)+\mathcal{O}(t^{-\frac{1}{2}}),
											\end{align}
											where $u^{r}$ is called as modulation-solitons and it has the following expressions
											\begin{align}\label{3-101}
												 u^{r}(y,t,\tilde{\mathcal{D}}_\Lambda)=\frac{1}{2i}\left(\frac{T_{1,-1}}{T_{0,-1}}+\frac{T_{1}}{T_{0}}+\frac{f_{2,11}+f_{2,21}}{f_{1,11}+f_{1,21}}+\frac{f_{2,12}+f_{2,22}}{f_{1,12}+f_{1,22}}\right);
											\end{align}
											and
											\begin{align}
												x(y,t)=y+\ln\left( \frac{M_{11}(\frac{i}{2})+M_{21}(\frac{i}{2})}{\left[M_{12}(\frac{i}{2})+M_{22}(\frac{i}{2})\right]a^2(\frac{i}{2})}\right) +\log{\frac{(f_{1,11}+f_{1,21})T_{0,-1}}{(f_{1,12}+f_{1,22})T_0}}.
											\end{align}
Finally,  summing up above results  gives  the following  theorem
											
											\begin{theorem}\label{last}   Let $u(x,t)$ be the solution for  the initial-value problem (\ref{ch})-(\ref{initial}) with generic data   $u_0 \in H^{4,2}(\mathbb{R})$ and scatting data $\left\lbrace  r(k),\left\lbrace ik_n,c_n\right\rbrace_{n\in \aleph}\right\rbrace$.
												And  $u^r(y,t;\tilde{\mathcal{D}}_\Lambda)$  denote  the modulation-soliton solution corresponding to  scattering data
												$\tilde{\mathcal{D}}_\Lambda=\left\lbrace  0,\left\lbrace ik_n,c_n\delta(ik_n)\right\rbrace_{n\in\Lambda}\right\rbrace$. Then there exist a large constant $T_1=T_1(\xi), \  \xi=\frac{y}{t}$,  such that   for all $t>T_1$,
												
												\noindent 1. For $\xi\in(-\infty,-1/4)$, we have asymptotic expansion
												\begin{align}
													u(x,t)=\mathcal{O}(t^{-\frac{1}{2}}),\nonumber
												\end{align}
												with
												$$ x(y,t)=y+\ln\left( \frac{M_{11}(\frac{i}{2})+M_{21}(\frac{i}{2})}{\left[M_{12}(\frac{i}{2})+M_{22}(\frac{i}{2})\right]a^2(\frac{i}{2})}\right).$$
												
												
												\noindent 2. For $\xi\in(2,+\infty)$, we have asymptotic expansion
												\begin{align}
													u(x,t)=u^{r}(y,t,\tilde{\mathcal{D}}_\Lambda)+\mathcal{O}(t^{-\frac{1}{2}}),\nonumber
												\end{align}
												with
												$$ x(y,t)=y+\ln\left( \frac{M_{11}(\frac{i}{2})+M_{21}(\frac{i}{2})}{\left[M_{12}(\frac{i}{2})+M_{22}(\frac{i}{2})\right]a^2(\frac{i}{2})}\right) +\log{\frac{(f_{1,11}+f_{1,21})T_{0,-1}}{(f_{1,12}+f_{1,22})T_0}},
												$$
												where   $u^r(y,t;\tilde{\mathcal{D}}_\Lambda)$, $f_{1,ij}(i,j=1,2 )$is shown in  (\ref{3-99})-(\ref{3-101}).
												
												
											\end{theorem}
											
											\section{Long-time asymptotics  in region  with  phase points}\label{sec4}
											As we shown in Subsection \ref{subsec2.3},
											for the cases II and III,
											there exist four and two  stationary phase points on the real axis,  denoted as $\xi_1>...>\xi_4$ and $\xi_1>\xi_2$ respectively.
											We use these phase points  to  divide the real axis $\mathbb{R}$ in the following way.
											Denote $\xi_0=+\infty$, $\xi_{n(\xi)+1}=-\infty$, and introduce some  intervals  when $j=1,...,n(\xi)$.
											For the case  $0<\xi<2$
											\begin{align}
												I_{j1}=I_{j2}=\left\{ \begin{array}{ll}
													\left( \frac{\xi_j+\xi_{j+1}}{2},\xi_j\right),\    & j\text{ is  odd },\\[10pt]
													\left(\xi_j ,\frac{\xi_j+\xi_{j-1}}{2}\right),   &j\text{ is  even },
												\end{array}\right.\label{In1}\
												I_{j3}=I_{j4}=\left\{ \begin{array}{ll}
													\left(\xi_j ,\frac{\xi_j+\xi_{j-1}}{2}\right),\    & j\text{ is  odd},\\[10pt]
													\left( \frac{\xi_j+\xi_{j+1}}{2},\xi_j\right),   &j\text{ is  even};
												\end{array}\right.
											\end{align}
											and for the case  $-1/4<\xi<0$,
											\begin{align}
												I_{j1}=I_{j2}=\left\{ \begin{array}{ll}
													\left(\xi_j ,\frac{\xi_j+\xi_{j-1}}{2}\right),\    & j\text{ is  odd},\\[10pt]
													\left( \frac{\xi_j+\xi_{j+1}}{2},\xi_j\right),   &j\text{ is  even},
												\end{array}\right.\
												I_{j3}=I_{j4}=\left\{ \begin{array}{ll}
													\left( \frac{\xi_j+\xi_{j+1}}{2},\xi_j\right),\    & j\text{ is  odd},\\[10pt]
													\left(\xi_j ,\frac{\xi_j+\xi_{j-1}}{2}\right),   &j\text{ is  even}.
												\end{array}\right. \label{In2}
											\end{align}
											For brevity, we denote
											\begin{align}
												n(\xi)=\left\{ \begin{array}{ll}
													2 , &\text{as } 0\leq\xi<2,\\[10pt]
													4,   &\text{as } -1/4<\xi<0,
												\end{array}\right.
											\end{align}
											as the number of stationary phase points.
											Denote
											\begin{align}
												L(\xi)=\left\{ \begin{array}{ll}
													(\xi_2,\xi_1),   &\text{as } 0\leq\xi<2,\\[4pt]
													(-\infty,\xi_4)\cup(\xi_{3},\xi_{2})\cup(\xi_1,+\infty),   &\text{as } -1/4<\xi<0.\\
												\end{array}\right.
											\end{align}
											In  the above formulas, we choose the principal branch of power and logarithm functions. We introduce   a  sign
											\begin{align}
												\eta(\xi,\xi_j)=\left\{ \begin{array}{ll}
													(-1)^j,   &\text{as } -1/4<\xi<0;\\
													(-1)^{j+1},   &\text{as } 0\leq\xi<2, \ \ j=1,\cdots, n(\xi).
												\end{array}\right.\label{eta}
											\end{align}
											Same as Subsection \ref{subsec3.1},  in these two case, $T(k)=\prod_{ik_n\in\Delta^+}\frac{ k+ik_n}{k-ik_n}\delta(k,\xi)$ admit a new proposition
											
\begin{Proposition}
												As $k\to \xi_j$ along any ray $\xi_j+e^{i\phi}\mathbb{R}^+$ with $|\phi|<\pi$,
												\begin{align}
													|T(k,\xi)-T_j(\xi)\left( \eta(\xi,\xi_j)(k-\xi_j)\right) ^{\eta(\xi,\xi_j) i\nu(\xi_j)}|\lesssim \parallel r\parallel_{H^{1}(\mathbb{R})}|k-\xi_j|^{1/2},\label{T-TJ}
												\end{align}
												where
												\begin{align}
&T_j(\xi)=\prod_{ik_n\in\Delta^+}\frac{ k+ik_n}{k-ik_n}e^{i\beta_j(\xi_j,\xi)},  \ j=1,...,n(\xi),\nonumber\\
&\beta_j(k,\xi)=\int_{L(\xi)}\frac{\nu(s)}{s-k}ds-\eta(\xi,\xi_j)\log\left( \eta(\xi,\xi_j)(k-\xi_j)\right) \nu(\xi_j).\nonumber
												\end{align}
											\end{Proposition}

											\begin{proof}
												Simply calculation gives that
												\begin{align*}
													|(k-\xi_j)^{i\eta(\xi,\xi_j)\nu(\xi_j)}|\leq e^{-\pi\nu(\xi_j)}=\sqrt{1+|r(\xi_j)|^2}\label{key}.
												\end{align*}
Further,
												\begin{align*}
													|\delta(k,\xi)-e^{\beta_j(\xi_j,\xi)}|&=\Big|\exp\left\lbrace i\eta(\xi,\xi_j)\log(\eta(\xi,\xi_j)(k-\xi_j)\nu(\xi_j)) \right\rbrace\left( e^{i\beta_j(k,\xi)}-e^{i\beta_j(\xi_j,\xi)}\right) \Big|.
												\end{align*}
												Then by $|\beta_j(k,\xi)-\beta_j(\xi_j,\xi)|\lesssim \parallel r\parallel_{H^{1 }(\mathbb{R})}|k-\xi_j|^{1/2}$,  we get (\ref{T-TJ}).
											\end{proof}
											
Making a transformation
$$ M^{(1)}(k)=M^J(k)T(k)^{ \sigma_3}, $$
then the RHP \ref{RHP2-2} change to
											\begin{RHP} \label{rhp10}
												Find a matrix-valued function $ M^{(1)}(k) $ which satisfies:
												
												$\blacktriangleright$ Analyticity: $M^{(1)}(k)$ is analytic in $\mathbb{C}\setminus (\mathbb{R}\bigcup\mathcal{Z})$ and has single poles;
												
												$\blacktriangleright$ Jump condition: $M^{(1)}(k)$ has continuous boundary values $M^{(1)}_\pm(k)$ on $\mathbb{R}$ and
												\begin{equation}
													M_+^{(1)}(k)=M_-^{(1)}(k)V^{(1)}(k,\xi),\hspace{0.5cm}k \in \mathbb{R},
												\end{equation}
												where
												\begin{equation}
													V^{(1)}(k,\xi)=\left\{\begin{array}{ll}
														\left(\begin{array}{cc}
															1 & 0 \\
															-T^2\overline{r}e^{2it\theta}&  1
														\end{array}\right)\left(\begin{array}{cc}
															1 & T^{-2}re^{-2it\theta} \\
															0 &  1
														\end{array}\right),   &\text{as }k\in L(\xi),\\[12pt]
														\left(\begin{array}{cc}
															1 & \frac{r(k)e^{-2it\theta}T_-^{-2}(k)}{1-|r(k)|^2}\\
															0 & 1
														\end{array}\right)\left(\begin{array}{cc}
															1 & 0 \\
															\frac{-\bar{r}(k)e^{2it\theta}T_+^{2}(k)}{1-|r(k)|^2} & 1
														\end{array}\right),   &\text{as } k\in 	\mathbb{R}\setminus L(\xi);
													\end{array}\right.\label{jumpv1}
												\end{equation}
												
												$\blacktriangleright$ Asymptotic behaviors:
												\begin{align}
													M^{(1)}(k) = I+\mathcal{O}(k^{-1}),\hspace{0.5cm}k \rightarrow \infty;
												\end{align}
												
												$\blacktriangleright$ Residue conditions: $M^{(1)}(k)$ has simple poles at each point in $ \mathcal{Z} $ with:
												\begin{align}
													\res_{k=ik_n}M^{(1)}(k)=\left\{\begin{array}{ll}\lim_{k\rightarrow ik_n}M^{(1)}(k)\left(\begin{array}{cc}
															0 & 0\\
															c_n^{-1} e^{2it\theta{(ik_n)}}[(\frac{1}{T})^{'}(ik_n)]^{-2}& 0
														\end{array}\right),   &\text{as } ik_n\in 	\Delta^+,\\[12pt]
														\lim_{k\rightarrow ik_n}M^{(1)}(k)\left(\begin{array}{cc}
															0 & c_nT^{-2}e^{-2it\theta{(ik_n)}}\\
															0 & 0
														\end{array}\right),   &\text{as } ik_n\in 	\Delta^-\bigcup\bigwedge;\\[12pt]
													\end{array}
													\right. \nonumber
												\end{align}
																		\begin{align}
																			\res_{k=-ik_n}M^{(1)}(k)=\left\{\begin{array}{ll}\lim_{k\rightarrow -ik_n}M^{(1)}(k)\left(\begin{array}{cc}
																					0 &  c_n^{-1} e^{-2it\theta{(-ik_n)}}[T^{'}(-ik_n)]^{-2}\\
																					0 & 0
																				\end{array}\right),   &\text{as } ik_n\in 	\Delta^+,\\[12pt]
																				\lim_{k\rightarrow -ik_n}M^{(1)}(k)\left(\begin{array}{cc}
																					0 & 0\\
																					c_nT^{2}e^{2it\theta{(-ik_n)}} & 0
																				\end{array}\right),   &\text{as } ik_n\in 	\Delta^-\bigcup\bigwedge;\\[12pt]
																			\end{array}
																			\right.\nonumber
																		\end{align}
																	\end{RHP}

																	\begin{figure}[h]
																		\begin{center}
																			\subfigure[]{
																				\begin{tikzpicture}
																					\draw[white, fill=white]
																					(-5,0.8)--(-5,-0.8)--(-2.5,0)--(0,0.8)--(0,-0.8);
																					\draw[white, fill=white]
																					(5,0.8)--(5,-0.8)--(2.5,0)--(0,0.8)--(0,-0.8);
																					\draw[->,dashed,red,thick](-5.5,0)--(5.5,0)node[right]{ \textcolor{black}{Re$k$}};
																					\draw[->,dashed,red,thick](0,-1.5)--(0,1.5)node[right]{\textcolor{black}{Im$k$}};
																					\coordinate (I) at (0,0);
																					\fill (I) circle (1pt) node[below] {$0$};
																					\draw(-5,0.8)--(-2.5,0);
																					\draw[-<](-5,0.8)--(-3.5,0.32)node[above]{\scriptsize$\Sigma_{24}$};
																					\draw(-5,-0.8)--(-2.5,0.0);
																					\draw[-<](-5,-0.8)--(-3.5,-0.32)node[below]{\scriptsize$\Sigma_{23}$};
																					\draw(-2.5,0)--(0,-0.8);
																					\draw(-2.5,0)--(0,0.8);
																					\draw[->](-2.5,0)--(-1.5,-0.32)node[below]{\scriptsize$\Sigma_{22}$};
																					\draw[->](-2.5,0)--(-1.5,0.32)node[above]{\scriptsize$\Sigma_{21}$};
																					\draw(5,0.8)--(2.5,0);
																					\draw[-<](5,0.8)--(3.5,0.32)node[above]{\scriptsize$\Sigma_{14}$};
																					\draw(5,-0.8)--(2.5,0.0);
																					\draw[-<](5,-0.8)--(3.5,-0.32)node[below]{\scriptsize$\Sigma_{13}$};
																					\draw(2.5,0)--(0,-0.8);
																					\draw(2.5,0)--(0,0.8);
																					\draw[->](2.5,0)--(1.5,-0.32)node[below]{\scriptsize$\Sigma_{12}$};
																					\draw[->](2.5,0)--(1.5,0.32)node[above]{\scriptsize$\Sigma_{11}$};
																					\draw[->](0,0)--(0,0.8)node[above]{\scriptsize$\Sigma_{2+}'$};
																					\draw[->](0,0)--(0,-0.8)node[below]{\scriptsize$\Sigma_{2-}'$};
																					\coordinate (R) at (-2.5,0);
																					\fill (R) circle (1pt) node[blue,below] {$\xi_2$};
																					\coordinate (T) at (2.5,0);
																					\fill (T) circle (1pt) node[blue,below] {$\xi_1$};
																					\coordinate (e) at (-4,-0.1);
																					\fill (e) circle (0pt) node[above] {\tiny$\Omega_{24}$};
																					\coordinate (e1) at (-4,-0.1);
																					\fill (e1) circle (0pt) node[below] {\tiny$\Omega_{23}$};
																					\coordinate (r) at (-0.5,-0.1);
																					\fill (r) circle (0pt) node[above] {\tiny$\Omega_{21}$};
																					\coordinate (r1) at (-0.5,-0.1);
																					\fill (r1) circle (0pt) node[below] {\tiny$\Omega_{22}$};
																					\coordinate (e) at (4,-0.1);
																					\fill (e) circle (0pt) node[above] {\tiny$\Omega_{14}$};
																					\coordinate (e1) at (4,-0.1);
																					\fill (e1) circle (0pt) node[below] {\tiny$\Omega_{13}$};
																					\coordinate (r) at (0.5,-0.1);
																					\fill (r) circle (0pt) node[above] {\tiny$\Omega_{11}$};
																					\coordinate (r1) at (0.5,-0.1);
																					\fill (r1) circle (0pt) node[below] {\tiny$\Omega_{12}$};
																				\end{tikzpicture}
																				\label{case1}}
																			\subfigure[]{
																				\begin{tikzpicture}
																					\draw[white, fill=white] (-5,0.5)--(-5,-0.5)--(-4,0)--(-2.5,-0.6)--(-1,0)--(0,-0.5)--(1,0)--(2.5,-0.6)--(4,0)--(5,-0.5)--(5,0.5)--(4,0)--(2.5,0.6)--(1,0)--(0,0.5)--(-1,0)--(-2.5,0.6)--(-4,0)--(-5,0.5);
																					\draw(-4,0)--(-5,0.5)node[above]{\scriptsize$\Sigma_{41}$};
																					\draw[->](-4,0)--(-4.5,0.25);
																					\draw(-4,0)--(-2.5,0.6);
																					\draw[->](-4,0)--(-3.25,-0.3)node[below]{\scriptsize$\Sigma_{43}$};
																					\draw(-4,0)--(-5,-0.5)node[below]{\scriptsize$\Sigma_{42}$};
																					\draw[->](-4,0)--(-3.25,0.3)node[above]{\scriptsize$\Sigma_{44}$};
																					\draw(-4,0)--(-2.5,-0.6);
																					\draw[->](-4,0)--(-4.5,-0.25);
																					\draw(-1,0)--(0,0.5);
																					\draw[->](-1,0)--(-0.5,0.25)node[above]{\scriptsize$\Sigma_{31}$};
																					\draw(-1,0)--(-2.5,0.6);
																					\draw[->](-1,0)--(-1.75,-0.3)node[below]{\scriptsize$\Sigma_{33}$};
																					\draw(-1,0)--(0,-0.5);
																					\draw[->](-1,0)--(-1.75,0.3)node[above]{\scriptsize$\Sigma_{34}$};
																					\draw(-1,0)--(-2.5,-0.6);
																					\draw[->](-1,0)--(-0.5,-0.25)node[below]{\scriptsize$\Sigma_{32}$};
																					\draw[->,dashed,red,thick](-5.5,0)--(5.5,0) node[right] {  \textcolor{black}{Re$k$}};
																					\draw[->,dashed,red,thick](0,-1.5)--(0,1.5)  node[right]{  \textcolor{black}{ Im$k$}};
																					\draw(1,0)--(0,0.5);
																					\draw[->](1,0)--(0.5,0.25)node[above]{\scriptsize$\Sigma_{21}$};
																					\draw(1,0)--(2.5,0.6);
																					\draw[->](1,0)--(1.75,-0.3)node[below]{\scriptsize$\Sigma_{23}$};
																					\draw(1,0)--(0,-0.5);
																					\draw[->](1,0)--(1.75,0.3)node[above]{\scriptsize$\Sigma_{24}$};
																					\draw(1,0)--(2.5,-0.6);
																					\draw[->](1,0)--(0.5,-0.25)node[below]{\scriptsize$\Sigma_{22}$};
																					\draw(4,0)--(5,0.5)node[above]{\scriptsize$\Sigma_{11}$};
																					\draw[->](4,0)--(4.5,0.25);
																					\draw(4,0)--(2.5,0.6);
																					\draw[->](4,0)--(3.25,-0.3)node[below]{\scriptsize$\Sigma_{13}$};
																					\draw(4,0)--(5,-0.5)node[below]{\scriptsize$\Sigma_{12}$};
																					\draw[->](4,0)--(3.25,0.3)node[above]{\scriptsize$\Sigma_{14}$};
																					\draw(4,0)--(2.5,-0.6);
																					\draw[->](4,0)--(4.5,-0.25);
																					\draw[->](2.5,0)--(2.5,0.6)node[above]{\scriptsize$\Sigma_{1+}'$};
																					\draw[->](2.5,0)--(2.5,-0.6)node[below]{\scriptsize$\Sigma_{1-}'$};
																					\draw[->](-2.5,0)--(-2.5,0.6)node[above]{\scriptsize$\Sigma_{3+}'$};
																					\draw[->](-2.5,0)--(-2.5,-0.6)node[below]{\scriptsize$\Sigma_{3-}'$};
																					\draw[->](0,0)--(0,0.5)node[above]{\scriptsize$\Sigma_{2+}'$};
																					\draw[->](0,0)--(0,-0.5)node[below]{\scriptsize$\Sigma_{2-}'$};
																					\coordinate (I) at (0,0);
																					\fill (I) circle (1pt) node[below] {$0$};
																					\coordinate (A) at (-4,0);
																					\fill (A) circle (1pt) node[blue,below] {$\xi_4$};
																					\coordinate (b) at (-1,0);
																					\fill (b) circle (1pt) node[blue,below] {$\xi_3$};
																					\coordinate (e) at (4,0);
																					\fill (e) circle (1pt) node[blue,below] {$\xi_1$};
																					\coordinate (f) at (1,0);
																					\fill (f) circle (1pt) node[blue,below] {$\xi_2$};
																					\coordinate (ke) at (4.7,0.1);
																					\fill (ke) circle (0pt) node[below] {\tiny$\Omega_{12}$};
																					\coordinate (k1e) at (4.7,-0.1);
																					\fill (k1e) circle (0pt) node[above] {\tiny$\Omega_{11}$};
																					\coordinate (le) at (3,0.1);
																					\fill (le) circle (0pt) node[below] {\tiny$\Omega_{13}$};
																					\coordinate (l1e) at (3,-0.1);
																					\fill (l1e) circle (0pt) node[above] {\tiny$\Omega_{14}$};
																					\coordinate (n2) at (0.27,0.1);
																					\fill (n2) circle (0pt) node[below] {\tiny$\Omega_{22}$};
																					\coordinate (n12) at (0.27,-0.1);
																					\fill (n12) circle (0pt) node[above] {\tiny$\Omega_{21}$};
																					\coordinate (m2) at (2.25,0.1);
																					\fill (m2) circle (0pt) node[below] {\tiny$\Omega_{23}$};
																					\coordinate (m12) at (2.25,-0.1);
																					\fill (m12) circle (0pt) node[above] {\tiny$\Omega_{24}$};
																					\coordinate (k) at (-4.7,0.1);
																					\fill (k) circle (0pt) node[below] {\tiny$\Omega_{42}$};
																					\coordinate (k1) at (-4.7,-0.1);
																					\fill (k1) circle (0pt) node[above] {\tiny$\Omega_{41}$};
																					\coordinate (l) at (-3,0.1);
																					\fill (l) circle (0pt) node[below] {\tiny$\Omega_{43}$};
																					\coordinate (l1) at (-3,-0.1);
																					\fill (l1) circle (0pt) node[above] {\tiny$\Omega_{44}$};
																					\coordinate (n) at (-0.27,0.1);
																					\fill (n) circle (0pt) node[below] {\tiny$\Omega_{32}$};
																					\coordinate (n1) at (-0.27,-0.1);
																					\fill (n1) circle (0pt) node[above] {\tiny$\Omega_{31}$};
																					\coordinate (m) at (-2.2,0.1);
																					\fill (m) circle (0pt) node[below] {\tiny$\Omega_{33}$};
																					\coordinate (m1) at (-2.2,-0.1);
																					\fill (m1) circle (0pt) node[above] {\tiny$\Omega_{34}$};
																				\end{tikzpicture}
																				\label{case2}}
																			\caption{\footnotesize   The stationary phase points of phase function $\theta(k)$.
																				The  figure   (a)  corresponds  to  two phase points for the case   $0< \xi<2$;    The  figure   (b) corresponds  to  four  phase points for  the case   $-1/4<\xi<0$. $\Sigma_{ij}$ separate
																				complex plane $\mathbb{C}$ into some  sectors denoted by $\Omega_{ij}$.}
																			\label{FigOmig}
																		\end{center}
																	\end{figure}
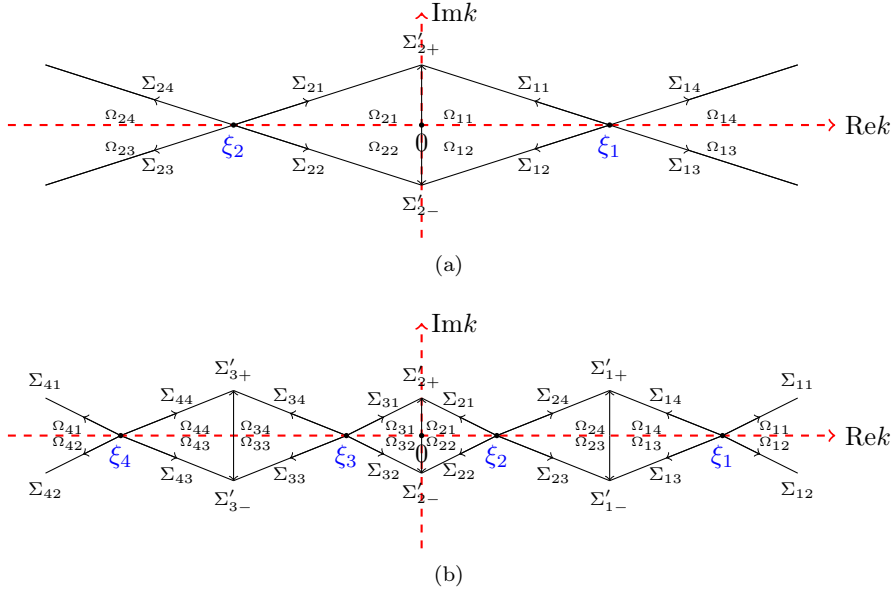
																	We show deformation contours and some notations in Figure \ref{FigOmig}. In addition, let
																	\begin{align}
																		&\tilde{\Sigma}(\xi)=\left( \underset{j=1,..,n(\xi)}{\underset{h=1,...,4,}{\cup}}\Sigma_{jh}\right) \cup\left(\underset{j}{\cup}\Sigma'{j\pm} \right),\hspace{0.5cm}\nonumber\hspace{0.5cm}
																		\Sigma^{(2)}(\xi)=\tilde{\Sigma}(\xi),\nonumber\\
																		&\Omega(\xi)=\underset{j=1,..,n(\xi)}{\underset{h=1,...,4,}{\cup}}\Omega_{jh},\hspace{0.5cm} \Omega_\pm(\xi)=\mathbb{C}\setminus\Omega. \nonumber
																	\end{align}
																	For convenience, denote $\Sigma_{n(\xi)\pm}'=\emptyset$ when $-1/4 <  \xi<0 $ and $\Sigma_{1\pm}'=\emptyset$ when $ 0\leq\xi< 2$.
																	And $\varphi>0$ is an fixed sufficiently small angle  achieving following conditions:\\
																	1.  each $\Omega_{jh}$ doesn't intersect $\left\lbrace k\in\mathbb{C}|\text{Im }\theta(k)=0\right\rbrace  $,\\
																	2. $\tan\varphi<\frac{1}{\xi_{n(\xi)/2}-\xi_{n(\xi)/2+1}}$.

																	We introduce  a new unknown function
																	\begin{equation}
																		R^{(2)}(k,\xi)=\left\{\begin{array}{ccc}
																			\left(\begin{array}{cc}
																				1 & R_{hj}(k,\xi)e^{-2it\theta}\\
																				0 & 1
																			\end{array}\right), & k\in \Omega_{hj},\ h=1,...,n(\xi),\ j=1,3;\\
																			\\
																			\left(\begin{array}{ccc}
																				1 & 0\\
																				R_{hj}(k,\xi)e^{2it\theta} & 1
																			\end{array}\right),  &k\in \Omega_{hj},\ h=1,...,n(\xi),\ j=2,4;\\
																			\\
																			I,  &elsewhere;\\
																		\end{array}\right.\label{R(2)1}
																	\end{equation}
																	where  the functions $R_{hj}(k,\xi)$, $h=1,...,n(\xi)$, $j=1,2,3,4$  are defined in the following Proposition.
																	\begin{Proposition}\label{proR1}
																		As  for   $\xi \in[0,2)$  and   $\xi \in(-1/4,0)$,  the functions $R_{hj}$: $\bar{\Omega}_{hj}\to \mathbb{C}$, $j=1,2,3,4$, $h=1,...,n(\xi)$ have boundary values as follows:
																		\begin{align}
																			&R_{h1}(k,\xi)=\Bigg\{\begin{array}{ll}
																				p_{h1}(k,\xi)T(k)^{2}, & k\in I_{h1},\\
																				p_{h1}(\xi_h,\xi)T_h(\xi)^{2}\left( \eta(\xi,\xi_h)(k-\xi_h)\right)^{2i\eta(\xi,\xi_h)\nu(\xi_h)},  &k\in \Sigma_{h1},\\
																			\end{array} \\
																			&R_{h2}(k,\xi)=\Bigg\{\begin{array}{ll}
																				p_{h2}(\xi_h,\xi)T_h(\xi)^{-2}\left( \eta(\xi,\xi_h)(k-\xi_h)\right)^{-2i\eta(\xi,\xi_h)\nu(\xi_h)},  &k\in \Sigma_{h2},\\
																				p_{h2}(k,\xi)T(k)^{-2}, &k\in  I_{h2},\\
																			\end{array} \\
																			&R_{h3}(k,\xi)=\Bigg\{\begin{array}{ll}
																				p_{h3}(k,\xi)T_-(k)^{-2}, &k\in I_{h3}, \\
																				p_{h3}(\xi_h,\xi)T_h(\xi)^{-2}\left( \eta(\xi,\xi_h)(k-\xi_h)\right)^{-2i\eta(\xi,\xi_h)\nu(\xi_h)}, &k\in \Sigma_{h3},\\
																			\end{array} \\
																			&R_{h4}(k,\xi)=\Bigg\{\begin{array}{ll}
																				p_{h4}(\xi_h,\xi)T_h(\xi)^{2}\left( \eta(\xi,\xi_h)(k-\xi_h)\right) ^{\eta(\xi,\xi_h)2i\nu(\xi_h)},  &k\in \Sigma_{h4},\\
																				p_{h4}(k,\xi)T_+(k)^{2}, &k\in I_{h4},\\
																			\end{array}
																		\end{align}	
																		where $I_{hj}$ are specified in (\ref{In1})-(\ref{In2}) and
																		\begin{align}
																			&p_{h1}(k,\xi)=-\bar{r}(k),\hspace{0.5cm}p_{h3}(k,\xi)=\dfrac{r(k) }{1-|r(k)|^2},\\
																			&p_{h2}(k,\xi)=-r(k),\hspace{0.6cm}p_{h4}(k,\xi)=\dfrac{\bar{r}(k) }{1-|r(k)|^2}.
																		\end{align}
																		The functions  $R_{hj}$  have the following properties:
																		\begin{align}
																			&|R_{hj}(k,\xi)|\lesssim \sin^2(k_0\arg(k-\xi_h))+ \left(1+ \text{Re}(k)^2\right) ^{-1/2}, \ \text{for all $k\in \Omega_{hj}$},\label{R}\\
																			&|\bar{\partial}R_{hj}(k,\xi)|\lesssim|p_{hj}'(\text{Re}k)|+|k-\xi_h|^{-1/2}, \text{for all $k\in \Omega_{hj}$.} \label{dbarRj3}\\
																			&\bar{\partial}R_{hj}(k,\xi)=0,\hspace{0.5cm}\text{if } k, at\  elsewhere.
																		\end{align}
																	\end{Proposition}
																	\begin{proof}
																		The proof is similar as \cite{YF3}.	
																	\end{proof}

Making a transformation
																	\begin{align}
																		 M^{(2)}(k) =M^{(21)}(k)  R^{(2)}(k,\xi),
																	\end{align}
then  the  RHP \ref{rhp10} becomes    to
							
							\begin{RHP}\label{rhp11}
								Find a matrix valued function  $ M^{(2)}(k)$ with following properties:
								
								$\blacktriangleright$ Analyticity:  $M^{(2)}(k)$ is continuous in $\mathbb{C}\backslash(\Sigma^{(2)}\cup\mathcal{Z})$, where $\Sigma^{(2)}=\cup_{j=1}^4\Sigma_j$;
								
								$\blacktriangleright$ Jump condition: $M^{(2)}(k)$ has continuous boundary values $M^{(2)}_\pm(k)$ on $\Sigma^{(2)}$ and
								\begin{equation}
									M^{(2)}_{+}(k)=M^{(2)}_{-}(k)V^{(2)}(k),\hspace{0.5cm}k \in \Sigma^{(2)},
								\end{equation}
								where
								\begin{align}
									V^{(2)}(k,\xi)=\Bigg\{\begin{array}{ll}
										R^{(2)}(k)^{-1}|_{\Sigma_{1}\cup\Sigma_{3}}&\text{as } 	k\in\Sigma_{1}\cup\Sigma_{3};\\[12pt]
										R^{(2)}(k)|_{\Sigma_{2}\cup\Sigma_{4}}&\text{as } 	k\in\Sigma_{2}\cup\Sigma_{4};\\[12pt]
									\end{array}\label{jumpv2}
								\end{align}
								
								$\blacktriangleright$ Asymptotic behaviors:
								\begin{align}
									M^{(2)}(k) = I+\mathcal{O}(k^{-1}),\hspace{0.5cm}k \rightarrow \infty;
								\end{align}
								
								$\blacktriangleright$ $\bar{\partial}$-Derivative: For $k\in\mathbb{C}$
								we have
								\begin{align}
									\bar{\partial}M^{(2)}(k)=M^{(2)}(k)\bar{\partial}R^{(2)}(k),
								\end{align}
where for  $\xi \in[0,2)$  and   $\xi \in(-1/4,0)$
																	\begin{equation}
																		\bar{\partial}R^{(2)}(k,\xi)=\left\{\begin{array}{lll}
																			\left(\begin{array}{cc}
																				0 & \bar{\partial}R_{hj}(k,\xi)e^{-2it\theta}\\
																				0 & 0
																			\end{array}\right), & k\in \Omega_{hj},\ h=1,...,n(\xi),\ j=1,3;\\
																			\\
																			\left(\begin{array}{cc}
																				0& 0\\
																				\bar{\partial}R_{hj}(k,\xi)e^{2it\theta} & 0
																			\end{array}\right),  &k\in \Omega_{hj},\ h=1,...,n(\xi),\ j=2,4;\\
																			\\
																			0,  &elsewhere;\\
																		\end{array}\right.\label{DbarR(2)1}
																	\end{equation}
																	And the jump condition become
																	\begin{equation}
																		V^{(2)}(k)=\left\{
																		\begin{array}{ll}
																			R^{(2)}(k)^{-1}|_{\Sigma_{h2}\cup\Sigma_{h4}}&\text{as } 	k\in\Sigma_{h2}\cup\Sigma_{h4};\\[12pt]
																			R^{(2)}(k)|_{\Sigma_{h1}\cup\Sigma_{h3}}&\text{as } 	k\in\Sigma_{h1}\cup\Sigma_{h3};\\[12pt]
																			R^{(2)}(k)|_{\Sigma_{h\ (3\pm1)/2}}R^{(2)}(k)^{-1}|_{\Sigma_{(h-1)\ (3\pm1)/2}}&\text{as } 	k\in\Sigma_{h\pm}',\ h \text{ is even} ;\\[12pt]
																			R^{(2)}(k)|_{\Sigma_{h\ (7\pm1)/2}}R^{(2)}(k)^{-1}|_{\Sigma_{(h-1)\ (7\pm1)/2}}&\text{as } 	k\in\Sigma_{h\pm}',\ h \text{ is odd} ;\\[12pt]
																		\end{array}\right. \label{jumpv21}
																	\end{equation}
								
								$\blacktriangleright$ Residue conditions: $M^{(2)}$(k) has the same residue conditions with the RHP \ref{rhp10}.
								
							\end{RHP}

																	The analysis of $M^{(2)}$  is similar as Subsection \ref{subsec3.3}. But $M^{(R)}$ has more differences.
																	Compared with $\xi \in(2,+\infty)\cup(-\infty,-1/4)$,  it can be found that its jump matrix $V^{(2)}$ has additional portion on $\Sigma_{hj}$ and $\Sigma_{h\pm}$ in the case of $\xi \in(-1/4,2)$.
																	So this case is  more difficult  to be dealt  with.   And denote $ U(n(\xi))$ as the union set of neighborhood of $\xi_h$ for $h=1,...,n(\xi)$
																	\begin{equation}
																		U(n(\xi))=\underset{h=1,...,n(\xi)}{\cup}U_{\xi_h},\ U_{\xi_h}= \left\lbrace k:|k-\xi_h|\leq \delta_1 \right\rbrace, \nonumber
																	\end{equation}
																	where $ n(\xi)=2, 4$ correspond to  two cases  $\xi \in(0,2)$  and $\xi \in(-1/4,0)$ respectively.
																Here, the $\delta_1$ is chosen so that the small circumferences do not intersect each other and do not touch the imaginary axis.  Then    jump matrix $V^{(2)}(z)$  outside of $ U(n(\xi))$  has the following  estimates.
																	\begin{Proposition}\label{prov2}
																		For $1\leq p\leq+\infty$, there exist a positive constant $K_p$ relied on $p$ satisfies that the jump matrix $V^{(2)}(k)$ defined in (\ref{jumpv21}) admits
																		\begin{align}
																			\parallel V^{(2)}(k)-I\parallel_{L^p(\Sigma_{hj}\setminus U_{\xi_h} )}= \mathcal{O}( e^{-K_pt}), \ \  t\to\infty,
																		\end{align}
																		for $h=1,...,n(\xi)$ and $j=1,...,4$.
																		And   there also exist a positive constant $K_p'$ relied on $p$ satisfies that the jump matrix $V^{(2)}$  admits
																		\begin{align}
																			\parallel V^{(2)}(k)-I\parallel_{L^p(\Sigma_{h\pm}')}= \mathcal{O}( e^{-K_p't}),   \ \  t\to\infty,
																		\end{align}
																		for $h=1,...,n(\xi)$.
																	\end{Proposition}
																	
																	This proposition means that the jump matrix $V^{(2)}(k)$    uniformly goes to  $I$  on     $\tilde{\Sigma}\setminus U(n(\xi))$.
																	So outside the $U(n(\xi))$ there is only exponentially small error (in $t$) by completely ignoring the jump condition of  $M^{(R)}(k)$.
																	And this proposition enlightens  us to construct the solution $M_{R}(k)$ as follows:
																	\begin{equation}
																		M_{R}(k)=\left\{\begin{array}{ll}
																			E(k,\xi)M^{err}(k) M^{\Lambda}(k):=E(k,\xi)M^{(r)}(k)& k\notin U(n(\xi))\\
																			E(k,\xi)M^{err}(k) M^{\Lambda}(k)M^{lo}(k):=E(k,\xi)M^{(r)}(k)M^{lo}(k)  &k\in U(n(\xi))\\
																		\end{array}\right..\label{transm4}
																	\end{equation}

																	\subsection{Local model near phase points}
																	We first analyze $M^{lo}$. Denote a new contour $\Sigma^{(0)}= (\underset{h=1,2,3,4}{\underset{j=1,..,n(\xi),}{\cup}}\Sigma_{jh} )\cap U(n(\xi))$  shown  in the  Figure \ref{sigma0}.
																	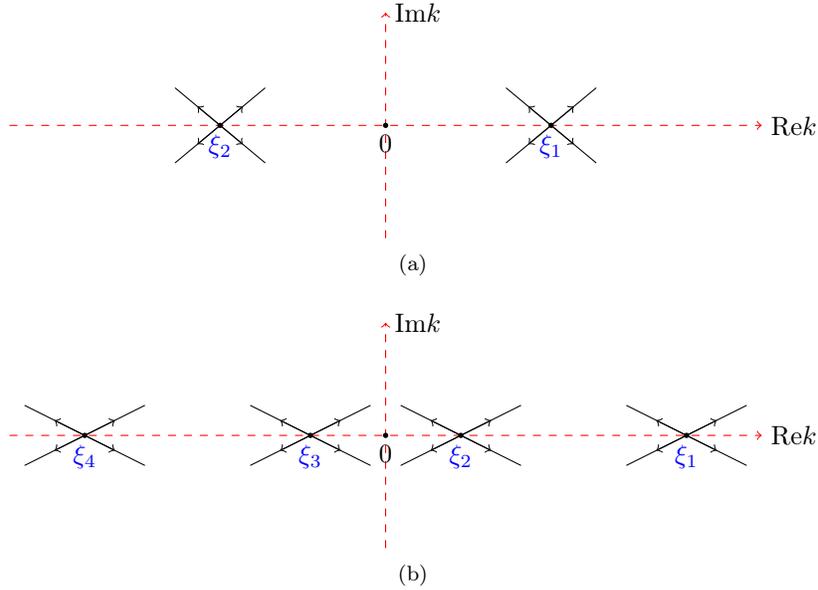
\begin{figure}[h]
																		\begin{center}
																			\subfigure[]{
																				\begin{tikzpicture}
																					\draw[->,dashed,red](-5,0)--(5,0)node[right]{\textcolor{black}{ Re$k$}};
																					\draw[->,dashed,red](0,-1.5)--(0,1.5)node[right]{\textcolor{black}{Im$k$}};
																					\coordinate (I) at (0,0);
																					\fill (I) circle (1pt) node[below] {$0$};
																					\draw(-2.2,0)--(-1.6,0.5);
																					\draw[->](-2.2,0)--(-2.5,0.25);
																					\draw(-2.2,0)--(-1.6,-0.5);
																					\draw[->](-2.2,0)--(-1.9,-0.25);
																					\draw(-2.2,0)--(-2.8,0.5);
																					\draw[->](-2.2,0)--(-1.9,0.25);
																					\draw(-2.2,0)--(-2.8,-0.5);
																					\draw[->](-2.2,0)--(-2.5,-0.25);
																					\draw(2.2,0)--(1.6,0.5);
																					\draw[->](2.2,0)--(2.5,0.25);
																					\draw(2.2,0)--(1.6,-0.5);
																					\draw[->](2.2,0)--(1.9,-0.25);
																					\draw(2.2,0)--(2.8,0.5);
																					\draw[->](2.2,0)--(1.9,0.25);
																					\draw(2.2,0)--(2.8,-0.5);
																					\draw[->](2.2,0)--(2.5,-0.25);
																					\coordinate (A) at (-2.2,0);
																					\fill (A) circle (1pt) node[blue,below] {$\xi_2$};
																					\coordinate (b) at (2.2,0);
																					\fill (b) circle (1pt) node[blue,below] {$\xi_1$};
																				\end{tikzpicture}
																				\label{si1}}
																			\subfigure[]{
																				\begin{tikzpicture}
																					\draw(-4,0)--(-4.8,0.4);
																					\draw[->](-4,0)--(-4.4,0.2);
																					\draw(-4,0)--(-3.2,0.4);
																					\draw[->](-4,0)--(-3.6,-0.2);
																					\draw(-4,0)--(-4.8,-0.4);
																					\draw[->](-4,0)--(-3.6,0.2);
																					\draw(-4,0)--(-3.2,-0.4);
																					\draw[->](-4,0)--(-4.4,-0.2);
																					\draw(-1,0)--(-0.2,0.4);
																					\draw[->](-1,0)--(-0.6,0.2);
																					\draw(-1,0)--(-1.8,0.4);
																					\draw[->](-1,0)--(-1.4,-0.2);
																					\draw(-1,0)--(-0.2,-0.4);
																					\draw[->](-1,0)--(-1.4,0.2);
																					\draw(-1,0)--(-1.8,-0.4);
																					\draw[->](-1,0)--(-0.6,-0.2);
																					\draw[->,dashed,red](-5,0)--(5,0)node[right]{ \textcolor{black}{Re$k$}};
																					\draw[->,dashed,red](0,-1.5)--(0,1.5)node[right]{\textcolor{black}{Im$k$}};
																					\draw(1,0)--(0.2,0.4);
																					\draw[->](1,0)--(0.6,0.2);
																					\draw(1,0)--(0.2,-0.4);
																					\draw[->](1,0)--(1.4,-0.2);
																					\draw(1,0)--(1.8,0.4);
																					\draw[->](1,0)--(1.4,0.2);
																					\draw(1,0)--(1.8,-0.4);
																					\draw[->](1,0)--(0.6,-0.2);
																					\draw(4,0)--(4.8,0.4);
																					\draw[->](4,0)--(4.4,0.2);
																					\draw(4,0)--(3.2,0.4);
																					\draw[->](4,0)--(3.6,-0.2);
																					\draw(4,0)--(4.8,-0.4);
																					\draw[->](4,0)--(3.6,0.2);
																					\draw(4,0)--(3.2,-0.4);
																					\draw[->](4,0)--(4.4,-0.2);
																					\coordinate (I) at (0,0);
																					\fill (I) circle (1pt) node[below] {$0$};
																					\coordinate (A) at (-4,0);
																					\fill (A) circle (1pt) node[blue,below] {$\xi_4$};
																					\coordinate (b) at (-1,0);
																					\fill (b) circle (1pt) node[blue,below] {$\xi_3$};
																					\coordinate (e) at (4,0);
																					\fill (e) circle (1pt) node[blue,below] {$\xi_1$};
																					\coordinate (f) at (1,0);
																					\fill (f) circle (1pt) node[blue,below] {$\xi_2$};
																					\coordinate (c) at (-2,0);
																				\end{tikzpicture}
																				\label{si2}}
																			\caption{  Jump contours $\Sigma^{(0)}$  of  $ M^{lo}(k)$.  The figures (a) and (b)   are corresponding to the   cases    $0\leq \xi<2$  and   $-1/4<\xi<0$ respectively.}
																			\label{sigma0}
																		\end{center}
																	\end{figure}
																	Consider the following RH problem:
																	\begin{RHP}\label{rhp11}
																		Find a matrix-valued function  $ M^{lo}(k)$ with following properties:
																		
																		$\blacktriangleright$ Analyticity: $M^{lo}(k)$ is analytical  in $\mathbb{C}\setminus \Sigma^{(0)} $;
																		
																		$\blacktriangleright$ Jump condition: $M^{lo}(k)$ has continuous boundary values $M^{lo}_\pm(k)$ on $\Sigma^{(0)}$ and
																		\begin{equation}
																			M^{lo}_+(k)=M^{lo}_-(k)V^{(2)}(k),\hspace{0.5cm}k \in \Sigma^{(0)};\label{jump6}
																		\end{equation}
																		
																		$\blacktriangleright$ Asymptotic behaviors:
																		\begin{align}
																			M^{lo}(k) =& I+\mathcal{O}(k^{-1}),\hspace{0.5cm}k \rightarrow \infty;
																		\end{align}
																	\end{RHP}	
																	This RH problem  only has jump conditions  without   poles. The  matrix $V^{(2)}(k)$ is a  upper/lower matrix with 1 on the  diagonal. For $h=1,...,n(\xi)$, we denote
																	\begin{align}
																		w_{hj}(k)=\left\{\begin{array}{lll}
																			\left(\begin{array}{cc}
																				0 & -R_{hj}(k,\xi)e^{-2it\theta}\\
																				0 & 0
																			\end{array}\right), &k\in \Sigma_{hj},j=1,3,\\[10pt]
																			\left(\begin{array}{cc}
																				0 & 0\\
																				R_{hj}(k,\xi)e^{2it\theta} & 0
																			\end{array}\right),  &k\in \Sigma_{hj},j=2,4.
																		\end{array}\right.
																	\end{align}
																	Then $V^{(2)}(k)=I\pm w_{hj}(k)$ for $k\in \Sigma_{hj}$. Here, it takes $"-"$ when $j$ is an even number. Besides, let
																	\begin{align}
																		&\Sigma^{(0)}_h=\bigcup_{j=1,...,4}\Sigma_{hj}, \ \ w_h(k)=\sum_{j=1,...,4} w_{hj}(k),\\
																		&w_{hj}^\pm(k)=w_{hj}(k)|_{\mathbb{C}^\pm}, \ \   w_h^\pm(k)=w_h(k)|_{\mathbb{C}^\pm}, \ \ w^\pm(k)=w(k)|_{\mathbb{C}^\pm}.
																	\end{align}
																	Recall the Cauchy projection operator $C_\pm$    on $\Sigma^{(2)}$
																	\begin{equation}
																		C_{\pm}(f)(s)=\lim_{z\to \Sigma^{(2)}_\pm}\frac{1}{2\pi i}\int_{\Sigma^{(2)}}\dfrac{f(s)}{s-k}ds,
																	\end{equation}
																	by which, we further define  operator
																	\begin{align}
																		C_w(f)=C_+(fw^-)+C_-(fw^+),\hspace{0.5cm}C_{w_h}(f)=C_+(fw_h^-)+C_-(fw_h^+).
																	\end{align}
																	Then $C_w=\sum_{h=1}^{n(\xi)}C_{w_h}$.
By simple calculation, we have 																	
\begin{lemma}
																		The matrix functions $w_{hj}$ defined above admits the following  estimation
																		\begin{align}
																			\parallel w_{hj}\parallel_{L^p(\Sigma_{hj})}=\mathcal{O}(t^{-1/p}),\ 1\leq p<+\infty.
																		\end{align}
																	\end{lemma}
																
This lemma implies that  the  RHP \ref{rhp11} admits a unique solution, which  can be written as
																	\begin{align}
																		M^{lo}(k) =I+\frac{1}{2\pi i}\int_{\Sigma^{(0)}}\frac{(I-C_w)^{-1}I\ w}{s-k}ds.
																	\end{align}
																	As shown in \cite{YF3}, the  contributions of every crosses $\Sigma^{(0)}_h$ can be separated out. So, as $t\to +\infty$, we consider to reduce the  RHP \ref{rhp11} to a model RH problem  whose solution can be given explicitly in terms of parabolic cylinder
																	functions on every contour $\Sigma^{(0)}_h$ respectively. And we only take $\Sigma^{(0)}_1$ as an example, the model near other  critical points can be  constructed similar. We denote $\hat{\Sigma}^{(0)}_1$ as the contour $\{k=\xi_1+le^{\pm\varphi i},\ l\in\mathbb{R}\}$ oriented from $\Sigma^{(0)}_1$, and $\hat{\Sigma}_{1j}$ is the extension of $\Sigma_{1j}$  respectively. And for $k$ near $\xi_1$, rewrite phase function as
																	\begin{align}
																		\theta(k)=\theta(\xi_1)+(k-\xi_1)^2\theta''(\xi_1)+\mathcal{O}((k-\xi_1)^3),
																	\end{align}
where $\theta''(\xi_1)>0$ for $\xi\in[0,2)$  and $\theta''(\xi_1)<0$ for $\xi\in(-1/4,0)$.
																	
																	In order to motivate the model, let $\zeta = \zeta(k)$ denote the rescaled local variable
																	\begin{align}
																		\zeta(k)=t^{1/2}\sqrt{4\eta(\xi,\xi_1)\theta''(\xi_1)}(k-\xi_1),
																	\end{align}
																where  $\eta(\xi,\xi_1)$ is defined in (\ref{eta}). This change of variable maps $U_{\xi_1}$ to an expanding neighborhood of $\zeta= 0$. Additionally, let
																	\begin{align}
																		r_{\xi_1}=-r(\xi_1)T_1(\xi)^{2}e^{2it\theta(\xi_1)}\eta(\xi,\xi_1)^{2i\nu(\xi_1)}\exp\left\lbrace -i\eta(\xi,\xi_1)\nu(\xi_1)\log \left( 4t\theta''(\xi_1)\eta(\xi,\xi_1)\right) \right\rbrace ,
																	\end{align}
																	with $|r_{\xi_1}|=|r(\xi_1)|$.
																	In the above expression, the complex powers are defined by choosing the branch of
																	the logarithm with  $-\pi< \arg \zeta < \pi$ in the cases $\xi\in[0,2)$, and the branch of the logarithm with $0 < \arg \zeta < 2\pi$ in the case $\xi\in(-1/4,0)$. Then Theorem A.1-A.6 in \cite{HG2009} proved that
																	\begin{align}
																		M^{lo,1}(k)=I+\frac{t^{-1/2}}{k-\xi_1} \left(\begin{array}{cc}
																			0 & \tilde{\beta}^1_{12}\\
																			\tilde{\beta}^1_{21} & 0
																		\end{array}\right)+\mathcal{O}(t^{-1}),\label{asyMpc}
																	\end{align}
																	where $M^{lo,1}(k)$ is the solution of the model RHP near the phase point $\xi_1$.

																	For the model around other stationary phase points, it  also admits
																	\begin{align}
																		M^{lo,h}(k)=I+\frac{t^{-1/2}}{k-\xi_h} \left(\begin{array}{cc}
																			0 & \tilde{\beta}^h_{12}\\
																			\tilde{\beta}^h_{21} & 0
																		\end{array}\right)+\mathcal{O}(t^{-1}),\label{asyMpck}
																	\end{align}
																	for $h=2,...,n(\xi)$.
																	Either $\xi\in(-1/4,0)$, $h$ is  odd number or $\xi\in[0,2)$, $h$ is  even number,
																	\begin{align}
																		r_{\xi_h}=-r(\xi_h)T_h(\xi)^{2}e^{2it\theta(\xi_h)}\exp\left\lbrace -i\nu(\xi_h)\log \left( 4t\theta''(\xi_h)\right) \right\rbrace,
																	\end{align}
																	and
																	\begin{align}
																		&\tilde{\beta}^h_{21}=\frac{\sqrt{2\pi}e^{\frac{5}{2}\pi\nu(\xi_h)}e^{-\frac{7\pi}{4} i}}{r_{\xi_h}\Gamma(-i\nu(\xi_h))},\hspace{0.5cm}\tilde{\beta}^h_{21}\tilde{\beta}^h_{12}=-\nu(\xi_h),\nonumber\\
																		&|\tilde{\beta}^h_{21}|=-\frac{\nu(\xi_h)}{(1-|r(\xi_h)|^2)^3},\nonumber\\
																		&\arg(\tilde{\beta}^h_{21})=\frac{5}{2}\pi\nu(\xi_h)-\frac{7\pi}{4} i-\arg r_{\xi_h}-\arg \Gamma(-i\nu(\xi_h))	.\nonumber
																	\end{align}
																	Either $\xi\in(-1/4,0)$, $h$ is  even number or $\xi\in[0,2)$, $h$ is  odd number,
																	\begin{align}
																		r_{\xi_h}=-r(\xi_h)T_h(\xi)^{2}e^{2it\theta(\xi_h)}\exp\left\lbrace -i\nu(\xi_h)\log \left(- 4t\theta''(\xi_h)\right) \right\rbrace,
																	\end{align}
																	and
																	\begin{align}
																		&\tilde{\beta}^h_{21}=\frac{\sqrt{2\pi}e^{\frac{\pi}{2}\nu(\xi_h)}e^{-\frac{\pi}{4} i}}{r_{\xi_h}\Gamma(i\nu(\xi_h))},\hspace{0.5cm}\tilde{\beta}^h_{21}\tilde{\beta}^h_{12}=-\nu(\xi_h),\nonumber\\
																		&|\tilde{\beta}^h_{21}|=-\frac{\nu(\xi_h)}{1-|r(\xi_h)|^2},\nonumber\\
																		&\arg(\tilde{\beta}^h_{21})=\frac{\pi}{2}\nu(\xi_h)-\frac{\pi}{4} i-\arg r_{\xi_h}-\arg \Gamma(i\nu(\xi_h)).\nonumber
																	\end{align}
																	We finally obtain
																	\begin{Proposition}\label{asymlo}
																		As $t\to+\infty$,
																		\begin{align}
																			M^{lo}(k)=I+t^{-1/2}\sum_{ h=1 }^{n(\xi)}\frac{A_h(\xi)}{k-\xi_h} +\mathcal{O}(t^{-1}),
																		\end{align}
																		where
																		\begin{align}
																			A_h(\xi)=\left(\begin{array}{cc}
																				0 & \tilde{\beta}^h_{12}\\
																				\tilde{\beta}^h_{21} & 0
																			\end{array}\right).
																		\end{align}
																	\end{Proposition}
																	\subsection{Error estimate through small norm RH problem}
																	The last step is to consider the following   RH   problem  for the matrix function  $E(k;\xi)$.
																	\begin{RHP}\label{e12}
																		Find a matrix-valued function $E(k;\xi)$  admitting following properties:
																		
																		$\blacktriangleright$ Analyticity: $E(k;\xi)$ is analytical  in $\mathbb{C}\setminus  \Sigma^{(E)} $, where
																		$$\Sigma^{(E)}= \partial U(n(\xi))\cup
																		(\Sigma^{(2)}\setminus U(n(\xi));$$

																		$\blacktriangleright$ Asymptotic behaviors:
																		\begin{align}
																			&E(k;\xi) \sim I+\mathcal{O}(k^{-1}),\hspace{0.5cm}|k| \rightarrow \infty;
																		\end{align}

																		$\blacktriangleright$ Jump condition: $E(k;\xi)$ has continuous boundary values $E_\pm(k;\xi)$ on $\Sigma^{(E)}$ satisfying
																		$$E_+(k;\xi)=E_-(k;\xi)V^{(E)}(k),$$
																		where the jump matrix $V^{(E)}(k)$ is given by
																		\begin{equation}
																			V^{(E)}(k)=\left\{\begin{array}{llll}
																				M^{(r)}(k)V^{(2)}(k)M^{(r)}(k)^{-1}, & k\in \Sigma^{(2)}\setminus U(n(\xi)),\\[4pt]
																				M^{(r)}(k)M^{lo}(k)M^{(r)}(k)^{-1},  & k\in \partial U(n(\xi)),
																			\end{array}\right. \label{deVE}
																		\end{equation}
																		which is  shown in  Figure \ref{figE}.
																	\end{RHP}

																	\begin{figure}[H]
																		\centering
																		\subfigure[]{
																			\begin{tikzpicture}
																				\draw[dashed](-5,0)--(5,0)node[right]{ Re$k$};
																				\draw[dashed](0,-1)--(0,1)node[right]{ Im$k$};
																				\coordinate (I) at (0,0);
																				\fill (I) circle (1pt) node[below] {$0$};
																				\coordinate (e) at (2,0);
																				\fill (e) circle (1pt) node[below] {$\xi_1$};
																				\coordinate (f) at (-2,0);
																				\draw[thick,red](-2,0) circle (0.4);
																				\fill (f) circle (1pt) node[below] {$\xi_2$};
																				\draw[thick,red](2,0) circle (0.4);
																				\draw(-1.629,0.149)--(0,0.8);
																				\draw[->](-1.629,0.149)--(-1,0.4);
																				\draw(1.629,0.149)--(0,0.8);
																				\draw[->](1.629,0.149)--(1,0.4);
																				\draw(-1.629,-0.149)--(0,-0.8);
																				\draw[->](-1.629,-0.149)--(-1,-0.4);
																				\draw(1.629,-0.149)--(0,-0.8);
																				\draw[->](1.629,-0.149)--(1,-0.4);
																				\draw(0,0.8)--(0,-0.8);
																				\draw[->](0,0)--(0,0.8);
																				\draw[->](0,0)--(0,-0.8);
																				\draw(-2.371,0.149)--(-4,0.8);
																				\draw[->](-2.371,0.149)--(-3,0.4);
																				\draw(-2.371,-0.149)--(-4,-0.8);
																				\draw[->](-2.371,-0.149)--(-3,-0.4);
																				\draw(2.371,0.149)--(4,0.8);
																				\draw[->](2.371,0.149)--(3,0.4);
																				\draw(2.371,-0.149)--(4,-0.8);
																				\draw[->](2.371,-0.149)--(3,-0.4);
																			\end{tikzpicture}
																		}
																		\subfigure[]{
																			\begin{tikzpicture}
																				\draw(4.35,0.18)--(5,0.5);
																				\draw[->](4.35,0.18)--(4.5,0.25);
																				\draw(3.64,0.15)--(2.5,0.6);
																				\draw[-<](3.64,-0.15)--(3.25,-0.3);
																				\draw(4.35,-0.18)--(5,-0.5);
																				\draw[-<](3.64,0.15)--(3.25,0.3);
																				\draw(3.64,-0.15)--(2.5,-0.6);
																				\draw[->](4.35,-0.18)--(4.5,-0.25);
																				\draw(-4.35,0.18)--(-5,0.5);
																				\draw[-<](-4.35,0.18)--(-4.5,0.25);
																				\draw(-3.64,0.15)--(-2.5,0.6);
																				\draw[->](-3.64,-0.15)--(-3.25,-0.3);
																				\draw(-4.35,-0.18)--(-5,-0.5);
																				\draw[->](-3.64,0.15)--(-3.25,0.3);
																				\draw(-3.64,-0.15)--(-2.5,-0.6);
																				\draw[-<](-4.35,-0.18)--(-4.5,-0.25);
																				\draw(-0.64,0.15)--(0,0.5);
																				\draw[->](-0.64,0.15)--(-0.4,0.28);
																				\draw(-1.35,0.16)--(-2.5,0.6);
																				\draw[-<](-1.35,-0.16)--(-1.75,-0.31);
																				\draw(-0.64,-0.15)--(0,-0.5);
																				\draw[-<](-1.35,0.16)--(-1.75,0.31);
																				\draw(-1.35,-0.16)--(-2.5,-0.6);
																				\draw[->](-0.64,-0.15)--(-0.4,-0.28);
																				\draw[dashed](-5,0)--(5,0)node[right]{ Re$k$};
																				\draw(0.64,0.15)--(0,0.5);
																				\draw[-<](0.64,0.15)--(0.4,0.28);
																				\draw(1.35,0.16)--(2.5,0.6);
																				\draw[->](1.35,-0.16)--(1.75,-0.31);
																				\draw(0.64,-0.15)--(0,-0.5);
																				\draw[->](1.35,0.16)--(1.75,0.31);
																				\draw(1.35,-0.16)--(2.5,-0.6);
																				\draw[-<](0.64,-0.15)--(0.4,-0.28);
																				\draw[->](2.5,0)--(2.5,0.6);
																				\draw[->](2.5,0)--(2.5,-0.6);
																				\draw[->](-2.5,0)--(-2.5,0.6);
																				\draw[->](-2.5,0)--(-2.5,-0.6);
																				\draw[->](0,0)--(0,0.5);
																				\draw[->](0,0)--(0,-0.5);
																				\coordinate (I) at (0,0);
																				\fill (I) circle (1pt) node[below] {$0$};
																				\coordinate (A) at (-4,0);
																				\fill (A) circle (1pt) node[below] {$\xi_4$};
																				\coordinate (b) at (-1,0);
																				\fill (b) circle (1pt) node[below] {$\xi_3$};
																				\coordinate (e) at (4,0);
																				\fill (e) circle (1pt) node[below] {$\xi_1$};
																				\coordinate (f) at (1,0);
																				\draw[thick,red](1,0) circle (0.4);
																				\fill (f) circle (1pt) node[below] {$\xi_2$};
																				\draw[thick,red](4,0) circle (0.4);
																				\draw[thick,red](-1,0) circle (0.4);
																				\draw[thick,red](-4,0) circle (0.4);
																			\end{tikzpicture}
																		}
																		\caption{  The jump contour $\Sigma^{(E)}$ for the $E(k;\xi)$. The red circles are $U(n(\xi))$. }
																		\label{figE}
																	\end{figure}
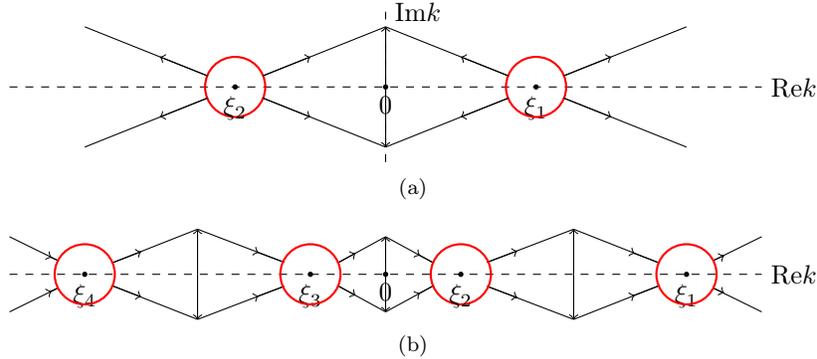
																	
																	By using   {Proposition \ref{prov2}}, we have the following estimates
																	\begin{equation}
																		\parallel V^{(E)}(k)-I\parallel_p\lesssim\left\{\begin{array}{llll}
																			\exp\left\{-tK_p\right\},  & k\in \Sigma_{hj}\setminus U(n(\xi)),\\[6pt]
																			\exp\left\{-tK_p'\right\},   & k\in \Sigma'_{h\pm}.
																		\end{array}\right. \label{VE-I}
																	\end{equation}
																	
																	For $k\in \partial U(n(\xi))$,  $M^{(r)}(k)$ is bounded,  so via Proposition \ref{asymlo},  we find that
																	\begin{equation}
																		| V^{(E)}(k)-I|=   \big|M^{(r)}(k)(M^{lo}(k)-I)M^{(r)}(k)^{-1} \big| = \mathcal{O}(t^{-1/2}).\label{VE}
																	\end{equation}
																	Therefore,    the   existence and uniqueness  of  the RHP \ref{e12} is  shown  by using  a  small-norm RH problem \cite{RN9,RN10}.  Moreover, according to Beals-Coifman theory,
																	the solution of  the RHP \ref{e12}  can be given by
																	\begin{equation}
																		E(k;\xi)=I+\frac{1}{2\pi i}\int_{\Sigma^{(E)}}\dfrac{\left( I+\varpi(s)\right) (V^{(E)}(s)-I)}{s-k}ds,\label{Ez}
																	\end{equation}
																	where the $\varpi\in L^\infty(\Sigma^{(E)})$ is the unique solution of the following equation
																	\begin{equation}
																		(1-C_E)\varpi=C_E\left(I \right).
																	\end{equation}
																	And  $C_E$ is  an integral operator:$L^\infty(\Sigma^{(E)})\to L^2(\Sigma^{(E)})$ defined by
																	\begin{equation}
																		C_E(f)(k)=C_-\left( f(V^{(E)}(k) -I)\right) ,
																	\end{equation}
																	where  $C_-$ is the usual Cauchy projection operator on $\Sigma^{(E)}$.
																	
																	By  (\ref{VE}),   we have
																	\begin{equation}
																		\parallel C_E\parallel\leq\parallel C_-\parallel \parallel V^{(E)}(k)-I\parallel_2 \lesssim \mathcal{O}(t^{-1/2}),
																	\end{equation}
																	which implies that  $1-C_E$ is invertible  for   sufficiently large $t$.    So  $\varpi$  exists and is unique.
																	Besides,
																	\begin{equation}
																		\parallel \varpi\parallel_{L^\infty(\Sigma^{(E)})}\lesssim\dfrac{\parallel C_E\parallel}{1-\parallel C_E\parallel}\lesssim t^{-1/2}.\label{normrho}
																	\end{equation}
																	In order to reconstruct the solution $u(y,t)$ of (\ref{ch}), we need the asymptotic behavior of $E(k;\xi)$ as $k\to\frac{i}{2}$ and the long time asymptotic behavior of $E(\frac{i}{2})$. When we estimate its  asymptotic behavior, from (\ref{Ez}) and (\ref{VE-I}) we only need to consider the calculation on $\partial U(n(\xi))$ because it  approach zero exponentially on other boundary.
																	\begin{Proposition}\label{asyE}
																		As $k\to \frac{i}{2}$, it follows that
																		\begin{align}
																			E(k;\xi)=E(\frac{i}{2})+E_1(k-\frac{i}{2})+\mathcal{O}((k-\frac{i}{2})^2),
																		\end{align}
																		where
																		\begin{align}
&E(\frac{i}{2})=I+t^{-1/2}H^{(0)}+\mathcal{O}(t^{-1}),\label{E0t}\\
&				H^{(0)}=\sum_{ h=1 }^{n(\xi)}\frac{-1}{\xi_h-\frac{i}{2}} M^{(r)}(\xi_h)A_h(\xi)M^{(r)}(\xi_h)^{-1}.\nonumber
																		\end{align}
 And
																		\begin{equation}
																			E_1=\frac{1}{2\pi i}\int_{\Sigma^{(E)}}\dfrac{\left( I+\varpi(s)\right) (V^{(E)}-I)}{(k-\frac{i}{2})^2}ds,
																		\end{equation}
which admits  long time asymptotic behavior
																		\begin{equation}
																			E_1=t^{-1/2}H^{(1)}+\mathcal{O}(t^{-1}),\label{E1t}
																		\end{equation}
where
																		\begin{align}
																			H^{(1)}=\sum_{ h=1 }^{n(\xi)}\frac{-1}{(\xi_h-\frac{i}{2})^2} M^{(r)}(\xi_h)A_h(\xi)M^{(r)}(\xi_h)^{-1}.
																		\end{align}
																	\end{Proposition}
																	\subsection{$\bar{\partial}$-problem analysis}
																	Although in case $-1/4<\xi<2$, $M^{(3)}$ admits same expression
																	in RHP \ref{RHP9}, the analysis region and the function $R^{(2)}$ are different. Similar with the process in Subsection \ref{subsec3.3}, we first give the estimation of Im$\theta$:
																	\begin{lemma}\label{theta2}
																		There exist a constant $c(\xi)>0$ relied on $\xi \in(-1/4, 2)$ that the imaginary part of phase function (\ref{Reitheta}) $\text{\rm Im }\theta(k)$ have the following estimation for $i=1,...,n(\xi)$:
																		\begin{align}
																			&\text{\rm Im }\theta(k)\geq c(\xi)\text{\rm Im} k\frac{\text{Re}^2k-\xi_i^2}{1+4\text{Re}^2k} ,\hspace{0.5cm} \text{as }k\in\Omega_{i1}, \Omega_{i3};\\
																			&\text{\rm Im }\theta(k)\leq -c(\xi)\text{\rm Im} k\frac{\text{Re}^2k-\xi_i^2}{1+4\text{Re}^2k}  ,\hspace{0.5cm} \text{as }k\in\Omega_{i2}, \Omega_{i4}.
																		\end{align}	
																	\end{lemma}
																	\begin{proof}
																		Take $i=1$ as an example. We consider the region $\Omega_{11}$. Let $k=u+\xi_1+vi$, then
																		\begin{align*}
																			{\rm Im} \theta(k)=v\left( \xi+2\frac{4|k|^2-1}{|1+4k^2|^2}\right),
																		\end{align*}
				in which  $\frac{4|k|^2-1}{|1+4k^2|^2}$ is a monotone creasing function of $v$ for given $u$, so
																		\begin{align*}
																			{\rm Im} \theta(k)\geq& v\left( \xi+2\frac{4(u+\xi_1)^2-1}{(1+4(u+\xi_1)^2)^2}\right) \\
																			&=v\left\lbrace 8\left[ (u+\xi_1)^2-\xi_1^2\right] \frac{3-16\xi_1^2(u+\xi_1)^2+4(\xi_1^2+(u+\xi_1)^2)}{\left[1+4(u+\xi_1)^2 \right]^2(1+4\xi_1^2)^2}\right\rbrace .
																		\end{align*}
																		The last equality is from $\xi=2\frac{1-4\xi_1^2}{(1+4\xi_1^2)^2}$. And In the product above, the last item has nonzero upper and lower bound for $u\geq0$. Therefore,
																		\begin{align*}
																			{\rm Im} \theta(k)\gtrsim v\frac{(u+\xi_1)^2-\xi_1^2}{1+4(u+\xi_1)^2}.
																		\end{align*}
																	\end{proof}
																	\begin{Proposition}
																		The Cauchy integral operator $ J: L^{\infty}(\mathbb{C}) \rightarrow L^{\infty}(\mathbb{C}) $. When $\xi\in(-1/4,2)$, for sufficiently large $t$, the operator $J$ is a small parametrization. And
																		\begin{equation}
																			|| J ||_{ L^{\infty}\rightarrow L^{\infty} }  \leqslant ct^{-\frac{1}{4}}.
																		\end{equation}
																		Therefore $(1-J)^{-1}$ exists, and thus the operator equation (\ref{opers}) has a solution. In addition, 	 the solution $M^{(3)}(k)$  of  $\bar{\partial}$-problem  admits the following estimation
																		\begin{align}
																			\parallel M^{(3)}(\frac{i}{2})-I\parallel=\parallel\frac{1}{\pi}\iint_\mathbb{C}\dfrac{M^{(3)}(s)W^{(3)} (s)}{s-\frac{i}{2}}dA(s)\parallel\lesssim t^{-3/4}.
																		\end{align}
																		As $k\to \frac{i}{2}$, $M^{(3)}(k)$ has asymptotic expansion
																		\begin{equation}
																			M^{(3)}(k)=M^{(3)}(\frac{i}{2})+M^{(3)}_1(x,t)(k-\frac{i}{2})+\mathcal{O}((k-\frac{i}{2})^{2}),
																		\end{equation}
																		where $M^{(3)}_1(x,t) $ is a $k$-independent coefficient with
																		\begin{equation}
																			M^{(3)}_1(x,t)=\frac{1}{\pi}\iint_C\dfrac{M^{(3)}(s)W^{(3)} (s)}{(s-\frac{i}{2})^2}dA(s),
																		\end{equation}
																		and $M^{(3)}_1(x,t)$  satisfies
																		\begin{equation}
																			|M^{(3)}_1(x,t)|\lesssim t^{-3/4}.
																		\end{equation}
																	\end{Proposition}
																	
																	\subsection{Long time asymptotic behaviors}

In the case II and case III, $T(k)=\delta(k)$. $M(k)$ can be expressed as
																	\begin{align}
																		M(k)=&  F(y,t,k)M^{(3)}(k)E(k,\xi)M^{err}(k)\delta^{-\sigma_3}(k,\xi),
																	\end{align}
											where
																	\begin{align}
																		&F(y,t,k)= \left(\begin{array}{cc}
																			F_{11}+F_{12}(k-\frac{i}{2})& F_{21}+F_{22}(k-\frac{i}{2})\\
																			F_{31}+F_{32}(k-\frac{i}{2})& F_{41}+F_{42}(k-\frac{i}{2})
																		\end{array}\right)+	\left(\begin{array}{cc}
																			G_{11}+G_{12}(k-\frac{i}{2})& G_{21}+G_{22}(k-\frac{i}{2})\\
																			G_{31}+G_{32}(k-\frac{i}{2})& G_{41}+G_{42}(k-\frac{i}{2})
																		\end{array}\right)t^{-\frac{1}{2}}\nonumber\\
																		&+\mathcal{O}((k-\frac{i}{2})^2)+\mathcal{O}(t^{-\frac{3}{4}}).\nonumber
																	\end{align}
																	In order to facilitate calculation, denote
																	\begin{align}
																		A(y,t,k)=\left(\begin{array}{cc}
																			F_{11}\delta_{0}^{-1}+\left(F_{11}\delta_2+F_{12}\delta_{0}^{-1}\right)(k-\frac{i}{2})& 	 F_{21}\delta_{0}+\left(F_{22}\delta_0+F_{21}\delta_{1}\right)(k-\frac{i}{2})\\
																			F_{31}\delta_{0}^{-1}+\left(F_{31}\delta_2+F_{32}\delta_{0}^{-1}\right)(k-\frac{i}{2})& 	 F_{41}\delta_{0}+\left(F_{42}\delta_0+F_{41}\delta_{1}\right)(k-\frac{i}{2})
																		\end{array}\right)+\mathcal{O}((k-\frac{i}{2})^2),\nonumber
																	\end{align}
																	where
																	\begin{align}
																		&\delta_0=e^{\frac{1}{2\pi i}\int_{L(\xi)}\frac{\log(1-|r|^2)}{s-\frac{i}{2}}ds};\delta_1=\frac{-\delta_0}{2\pi i}\int_{L(\xi)}\frac{\log(1-|r|^2)}{(s-\frac{i}{2})^2}ds;\delta_2=\frac{\delta_0^{-1}}{2\pi i}\int_{L(\xi)}\frac{\log(1-|r|^2)}{(s-\frac{i}{2})^2}ds;\nonumber
																	\end{align}
																	\begin{align}
																		B(y,t,k)=&\left(\begin{array}{cc}
																			(F_{11}H_{11}^{(0)}+F_{21}H_{21}^{(0)})\delta_{0}^{-1}+h_{11}(k-\frac{i}{2})& 	(F_{11}H_{12}^{(0)}+F_{21}H_{22}^{(0)})\delta_{0}+h_{12}(k-\frac{i}{2})\\
																			(F_{31}H_{11}^{(0)}+F_{41}H_{21}^{(0)})\delta_{0}^{-1}+h_{21}(k-\frac{i}{2})& 	(F_{31}H_{12}^{(0)}+F_{41}H_{22}^{(0)})\delta_{0}+h_{22}(k-\frac{i}{2})
																		\end{array}\right)\nonumber \\
																		+ &\left(\begin{array}{cc}
																			G_{11}\delta_{0}^{-1}+\left(G_{12}\delta_{0}^{-1}+G_{11}\delta_{2}\right)(k-\frac{i}{2})& 	 G_{21}\delta_{0}+\left(G_{22}\delta_{0}+G_{21}\delta_{1}\right)(k-\frac{i}{2})\\
																			G_{31}\delta_{0}^{-1}+\left(G_{32}\delta_{0}^{-1}+G_{31}\delta_{2}\right)(k-\frac{i}{2})& 	 G_{41}\delta_{0}+\left(G_{42}\delta_{0}+G_{41}\delta_{1}\right)(k-\frac{i}{2})
																		\end{array}\right)+\mathcal{O}((k-\frac{i}{2})^2),\nonumber
																	\end{align}
																	\begin{align}
																		h_{11}=&(F_{11}H_{11}^{(0)}+F_{21}H_{21}^{(0)})\delta_2+(F_{12}H_{11}^{(0)}+F_{11}H_{11}^{(1)}+F_{21}H_{21}^{(1)}+F_{22}H_{21}^{(0)})\delta_{0}^{-1};\nonumber \\
																		h_{21}=&(F_{31}H_{11}^{(0)}+F_{41}H_{21}^{(0)})\delta_2+(F_{32}H_{11}^{(0)}+F_{31}H_{11}^{(1)}+F_{41}H_{21}^{(1)}+F_{42}H_{21}^{(0)})\delta_{0}^{-1};\nonumber \\
																		h_{12}=&(F_{11}H_{12}^{(0)}+F_{21}H_{22}^{(0)})\delta_1+(F_{12}H_{12}^{(0)}+F_{11}H_{12}^{(1)}+F_{21}H_{22}^{(1)}+F_{22}H_{22}^{(0)})\delta_{0};\nonumber\\
																		h_{22}=&(F_{31}H_{12}^{(0)}+F_{41}H_{22}^{(0)})\delta_1+(F_{32}H_{12}^{(0)}+F_{31}H_{12}^{(1)}+F_{41}H_{22}^{(1)}+F_{42}H_{22}^{(0)})\delta_{0}.\nonumber
																	\end{align}
																	Then $M(k)$ is abbreviated to
																	\begin{align}
																		M(k)=A(y,t,k)+B(y,t,k)t^{-\frac{1}{2}}+\mathcal{O}(t^{-\frac{3}{4}}).
																	\end{align}
																	The reconstruction formula (\ref{newu}) recover the solution
\begin{align}
																		 u(x,t) =\left(\frac{a_1}{c_1}+\frac{a_2}{c_2}\right)+\left(\frac{b_1c_1-a_1d_1}{c_1^2}+\frac{b_2c_2-a_2d_2}{c_2^2}\right)t^{-\frac{1}{2}}+\mathcal{O}(t^{-\frac{3}{4}}),\label{pie}
																	\end{align}
																	where
																	\begin{align}
																		a_1=&(F_{11}+F_{31})\delta_2+(F_{12}+F_{32})\delta_0^{-1};\hspace{0.2cm}a_2=(F_{21}+F_{41})\delta_1+(F_{22}+F_{42})\delta_0;\nonumber \\
																		b_1=&h_{11}+h_{21}+(G_{12}+G_{32})\delta_0^{-1}+(G_{11}+G_{31})\delta_2;
																		b_2=h_{12}+h_{22}+(G_{22}+G_{42})\delta_0+(G_{21}+G_{41})\delta_1;\nonumber\\
																		c_1=&(F_{11}+F_{31})\delta_0^{-1};\hspace{0.2cm}c_2=(F_{21}+F_{41})\delta_0;\nonumber \\
																		d_1=&(F_{11}H_{11}^{(0)}+F_{21}H_{21}^{(0)}+F_{31}H_{11}^{(0)}+F_{41}H_{21}^{(0)}+G_{11}+G_{31})\delta_0^{-1};\nonumber \\
																		d_2=&(F_{11}H_{12}^{(0)}+F_{21}H_{22}^{(0)}+F_{31}H_{12}^{(0)}+F_{41}H_{22}^{(0)}+G_{21}+G_{41})\delta_0. \nonumber
																	\end{align}																	
Finally,  summing up above results  gives  the following  theorem
			\begin{theorem}\label{last1} Under the same condition with Theorem \ref{last},    the solution  of the Cauchy problem
(\ref{ch})-(\ref{initial})  admits the  asymptotic expansion (\ref{pie})
 in the region   $\xi\in(-1/4,2)$.

																	\end{theorem}

																	\section{Long-time asymtotics in transition regions}\label{sec5}
																	In this section, we still use the $\bar{\partial}$-methods to consider  the Painlesv\'e asymptotics for the solution  of the Cauchy problem
(\ref{ch})-(\ref{initial}) in two  transition regions.
																	\subsection{Decomposition of RH problem}
																	For $\xi=2$, we take the region  I:  $0<|\xi-2|t^{\frac{2}{3}}<C$   as an example to analyze  the Painlesv\'e asymptotics. And the region II:  $0<|\xi+\frac{1}{4}|t^{\frac{2}{3}}<C$  corresponds to the case  $\xi<-\frac{1}{4}$ can be analyzed by the same way.
																	
																	We  first change the RHP \ref{RHP2-2} into the RH problem for  $N(k)$ by  the following deformation																	 \begin{align}
																		N(k)=M(k)\tilde{a}(k)^{-\sigma_3}=\left\{ \begin{array}{ll}
																			\left(\frac{\Phi_{-}^{(1)}(k)}{a(k)},  \Phi_{+}^{(2)} (k)\right),   &\text{as } k\in \mathbb{C}^+,\\[12pt]
																			\left( \Phi_{+}^{(1)}(k),\frac{\Phi_{-}^{(2)}(k)}{\overline{a(\bar{k})}}\right)  , &\text{as }k\in \mathbb{C}^-,\\
																		\end{array}\right. \label{trans1}
																	\end{align}							
																	where  $\tilde{a}(k)=\left\{ \begin{array}{ll}
																		a(k),   &\text{as } k\in \mathbb{C}^+,\\[12pt]
																		\frac{1}{\overline{a(\bar{k})}}  , &\text{as }k\in \mathbb{C}^-.\\
																	\end{array}\right.$
																
																	 Further, in
																	   order to    eliminate singularity $k=1$  for   $N(k) $  further making
																transformation
																\begin{align} \left(I-\frac{1}{k-1}\sigma_1\right)N(k) =\left(I-\frac{1}{k-1}\sigma_1\left(N^J(y,t,1)\right)^{-1}\right)N^J(y,t,k),
																\end{align}
																then  $N^{J}(k)$ satisfies the following   RH problem.

																	\begin{RHP}\label{RHPp}
																		Find a matrix-valued function $	N^{J}(k)\triangleq N^{J}(k;y,t)$ which satisfies:
																		
																		$\blacktriangleright$ Analyticity: $N^{J}(k)$ is analytic in $\mathbb{C}\setminus \left(\mathbb{R} \cup \mathcal{Z}  \right)$;
																		
																		$\blacktriangleright$ Jump condition: $N^{J}(k)$ has continuous boundary values $N^{J}_\pm(k)$ on $\mathbb{R}$ and
																		\begin{align}
																			N^{J}_+(k)=N^{J}_-(k)V_N(k), \label{N}
																		\end{align}	
																		where
\begin{align}
																			 V_N(k)= \left(\begin{array}{cc}
																				1-|r_p(k)|^2	 & r_p(k)e^{-2it\theta}\\
																				-\overline{r_p(k)}e^{2it\theta}&  1
																			\end{array}\right), \nonumber
																		\end{align}	
and																		the reflection coefficients   by $r_p(k)=\frac{b(k)}{\overline{a(\bar{k})}}$.
																		
																		$\blacktriangleright$ Asymptotic behaviors:
																		\begin{align}
																			&N^{J}(k) = I+\mathcal{O}(k^{-1}),\hspace{0.5cm}k \rightarrow \infty,\label{jianjin}
																		\end{align}
																		
																		$\blacktriangleright$ Residue conditions: $N^{J}(k)$ has simple poles at each point in $ \mathcal{Z} $ with:
																		\begin{align}
																			&\res_{k=ik_n}N^{J}(k)=\lim_{k\to ik_n}N(k)\left(\begin{array}{cc}
																				0 & 0\\
																				c_{n,p}e^{2it\theta(ik_n)}& 0
																			\end{array}\right),\\
																			&\res_{k=-ik_n}N^{J}(k)=\lim_{k\to -ik_n}N(k)\left(\begin{array}{cc}
																				0 &c_{n,p}e^{2it\theta(ik_n)}\\
																				0& 0
																			\end{array}\right),
																		\end{align}
																		where $c_{n,p}=-\frac{b(-ik_n)}{a^{\prime}(ik_n)}$.
																	\end{RHP}
																	
																	Since $V_N(k)$ have two kinds  of decompositions
																	\begin{align}
																		V_N(k)&=\left(\begin{array}{cc}
																			1 &  r_p(k) e^{-2it\theta(k)} \\
																			0& 1
																		\end{array}\right)\left(\begin{array}{cc}
																			1 &0\\
																			-\overline{r_p(\bar{k}))} e^{2it\theta(k)} & 1
																		\end{array}\right) \label{V1}\\
																		&=\left(\begin{array}{cc}
																			1 &0\\
																			\frac{-\overline{r_p(\bar{k}))} e^{2it\theta(k)} }{1-|r_p|^2}& 1
																		\end{array}\right)(1-|r_p|^2)^{\sigma_3}\left(\begin{array}{cc}
																			1 & \frac{  r_p(k) e^{-2it\theta(k)} }{1-|r_p|^2}\\
																			0& 1
																		\end{array}\right),\label{V_Nc}
																	\end{align}	
where  (\ref{V1}) is used to the  case I  and  (\ref{V_Nc})  to the  case II.
																	
																	From the expression of $\theta(k)$, the saddle points can be calculated as
$$\kappa_0^2(\xi)=\frac{\sqrt{1+4\xi}-1-\xi}{4\xi}<C.$$

																	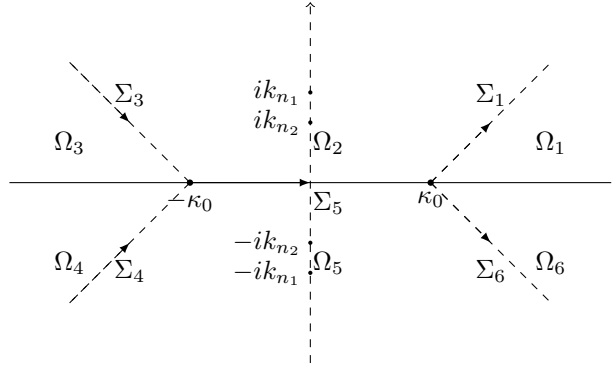
\begin{figure}
																		\centering
																		\begin{tikzpicture}[scale=0.8]
																			\draw[-](-2,0)--(2,0);
																			\draw[dashed,->](0,-3)--(0,3);
																			\draw[-](-5,0)--(-2,0);
																			\draw[-](2,0)--(5,0);
																			\coordinate (a) at (2,0);
																			\fill (a) circle (1.5pt) node[below] {$\kappa_0$};
																			\coordinate (b) at (-2,0);
																			\fill (b) circle (1.5pt) node[below] {$-\kappa_0$};
																			\draw[dashed,-](-2,0)--(-4,2);
																			\draw[dashed,-latex](-4,2)--(-3,1);
																			\draw[dashed,-](-2,0)--(-4,-2);
																			\draw[dashed,-latex](-4,-2)--(-3,-1);
																			\draw[dashed,-](2,0)--(4,-2);
																			\draw[dashed,-latex](2,0)--(3,-1);
																			\draw[dashed,-](2,0)--(4,2);
																			\draw[dashed,-latex](2,0)--(3,1);
																			\draw[-latex](-2,0)--(0,0);
																			\coordinate (c) at (0.3,1);
																			\fill (c) circle (0pt) node[below] {$\Omega_{2}$};
																			\coordinate (d) at (0.3,0);
																			\fill (d) circle (0pt) node[below] {$\Sigma_5$};
																			\coordinate (c) at (0.3,-1);
																			\fill (c) circle (0pt) node[below] {$\Omega_{5}$};
																			\coordinate (d) at (4,1);
																			\fill (d) circle (0pt) node[below] {$\Omega_{1}$};
																			\coordinate (e) at (4,-1);
																			\fill (e) circle (0pt) node[below] {$\Omega_{6}$};
																			\coordinate (f) at (-4,1);
																			\fill (f) circle (0pt) node[below] {$\Omega_{3}$};
																			\coordinate (g) at (-4,-1);
																			\fill (g) circle (0pt) node[below] {$\Omega_{4}$};
																			\coordinate (h) at (3,1.8);
																			\fill (h) circle (0pt) node[below] {$\Sigma_1$};
																			\coordinate (i) at (-3,1.8);
																			\fill (i) circle (0pt) node[below] {$\Sigma_3$};
																			\coordinate (j) at (3,-1.8);
																			\fill (j) circle (0pt) node[above] {$\Sigma_6$};
																			\coordinate (k) at (-3,-1.8);
																			\fill (k) circle (0pt) node[above] {$\Sigma_4$};
																			\coordinate (l) at (0,1.5);
																			\fill (l) circle (1pt) node[left] {$ik_{n_1}$};
																			\coordinate (m) at (0,1);
																			\fill (m) circle (1pt) node[left] {$ik_{n_2}$};
																			\coordinate (l1) at (0,-1.5);
																			\fill (l1) circle (1pt) node[left] {$-ik_{n_1}$};
																			\coordinate (m1) at (0,-1);
																			\fill (m1) circle (1pt) node[left] {$-ik_{n_2}$};
																		\end{tikzpicture}
																		\caption{ The deformation of the Jump contour }
																		\label{figE1}
																	\end{figure}

																	Introduce some notations to be mentioned below
																	\begin{align}
																		&\Sigma_1=\left\{k=\kappa_0+le^{i\varphi},l\geq0\right\},\ \ \Sigma_3=\left\{k=-\kappa_0+le^{3i\varphi},l\geq0\right\},\nonumber\\
																		&\Sigma_4=\left\{k=-\kappa_0+le^{-3i\varphi},l\geq0\right\},\ \ \Sigma_6=\left\{k=\kappa_0+le^{-i\varphi},l\geq0\right\},\nonumber
																	\end{align}
then we have the following Proposition. 	
																	\begin{Proposition}\label{dbar}
																		There exist four functions $R_j:\mathbb{C} \to \Omega_j$, $j=1,3,4,6$ that satisfy the following boundary conditions:
																		\begin{align}
																			&	R_1(k)=\left\{ \begin{array}{ll}
																				\overline{\tilde{r}_p(\bar{k})},   &\text{as } k\in (\kappa_0,+\infty),\\[12pt]
																				\overline{\tilde{r}_p(\kappa_0)}	, &\text{as }k\in \Sigma_1,\\
																			\end{array}\right.\ \
																			R_3(k)=\left\{ \begin{array}{ll}
																				\overline{\tilde{r}_p(k)},   &\text{as } k\in (-\infty,-\kappa_0),\\[12pt]
																				\overline{\tilde{r}_p(-\kappa_0)}	, &\text{as }k\in \Sigma_3,\\
																			\end{array}\right.\nonumber\\
																			&	R_4(k)=\left\{ \begin{array}{ll}
																				\tilde{r}_p(k),   &\text{as } k\in (-\infty,-\kappa_0),\\[12pt]
																				\tilde{r}_p(-\kappa_0), &\text{as }k\in \Sigma_4,\\
																			\end{array}\right.\ \
																			R_6(k)=\left\{ \begin{array}{ll}
																				\tilde{r}_p(k),   &\text{as } k\in (\kappa_0,+\infty),\\[12pt]
																				\tilde{r}_p(\kappa_0), &\text{as }k\in \Sigma_6,\\
																			\end{array}\right. \nonumber
																		\end{align}
																		where $\tilde{r}_p(k)=r_p(k)(\prod\limits_{j}\frac{k-ik_n}{k+ik_n})^2$ and $R_j$ has the following estimates
																		
																		\begin{align}
																			&	|\bar{\partial}R_j(k)|\lesssim|\tilde{r}_p^{\prime}(|k|)|+|k-\kappa_0|^{-\frac{1}{2}}, \ \ k \in \Omega_j,j=1,3,4,6;\label{dbar1}\\
																			&	|\bar{\partial}R_j(k)|=0,\ \ j=2,5.\label{dbar2}
																		\end{align}
	\end{Proposition}
																		\begin{proof}
																			Let $k-\kappa_0=\rho e^{i\phi}$, thus $\bar{\partial}=\frac{e^{i\phi}}{2}\left(\partial_\rho+i\rho^{-1}\partial_\phi\right)$.
																			Take $R_1$ as an example, it can be constructed as follows to satisfy the boundary conditions
																			\begin{align}
																				R_1(k)=\left(	\overline{\tilde{r}_p(|\bar{k}|)}-\overline{\tilde{r}_p(\kappa_0)}\right)\cos(2\phi)+\overline{\tilde{r}_p(\kappa_0)}.
																			\end{align}
																		The estimates 	(\ref{dbar1})  can be derived by a  simple calculation.
																		\end{proof}
				Further we define
																	\begin{align}
																		R^{(2)}(k)=\left\{\begin{array}{ccc}
																			\left(\begin{array}{cc}
																				1 & 0\\
																				R_{j}(k)e^{2it\theta(k)} & 1
																			\end{array}\right), &k\in \Omega_{j},j=1,3,\\[10pt]
																			\left(\begin{array}{cc}
																				1 & 	R_{j}(k)e^{-2it\theta(k)} \\
																				0& 1
																			\end{array}\right),  &k\in \Omega_{j},j=4,6,\\[10pt]
																			I,&k\in \Omega_{2}\cup\Omega_{5}.
																		\end{array}\right.
																	\end{align}		
 and

															\begin{align}
																		J(k)=\left\{\begin{array}{lll}	\left(\begin{array}{cc}
																				1 & 0\\
																				\frac{c_{n,p}e^{2it\theta(ik_n)}}{k-ik_n} & 1
																			\end{array}\right)	\left(\begin{array}{cc}
																				1 & \frac{k-ik_n}{c_{n,p}e^{2it\theta(ik_n)}}\\
																				\frac{c_{n,p}e^{2it\theta(ik_n)}}{k-ik_n} & 0
																			\end{array}\right)	\left(\begin{array}{cc}
																				\prod\limits_{j}\frac{k-ik_n}{k+ik_n} & 0\\
																				0 & \prod\limits_{j}\frac{k+ik_n}{k-ik_n}
																			\end{array}\right),\ \ k\in \cup\mathcal{D}_n,\\[10pt]
																			\left(\begin{array}{cc}
																				1 & 	\frac{c_{n,p}e^{2it\theta(ik_n)}}{k-ik_n}\\
																				0 & 1
																			\end{array}\right)	\left(\begin{array}{cc}
																				1 &\frac{c_{n,p}e^{2it\theta(ik_n)}}{k+ik_n} \\
																				\frac{k+ik_n}{c_{n,p}e^{2it\theta(ik_n)}}	& 0
																			\end{array}\right)	\left(\begin{array}{cc}
																				\prod\limits_{j}\frac{k-ik_n}{k+ik_n} & 0\\
																				0 & \prod\limits_{j}\frac{k+ik_n}{k-ik_n}
																			\end{array}\right),\ \ k\in \cup\bar{\mathcal{D}}_n,\\[10pt]\\
																			\left(\begin{array}{cc}
																				\prod\limits_{j}\frac{k-ik_n}{k+ik_n} & 0\\
																				0 & \prod\limits_{j}\frac{k+ik_n}{k-ik_n}
																			\end{array}\right),\ \  elsewhere.\\[10pt]
																		\end{array}\right.\nonumber
																	\end{align}

Let $\Sigma^{(2)}=\Sigma_1\cup\Sigma_3\cup\Sigma_4\cup\Sigma_5\cup\Sigma_6$, thus the RHP on $\mathbb{R}$ converted to a mixed RHP on $\Sigma^{(2)}\cup(\cup\mathcal{D}_n)\cup(\cup\bar{\mathcal{D}}_n)$. Make  the following two  matrix transforms
\begin{align}																		
N^{(1)}(k)=N^{J}(k)J(k),\ \ N^{(2)}(k)=N^{(1)}(k)R^{(2)}(k),\label{trans3}
\end{align}
then $N^{(2)}(k)$ satisfies the mixed $\bar\partial$-RH problem  																
							
							\begin{RHP}
								Find a matrix valued function  $ M^{(2)}(k)$ with following properties:
								
								$\blacktriangleright$ Analyticity:  $M^{(2)}(k)$ is continuous in $\mathbb{C}\backslash(\Sigma^{(2)}\cup(\cup\mathcal{D}_n)\cup(\cup\bar{\mathcal{D}}_n))$;
								
								$\blacktriangleright$ Jump condition: $M^{(2)}(k)$ satisfies
								\begin{equation}
									M^{(2)}_{+}(k)=M^{(2)}_{-}(k)V^{(2)}(k),\hspace{0.5cm}k \in \Sigma^{(2)}\cup(\cup\mathcal{D}_n)\cup(\cup\bar{\mathcal{D}}_n),
								\end{equation}
								where 			
\begin{align}
																		V^{(2)}(k)=\left\{\begin{array}{lll}	\left(\begin{array}{cc}
																					1 &0 \\
																					\overline{\tilde{r}_p(\kappa_0)}e^{2it\theta(k)}& 1
																				\end{array}\right)^{-1},\ \ k\in \Sigma_1,\\[10pt]
																				\left(\begin{array}{cc}
																					1 & 0\\
																					\overline{\tilde{r}_p(-\kappa_0)}e^{2it\theta(k)} & 1
																				\end{array}\right)^{-1},\ \ k\in \Sigma_3,\\[10pt]
																			\left(\begin{array}{cc}
																					1-|\tilde{r}_p(k)|^2	 & \tilde{r}_p(k)e^{-2it\theta}\\
																					-\overline{\tilde{r}_p(k)}e^{2it\theta}&  1
																				\end{array}\right),\ \  k\in [-k_0, k_0],\\[10pt]\\
 \left(\begin{array}{cc}
																					1 & \tilde{r}_p(-\kappa_0)e^{-2it\theta(k)}\\
																					0 & 1
																				\end{array}\right),\ \ k\in \Sigma_4, \\[10pt]
 \left(\begin{array}{cc}
																					1 & 	\tilde{r}_p(\kappa_0)e^{-2it\theta(k)} \\
																					0& 1
																				\end{array}\right), \ \ k\in \Sigma_6,
																		\end{array}\right.\nonumber
																	\end{align}
see  Figure \ref{figN2}.

								$\blacktriangleright$ $\bar{\partial}$-Derivative: For $k\in\mathbb{C}$
								we have
								\begin{align}
									\bar{\partial}M^{(2)}(k)=M^{(2)}(k)\bar{\partial}R^{(2)}(k),
								\end{align}
								where
								\begin{equation}
									\bar{\partial}R^{(2)}(k,\xi)=\left\{\begin{array}{lll}
										\left(\begin{array}{cc}
											0 & 0\\
											\bar{\partial}R_j(k,\xi)e^{2it\theta} & 0
										\end{array}\right), &  k\in \Omega_j,j=1,3;\\
										\\
										\left(\begin{array}{cc}
											0 & \bar{\partial}R_j(k,\xi)e^{-2it\theta}\\
											0 & 0
										\end{array}\right),  & k\in \Omega_j,j=4,6;\\
										\\
										0,  &elsewhere;\\
									\end{array}\right.\label{DBARR1}
								\end{equation}

							\end{RHP}

																	\begin{figure}
																		\centering
																		\begin{tikzpicture}[scale=0.8]
																			\draw[blue,-](-2,0)--(2,0);
																			\draw[dashed,->](0,-3)--(0,3);
																			\draw[dashed,-](-6,0)--(-2,0);
																			\draw[dashed,-](2,0)--(6,0);
																			\coordinate (a) at (2,0);
																			\fill (a) circle (1.5pt) node[below] {$\kappa_0$};
																			\coordinate (b) at (-2,0);
																			\fill (b) circle (1.5pt) node[below] {$-\kappa_0$};
																			\draw[blue,-](-2,0)--(-4,2);
																			\draw[blue,-latex](-4,2)--(-3,1);
																			\draw[blue,-](-2,0)--(-4,-2);
																			\draw[blue,-latex](-4,-2)--(-3,-1);
																			\draw[blue,-](2,0)--(4,-2);
																			\draw[blue,-latex](2,0)--(3,-1);
																			\draw[blue,-](2,0)--(4,2);
																			\draw[blue,-latex](2,0)--(3,1);
																			\draw[blue,-latex](-2,0)--(0,0);
																			\coordinate (c) at (0,2);
																			\fill (c) circle (0pt) node[below] {$\Omega_{2}$};
																			\coordinate (d) at (0,0);
																			\fill (d) circle (0pt) node[below] {$\Sigma_5$};
																			\coordinate (c) at (0,-2);
																			\fill (c) circle (0pt) node[above] {$\Omega_{5}$};
																			\coordinate (d) at (4,1);
																			\fill (d) circle (0pt) node[below] {$\Omega_{1}$};
																			\coordinate (e) at (4,-1);
																			\fill (e) circle (0pt) node[below] {$\Omega_{6}$};
																			\coordinate (f) at (-4,1);
																			\fill (f) circle (0pt) node[below] {$\Omega_{3}$};
																			\coordinate (g) at (-4,-1);
																			\fill (g) circle (0pt) node[below] {$\Omega_{4}$};
																			\coordinate (h) at (3,1.8);
																			\fill (h) circle (0pt) node[below] {$\Sigma_1$};
																			\coordinate (i) at (-3,1.8);
																			\fill (i) circle (0pt) node[below] {$\Sigma_3$};
																			\coordinate (j) at (3,-1.8);
																			\fill (j) circle (0pt) node[above] {$\Sigma_6$};
																			\coordinate (k) at (-3,-1.8);
																			\fill (k) circle (0pt) node[above] {$\Sigma_4$};
																			\coordinate (l) at (4,2);
																			\fill (l) circle (0pt) node[right] {};
																			\coordinate (l) at (-4,2);
																			\fill (l) circle (0pt) node[left] {};
																			\coordinate (m) at (-4,-2);
																			\fill (m) circle (0pt) node[left] {};
																			\coordinate (n) at (4,-2);
																			\fill (n) circle (0pt) node[right] {};
																			\coordinate (o) at (0,0);
																			\fill (o) circle (0pt) node[above] {};
																			\coordinate (l) at (0,2.2);
																			\draw[blue] (0,2.2) circle (0.2);
																			\fill (l) circle (1pt) node[left] {$ik_{n_1}$};
																			\coordinate (m) at (0,1.3);
																			\draw[blue] (0,1.3) circle (0.2);
																			\fill (m) circle (1pt) node[left] {$ik_{n_2}$};
																			\coordinate (l1) at (0,-2.2);
																			\draw[blue] (0,-2.2) circle (0.2);
																			\fill (l1) circle (1pt) node[left] {$-ik_{n_1}$};
																			\coordinate (m1) at (0,-1.3);
																			\draw[blue] (0,-1.3) circle (0.2);
																			\fill (m1) circle (1pt) node[left] {$-ik_{n_2}$};
																		\end{tikzpicture}
																		\caption{ The jump contour  of $M^{(2)}(k)$ }
																		\label{figN2}
																	\end{figure}
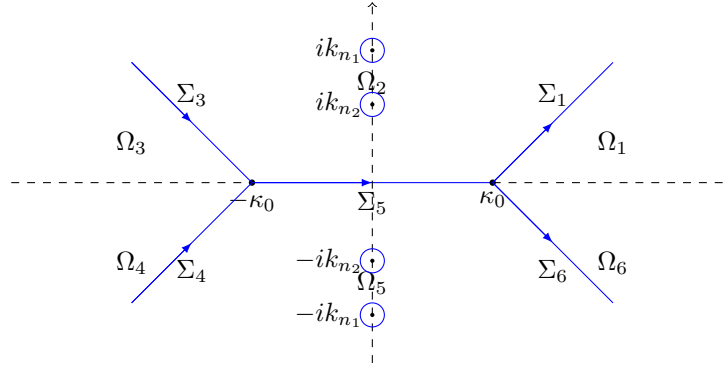
The main contribution to the solution of the hybrid RH problem comes from the $\bar{\partial}$ part and the jump path. For this reason, we have decomposed the hybrid RH problem  as follows:
																	\begin{align}
																		N^{(2)}(k)=N^{(3)}(k)N^{(R)}(k), \label{trans4}
																	\end{align}					
																	where $ N^{(R)}(k)$ is the solution of a pure RH problem, $N^{(3)}(k)$  is the solution of a pure $\bar{\partial}$ problem.
										
																	\subsection{Analysis on pure RH problem}
																	
Noting that  the jump matrices on the circles $D_n$ or $\bar D_n$  decay  exponentially to the
  identity matrix as $t \to  \infty$,  it follows that
																	\begin{align}
																		N_R(k)=\tilde{N}_R(k)+\mathcal{O}(e^{-2\rho_0t}),\label{trans5}
																	\end{align}	
where   $\tilde{N}_R(k)$  satisfies the following RH problem			
																	\begin{RHP}\label{rhp14}
																		Find a matrix-valued function $\tilde{N}_R(k)$ with following properties:
																		
																		$\blacktriangleright$ Analyticity: $\tilde{N}_R(k)$is analytical  in $\mathbb{C}\setminus  \Sigma^{(2)} $;
																		
																		$\blacktriangleright$ Asymptotic behaviors:
																		\begin{align}
																			&\tilde{N}_R(k) \to I,\hspace{0.5cm}|k| \rightarrow \infty;
																		\end{align}

																		$\blacktriangleright$ Jump condition: $\tilde{N}_R(k)$   satisfies the jump relation
																		$$\tilde{N}_{R,+}(k)=\tilde{N}_{R,-}(k)V^{(2)}(k).$$
																	\end{RHP}				
	We make some approximations to $\tilde{N}_R(k)$  to match the solvable Painlev\'e RH model.
By   the Taylor expansion of $\theta(k)$,   we  rewrite
 $$t\theta(k)=\frac{4}{3}\hat{k}^3+\tilde{s}\hat{k}+\mathcal{O}(\hat{k}^{\alpha}t^{-\beta}),$$
    where   $\hat{k}=(6t)^{\frac{1}{3}}k$, $\tilde{s}=6^{-\frac{1}{3}}(\xi-2)t^{\frac{2}{3}}$.
																	
																	\begin{Proposition}\label{bijin}
																		The elements in the jump matrix can be approximated in the following way
																		\begin{align}
																			&	|\overline{\tilde{r}_p(\hat{k}(6t)^{-\frac{1}{3}})}e^{2it\theta(\hat{k}))}-\overline{r_p(0)}e^{i(\frac{8}{3}\hat{k}^3+2\tilde{s}\hat{k})}|\leq ct^{-\frac{1}{6}},& \hspace{0.1cm} \hat{k}\in \hat{\Sigma}_5,\label{bijin1}\\
																			&	|\overline{\tilde{r}_p(\pm\hat{\kappa}_0)}e^{2it\theta(\hat{k}))}-\overline{r_p(0)}e^{i(\frac{8}{3}\hat{k}^3+2\tilde{s}\hat{k})}|\leq ct^{-\frac{1}{6}},&  \hat{k}\in \hat{\Sigma}_j,\label{bijin2}
																		\end{align}
																		where j=1,6 take "+" in (\ref{bijin1}), j=3,4 take "-" in (\ref{bijin2}).
																		\begin{proof}
																			Consider $S(\hat{k})=|t\theta(k)-\frac{4}{3}\hat{k}^3+\tilde{s}\hat{k}|$. Doing a Taylor expansion on $\theta(k)$ at $k=0$ and a series summation on $S(\hat{k})$, a simple estimate gives $S(\hat{k})\lesssim \mathcal{O}(t^{-\frac{2}{3}})$. Let $\nu=|2iS(\hat{k})|$, then $\nu \to 0$
																			as $t \to \infty$. For a fixed constant $\gamma$, $|\nu|\leq \gamma$ when $t$ is large enough. Notice
																			\begin{align}
																				|e^{2it\theta(\hat{k})}-e^{2it(\frac{4}{3}\hat{k}^3+\tilde{s}\hat{k})}|=|e^{2iS(\hat{k})}-1|\leq e^{|2iS(\hat{k})|}-1\leq \frac{e^{\gamma}-1}{\gamma}|\nu|\leq ct^{-\frac{2}{3}}.
																			\end{align}
																			When  $\hat{k} \in \hat{\Sigma}_5$, $ |e^{2it\theta(\hat{k})}|=1$, thus
																			\begin{align}
																				 &|\overline{\tilde{r}_p(\hat{k}(6t)^{-\frac{1}{3}})}e^{2it\theta(\hat{k}))}-\overline{r_p(0)}e^{i(\frac{8}{3}\hat{k}^3+2\tilde{s}\hat{k})}|\nonumber\\
																				&\leq |\overline{\tilde{r}_p(\hat{k}(6t)^{-\frac{1}{3}})}-\overline{r_p(0)}|+|\overline{r_p(0)}||e^{2it\theta(\hat{k})}-e^{2it(\frac{4}{3}\hat{k}^3+\tilde{s}\hat{k})}|\nonumber\\
																				&\leq ||\overline{\tilde{r}_p^{\prime}}||_{L^2}\left(\hat{k}(6t)^{-\frac{1}{3}}\right)^{\frac{1}{2}}+ct^{-\frac{2}{3}}
																				\leq ct^{-\frac{1}{6}}.
																			\end{align}
Let $\hat{k}=\hat{\kappa}_0+u+iv$, then $Re\left(i(\frac{8}{3}\hat{k}^3+2\tilde{s}\hat{k})\right)\leq -\frac{16}{3}u^2v\leq 0$, so $|e^{i(\frac{8}{3}\hat{k}^3+2\tilde{s}\hat{k})}|$ is bounded. When $\hat{k} \in \hat{\Sigma}_1$,
\begin{align}
&|\overline{\tilde{r}_p(\hat{\kappa}_0)}e^{2it\theta(\hat{k}))}-\overline{r_p(0)}e^{i(\frac{8}{3}\hat{k}^3+2\tilde{s}\hat{k})}|\nonumber\\
																				 &\leq|\overline{\tilde{r}_p(\hat{\kappa}_0)}-\overline{r_p(0)}||e^{i(\frac{8}{3}\hat{k}^3+2\tilde{s}\hat{k})}|+|\overline{\tilde{r}_p(\hat{\kappa}_0)}||e^{2it\theta(\hat{k})}-e^{2it(\frac{4}{3}\hat{k}^3+\tilde{s}\hat{k})}|\nonumber\\
&\leq c||\overline{\tilde{r}^{\prime}_p}||_{L^2}|\hat{\kappa}_0|^{\frac{1}{2}}+ct^{-\frac{2}{3}}\leq ct^{-\frac{1}{6}}.
\end{align}
The other results in (\ref{bijin2}) can be proved in similar way.
																		\end{proof}
																	\end{Proposition}				
																	
																	By Proposition \ref{bijin},  as $t \to \infty$, we claim that the solution of RHP \ref{rhp14} can be approximated by the solution of the limit model
																	\begin{align}
																		\tilde{N}_R(k)=\hat{N}_R(\hat{k})+\mathcal{O}(t^{-\frac{1}{6}}), \label{trans6}
																	\end{align}
																	where $\hat{N}_R(\hat{k})$ is the solution of  the following limit model RH problem	
																	\begin{RHP}\label{rhp15}
																		Find a matrix-valued function $\hat{N}_R(\hat{k})$ with following properties:
																		
																		$\blacktriangleright$ Analyticity: $\hat{N}_R(\hat{k})$is analytical  in $\mathbb{C}\setminus  \Sigma^{(2)} $;
																		
																		$\blacktriangleright$ Asymptotic behaviors:
																		\begin{align}
																			&\hat{N}_R(\hat{k}) \to I,\hspace{0.5cm}|\hat{k}| \rightarrow \infty;\nonumber
																		\end{align}

																		$\blacktriangleright$ Jump condition: $\hat{N}_R(\hat{k})$ has continuous boundary values $\hat{N}_{R,\pm}(\hat{k})$ on $\Sigma^{(2)}$ satisfying
																		$$\hat{N}_{R,+}(\hat{k})=\hat{N}_{R,-}(\hat{k})\widehat{V}^{(2)}(\hat{k}),$$
																		where
																		\begin{align}
																			\widehat{V}^{(2)}(\hat{k})=\left\{\begin{array}{lll}
																				&	\left(\begin{array}{cc}
																					1 & 0 \\
																					r_p(0)e^{\frac{8}{3} i \hat{k}^3+2 i \tilde{s}\hat{k}}  & 1
																				\end{array}\right)^{-1}:=b_{+}^{-1}, \quad \hat{k} \in \hat{\Sigma}_1 \cup \hat{\Sigma}_3,\\
																				&\left(\begin{array}{cc}
																					1 & 	r_p(0)e^{-\frac{8}{3} i \hat{k}^3-2 i \tilde{s}\hat{k}}  \\
																					0 & 1
																				\end{array}\right):=b_{-},\quad \hat{k} \in \hat{\Sigma}_4 \cup \hat{\Sigma}_6,\\
																				&	\left(\begin{array}{cc}
																					1 &  r_p(0)e^{-\frac{8}{3} i \hat{k}^3-2 i \tilde{s}\hat{k}} \\
																					0 & 1
																				\end{array}\right)\left(\begin{array}{cc}
																					1 & 0 \\
																					r_p(0) e^{\frac{8}{3} i \hat{k}^3+2 i \tilde{s}\hat{k}} & 1
																				\end{array}\right)^{-1}:=b_{-} b_{+}^{-1}, \quad \hat{k} \in \hat{\Sigma}_5,\\
																			\end{array}\right. \nonumber
																		\end{align}		
																	\end{RHP}				
Let $r_p(0)=|r_p(0)|e^{2i\phi}	$, where   $2\phi$ is the auxiliary angle of $r_p(0)$.
Introduce a matrix function
																	\begin{align}
																		P(\hat{k})=\left\{\begin{array}{cc}
																			b_{+}, & \hat{k} \in \Omega_2 \cup \Omega_4, \\
																			b_{-}^{-1}, & \hat{k} \in \Omega_6 \cup \Omega_8, \\
																			I, & \hat{k} \in \Omega_1 \cup \Omega_3 \cup \Omega_5 \cup \Omega_7,
																		\end{array}\right.
																	\end{align}		
																	
																	\begin{figure}
																		\centering
																		\begin{tikzpicture}[scale=0.8]
																			\draw[dashed,-](-1,0)--(1,0);
																			\coordinate (a) at (2,0);
																			\fill (a) circle (1.5pt) node[below] {$\kappa_0$};
																			\coordinate (b) at (-2,0);
																			\fill (b) circle (1.5pt) node[below] {$-\kappa_0$};
																			\draw[dashed,-](-2,0)--(-4,2);
																			\draw[dashed,-latex](-4,2)--(-3,1);
																			\draw[dashed,-](-2,0)--(-4,-2);
																			\draw[dashed,-latex](-4,-2)--(-3,-1);
																			\draw[dashed,-](2,0)--(4,-2);
																			\draw[dashed,-latex](2,0)--(3,-1);
																			\draw[dashed,-](2,0)--(4,2);
																			\draw[dashed,-latex](2,0)--(3,1);
																			\draw[dashed,-latex](-2,0)--(-1,0);
																			\draw[dashed,-](-1,0)--(1,0);
																			\draw[dashed,-latex](0,0)--(1,0);
																			\draw[dashed,-](1,0)--(2,0);
																			\draw[-latex](-2,2)--(-1,1);
																			\draw[-](-1,1)--(0,0);
																			\draw[-latex](0,0)--(1,1);
																			\draw[-](0,0)--(2,2);
																			\draw[-latex](-2,-2)--(-1,-1);
																			\draw[-](-1,-1)--(0,0);
																			\draw[-latex](0,0)--(1,-1);
																			\draw[-](2,-2)--(0,0);
																			\coordinate (c) at (0,1.5);
																			\fill (c) circle (0pt) node[below] {$\Omega_{3}$};
																			\coordinate (d) at (0,-1.5);
																			\fill (d) circle (0pt) node[above] {$\Omega_{7}$};
																			\coordinate (d) at (3,0);
																			\fill (d) circle (0pt) node[left] {$\Omega_{1}$};
																			\coordinate (d1) at (-3,0);
																			\fill (d1) circle (0pt) node[right] {$\Omega_{5}$};
																			\coordinate (e) at (4,2);
																			\fill (e) circle (0pt) node[right] {$\hat{\Sigma}_{1}$};
																			\coordinate (e1) at (-4,2);
																			\fill (e1) circle (0pt) node[left] {$\hat{\Sigma}_{3}$};
																			\coordinate (e2) at (-4,-2);
																			\fill (e2) circle (0pt) node[left] {$\hat{\Sigma}_{4}$};
																			\coordinate (e3) at (4,-2);
																			\fill (e3) circle (0pt) node[right] {$\hat{\Sigma}_{6}$};
																			\coordinate (f) at (2,2);
																			\fill (f) circle (0pt) node[right] {$\Sigma^*_{1}$};
																			\coordinate (f1) at (-2,2);
																			\fill (f1) circle (0pt) node[left] {$\Sigma^*_{3}$};
																			\coordinate (f2) at (-2,-2);
																			\fill (f2) circle (0pt) node[left] {$\Sigma^*_{4}$};
																			\coordinate (f3) at (2,-2);
																			\fill (f3) circle (0pt) node[right] {$\Sigma^*_{6}$};
																			\coordinate (l) at (0.4,0.2);
																			\fill (l) circle (0pt) node[right] {$\hat{\Sigma}_{5}$};
																			\coordinate (l1) at (-0.4,0.2);
																			\fill (l1) circle (0pt) node[left] {$\hat{\Sigma}_{5}$};
																			\coordinate (g) at (2,1);
																			\fill (g) circle (0pt) node[left] {$\Omega_{2}$};
																			\coordinate (g1) at (2,-1);
																			\fill (g1) circle (0pt) node[left] {$\Omega_{8}$};
																			\coordinate (g2) at (-2,-1);
																			\fill (g2) circle (0pt) node[right] {$\Omega_{6}$};
																			\coordinate (g3) at (-2,1);
																			\fill (g3) circle (0pt) node[right] {$\Omega_{4}$};
																		\end{tikzpicture}
																		\caption{ Jump path deformation to a standard painlesv\'e jump }
																		\label{figpainleve}
																	\end{figure}
Making a transformation
\begin{align}
N^P(\hat{k})=e^{i(\phi-\frac{\pi}{4})}\hat{N}_R(\hat{k})P(\hat{k})e^{-i(\phi-\frac{\pi}{4})}, \label{trans7}
\end{align}																		
then $N^P(\hat{k})$  satisfies the following model Painlev\'e  RH problem (  Figure \ref{figpainleve})

\begin{RHP}\label{painleve rhp}
Find a matrix-valued function $N^P(\hat{k})$ with following properties:

$\blacktriangleright$ Analyticity: $N^P(\hat{k})$is analytical  in $\mathbb{C}\setminus  \Sigma^{*} $, \ \  $\Sigma^{*}=\Sigma^*_{1}\cup\Sigma^*_{3}\cup\Sigma^*_{4}\cup\Sigma^*_{6}$;
																		
																		$\blacktriangleright$ Asymptotic behaviors:
																		\begin{align}
																			&N^P(\hat{k}) \to I,\hspace{0.5cm}|\hat{k}| \rightarrow \infty;\nonumber
																		\end{align}

																		$\blacktriangleright$ Jump condition: $N^P(\hat{k})$ has continuous boundary values $N^P_{\pm}(\hat{k})$ on $\Sigma^{*}$ satisfying
																		$$N^P_+(\hat{k})=N^P_-(\hat{k})\widehat{V}^{P}(\hat{k}),$$
																		where the jump matrix $\widehat{V}^{P}(\hat{k})=e^{-i(\frac{4}{3} \hat{k}^3+ \tilde{s}\hat{k})\hat{\sigma}_3}V_0^{P}(\hat{k})$,
																		\begin{align}
																			V_0^{P}(\hat{k})=\left\{\begin{array}{lll}
																				&	\left(\begin{array}{cc}
																					1 & 0 \\
																					i|r_p(0)| & 1
																				\end{array}\right), \quad \hat{k} =\zeta e^{\frac{i\pi}{4}},\zeta>0,\\
																				&\left(\begin{array}{cc}
																					1 & 		-i|r_p(0)|   \\
																					0 & 1
																				\end{array}\right),\quad \hat{k} =\zeta e^{-\frac{i\pi}{4}},\zeta>0,\\
																				&	\left(\begin{array}{cc}
																					1 & 		0  \\
																					-i|r_p(0)|  & 1
																				\end{array}\right),\quad \hat{k} =\zeta e^{\frac{3i\pi}{4}},\zeta>0,\\
																				&\left(\begin{array}{cc}
																					1 & 	i|r_p(0)|   \\
																					0 & 1
																				\end{array}\right),\quad \hat{k} =\zeta e^{-\frac{3i\pi}{4}},\zeta>0.\\
																			\end{array}\right.
																		\end{align}		
																	\end{RHP}				
The solution of RHP \ref{painleve rhp} 	corresponds to the real-valued solution of Painlev\'e II equation, which has the following expansion
																	$$
																	N^P(\hat{k} ;\tilde{ s})=\frac{N_1^P(\tilde{ s})}{\hat{k}}+\frac{N_2^P(\tilde{ s})}{\hat{k}^2}+...
																	$$
																	When $\hat{k} \to \infty$, $N_1^P(\tilde{ s})$, $N_2^P(\tilde{ s})$ can be represented by the solution of Painlev\'e II equation
																	\begin{align}
																		&	N_1^P(\tilde{s})=\frac{1}{2}\left(\begin{array}{cc}
																			i \int^s v^2(\xi) d \xi & v(\tilde{s}) \\
																			v(\tilde{s}) & -i \int^{\tilde{s}} v^2(\xi) d \xi
																		\end{array}\right),\nonumber\\
																		&	N_2^P(\tilde{ s})=\left(\begin{array}{cc}
																			\xi & \eta \\
																			-\eta & \xi
																		\end{array}\right), \nonumber
																	\end{align}			
																	where $v$ is the solution of Painlev\'e II equation: $v^{\prime \prime}(\tilde{s})=\tilde{s}v(s\tilde{s})+2 v^3(\tilde{s})$, and			
																	\begin{equation}
																		\xi(\tilde{s})=-\frac{1}{8}\left(\left(\int^{\tilde{s}} v^2(\xi) d \xi\right)^2-v^2\right),\nonumber
																	\end{equation}
																	\begin{equation}
																		\eta(\tilde{s})=\frac{i}{4}\left(-v(\tilde{s}) \int^{\tilde{s}} v^2(\xi) d \xi+v^2\right). \nonumber
																	\end{equation}
																	
\subsection{Estimate on the $\bar{\partial}$-part}				

We only give the main results,  which can shown in a  similar way  to Section \ref{subsec3.3}.
Define
			\begin{align}
																		N^{(3)}(\hat{k})=N^{(2)}(\hat{k})(N^{(R)}(\hat{k}))^{-1},
																	\end{align}
																	then $N^{(3)}(\hat{k})$ satisfies a pure $\bar{\partial}$ problem
			 and has the following expansion
																	\begin{equation}
																		N^{(3)}(\hat{k})=I+\frac{N_1^{(3)}}{\hat{k}}+\mathcal{O}\left(\hat{k}^{-2}\right), \label{trans8}
																	\end{equation}
																where $N_1^{(3)}$ has following estimation
																		\begin{equation}
																			|N_1^{(3)}|\leq ct^{-\frac{1}{3}}.
																		\end{equation}

																	\subsection{The Painlev\'e asymtotics}

																	\begin{theorem}  Under the same condition with Theorem \ref{last}, the $u(x,t)$ of the Cauchy problem
(\ref{ch})-(\ref{initial}) admits   the following Painlev\'e asymptotics  in different regions\\
				\noindent  1. For $0<|\xi-2|t^{\frac{2}{3}}<C$,
				\begin{align}
	u(y,t)& =-\left(\frac{4}{3}\right)^{2 / 3} \frac{1}{t^{2 / 3}}\left(v^2(\tilde{s})-v_{\tilde{s}}(\tilde{s})\right)+\mathrm{O}\left(t^{-\frac{1}{2}}\right),\label{up}
																	\end{align}
		where $v(s)$ is the real-valued  solution of the Painlev\'e II equation , as $s \to \infty$,
						\begin{equation}
				v(s) \sim-r_p(0) \operatorname{Ai}(s) \sim-r_p(0) \frac{1}{2 \sqrt{\pi}} s^{-\frac{1}{4}} \mathrm{e}^{-\frac{2}{3}s^{3 / 2}}. \nonumber
		\end{equation}
	\noindent  2. For $0<|\xi+\frac{1}{4}|t^{\frac{2}{3}}<C$,
\begin{align}
u(y, t) =\frac{12^{1 / 6}}{t^{1 / 3}} v_1\left(s_1\right) \sin \psi\left(s_1, t\right)+\mathrm{O}\left(t^{-1/ 2}\right),\label{pu}
																	\end{align}
																	where $v_1(s)$ is the real-valued, nonsingular solution of the Painlev\'e II equation, as $s \to \infty$,
	\begin{align}
&v_1(s) \sim\left|\tilde{r}_p\left(\frac{\sqrt{3}}{2}\right)\right| \frac{1}{2 \sqrt{\pi}} s^{-\frac{1}{4}} \mathrm{e}^{-\frac{2}{3} s^{3 / 2}}, \nonumber\\
&s_1=-\left(\frac{16}{3}\right)^{1 / 3}\left(\xi+\frac{1}{4}\right) t^{2 / 3}, \ \
																	\psi\left(s_1, t\right)=-\frac{3 \sqrt{3}}{4} t-\frac{3^{5 / 6}}{2^{4 / 3}} s_1 t^{1 / 3}+\Delta,\nonumber
	\end{align}
with
$$
																	\begin{aligned}
																		\Delta= & -\frac{4 \sqrt{3}}{\pi} \int_0^{\infty} \frac{\log \left(1-|\tilde{r}_p(\xi)|^2\right)}{1+4 \xi^2} \mathrm{~d} \xi+\arg \tilde{r}_p\left(\frac{\sqrt{3}}{2}\right) \\
																		& +\frac{1}{\pi} \int_{-\infty}^{\infty} \frac{\log \left(1-|\tilde{r}_p/(\xi)|^2\right)}{\xi-\frac{\sqrt{3}}{2}} \mathrm{~d} \xi-4 \sum_{n=1}^N \arg \left(\frac{\sqrt{3}}{2}+\mathrm{i} k_n\right) \\
																		& -2 \sqrt{3} \sum_{n=1}^N \log \left(\frac{1+2 k_n}{1-2 k_n}\right).
																	\end{aligned}
$$
																		
																	\end{theorem}
\begin{proof}
Recalling a  series of transformations (\ref{trans1}), (\ref{trans2}), (\ref{trans3}), (\ref{trans4}), (\ref{trans5}), (\ref{trans6}), (\ref{trans7}) and  (\ref{trans7}),
the reconstruction formula  (\ref{newu})  leads to  the asymptotic expansion  (\ref{up}).
  The asymptotic expansion (\ref{pu}) can be obtained in a similar way.

\end{proof}

\vspace{2mm}

\noindent\textbf{Acknowledgements}

This work is supported by  the National Science
Foundation of China (Grant No. 12271104,  51879045).\vspace{2mm}

    \noindent\textbf{Data Availability Statements}

    The data which supports the findings of this study is available within the article.\vspace{2mm}

    \noindent{\bf Conflict of Interest}

    The authors have no conflicts to disclose.
																	\hspace*{\parindent}
																	\\
																	
																\end{document}